\documentclass[8pt,a4paper,twoside]{article}
\usepackage{a4wide, amssymb, amsmath, amsthm, graphics, comment, xspace, enumerate}
\usepackage{graphicx,multirow}
\usepackage[a4paper,colorlinks=true,citecolor=blue,urlcolor=blue,linkcolor=blue
                   , bookmarksopen=true,unicode=true,pdffitwindow=true]{hyperref}
\usepackage[section]{algorithm}
\usepackage[noend]{algpseudocode}
\usepackage{caption}
\usepackage{tabularx}
\usepackage{enumitem, multirow}
\usepackage[usenames,dvipsnames]{color}
\usepackage{alltt}
\usepackage{color}
\usepackage{listings}
\usepackage{cite}

\usepackage[subtle]{savetrees}

\usepackage{array, longtable}

\definecolor{string}{rgb}{0.7,0.0,0.0}
\definecolor{comment}{rgb}{0.13,0.54,0.13}
\definecolor{keyword}{rgb}{0.0,0.0,1.0}
\usepackage{amsmath,amsfonts,amssymb,subfigure}
\usepackage{tikz}
\usetikzlibrary{arrows}
\usetikzlibrary{decorations}
\usetikzlibrary{patterns}

\subfiguretopcaptrue
\tikzstyle{vtx}=[circle, inner sep= 0pt, minimum size= 1.2mm, fill]

\hypersetup{pdftitle={Complete characterization of the minimal-ABC trees}}

\newtheorem{te}{Theorem}[section]
\newtheorem{pro}[te]{Proposition}

\newtheorem{co}[te]{Corollary}
\newtheorem{lemma}[te]{Lemma}

\newtheorem{conjecture}{Conjecture}[section]

\newtheorem{re}{Remark}[section]

\newtheorem{observation}{Observation}[section]

\newcommand{\beq}{\begin{eqnarray}}
\newcommand{\eeq}{\end{eqnarray}}

\newcommand{\beqs}{\begin{eqnarray*}}
\newcommand{\eeqs}{\end{eqnarray*}}

\newcommand{\ABC}{{\rm ABC}}

\allowdisplaybreaks

\begin{document}
\title{ Complete characterization of the minimal-ABC trees}
\maketitle
{
\begin{center}
{ \bf Darko Dimitrov$^{a,b}$, Zhibin Du$^{c}$}
\end{center}

\baselineskip=0.20in
\begin{center}
{\it $^a$Hochschule f\"ur Technik und Wirtschaft Berlin, Germany \\
$^b$Faculty of Information Studies,  Novo mesto, Slovenia \\
E-mail: {\tt darko.dimitrov11@gmail.com}}
\end{center}
\baselineskip=0.20in
\begin{center}
{\it $^c$School of Software, South China Normal University, Guangdong 528225, China}
\\E-mail: {\tt zhibindu@126.com}
\end{center}
}


%
%
%
%

%
%

%
%

%
%
\vspace{6mm}
\begin{abstract}
The problem of characterizing trees with minimal atom-bond-connectivity index (minimal-ABC trees)
has a reputation as one of the most demanding recent open optimization problems in mathematical chemistry.
Here firstly, we give an affirmative answer to the conjecture, which states  that enough large minimal-ABC trees are comprised solely of
a root vertex and so-called $D_z$- and $D_{z+1}$-branches.
Based on the presented theoretical results here and some already known results,
we obtain enough constraints to reduce the search space and solve the optimization problem, and thus,
determine exactly the minimal-ABC trees of a given arbitrary order.

\bigskip


\smallskip


\noindent

\end{abstract}

\section[Introduction]{Introduction}
The {\em atom-bond connectivity} ({\em ABC})  index is one of the most investigated molecular structure descriptors,
which was proved to be applicable for practical purposes.
It was introduced by Estrada et al.  \cite{etrg-abc-98}  in $1998$, who showed  that there is an
excellent (linear) correlation between the ABC index  and the experimental heats of formation of alkanes.
Ten years later, an inventive explanation of this topological index  based on quantum theory was provided in \cite{e-abceba-08}.
The applicative potentiality of the ABC index was further supported in \cite{gtrm-abcica-12}  by Gutman et al.
As a result of those seminal works, the ABC index attracted the attention of mathematicians, chemists, and computer scientists,
which resulted in many  theoretical and  computational  results
\cite{adgh-dctmabci-14, cg-eabcig-11, cg-abccbg-12, clg-subabcig-12,  cll-abcbsp-13, dd-mabctb1b2-2018, dgf-abci-11, dmga-cbabcig-16, d-sptmabci-2-2015, dt-cbfgaiabci-10, d-etrabci-2017, ddf-sptmabci-3-2015, ddf-mabctb1b-2018, dis-rmabciggp-17, d-abcircg-2015,df-ftmabc-2015,  e-abcm-17, e-abceba-08, etrg-abc-98, ftvzag-siabcigo-2011, fgv-abcit-09, fgiv-cstmabci-12, ghl-srabcig-11, gs-sabcitnpv-16, gf-tsabci-12, gfahsz-abcic-2013, gg-nwabci-10, gtrm-abcica-12, hdw-abcig-19, k-abcibsfc-12, ks-ca-98, lcmzczj-twmabciatgnl-16, rb-egevvdbti-19, cr-psevvdbtiat-19, xzd-abcicg-2011, xzd-frabcit-2010, yxc-abcbsp-11, zlcl-bgabci-20, zc-rbabcfgai-2015, hma-esabcmt-2021}.

Among remaining open problems, the problem of full characterization of trees,
whose  ABC index is minimal (referred further  as {\em minimal-ABC trees}),
has a reputation as one of the most demanding recent open problems in mathematical chemistry.

The ABC index of a simple undirected graph $G=(V, E)$ with vertex set $V = V(G)$ and edge set $E= E(G)$ is defined as
\beq \label{eqn:001}
\ABC(G)=\sum_{uv\in E}\sqrt{\frac{d(u) +d(v)-2}{d(u)d(v)}}, \nonumber
\eeq
where $d(v)$ is the degree of a vertex $v \in V$.

A {\it pendant vertex} is a vertex of degree one.
If the degree of a vertex is larger than two and no vertex of degree two is adjacent to it, then the vertex is called a {\it big} vertex.
A path, whose one end-vertex has the degree of at least three, the other end-vertex is a pendant vertex,
and the rest of the vertices have degree two, is called a {\it pendant path}.

A path of length two adjacent to a vertex that has at least one child of degree at least three is called a $B_1$-branch.
A vertex $v$ with degree $k+1$, $k \geq 2$, together with $k$ pendant paths of length $2$ attached to it, comprised a so-called
$B_k$-branch.
The vertex $v$ is referred to as the center of the $B_k$-branch.
By attaching a vertex to a pendant vertex of $B_k$-branch, one obtains a so-called  $B_k^*$-branch.
Illustrations of $B_k$-, $B_k^*$-, $k\ge 1$, and $B_3^{**}$-branches
are given in Figure~\ref{B_k-branches}.
We will refer to them in general as $B$-branches.

\begin{figure}[h]
\begin{center}
\includegraphics[scale=0.75]{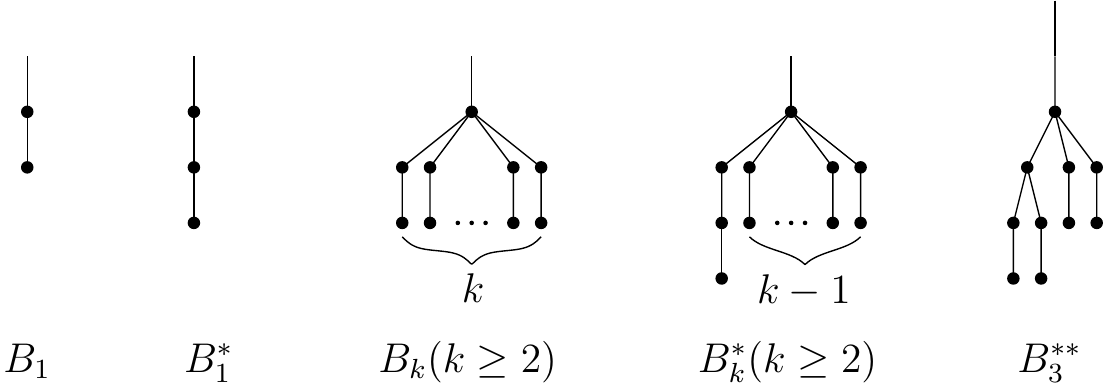}
\caption{$B_k$, $B_k^*$, $k \geq 1$ and $B_3^{**}$ branches.}
\label{B_k-branches}
\end{center}
\end{figure}

For the further analysis, besides
the  $D_z$-branches, which were introduced  in \cite{adgh-dctmabci-14},
we consider here also the $D_z^{**}$, $D_{z,x}^{2}$ $(x=1,2)$ and $D_{z,x}^{4}$ $(x=1, 2, 3, 4)$ branches depicted in Figure~\ref{fig-30},
which may occur in the minimal-ABC trees.
We will refer to them in general as $D$-branches.
The vertex of a $D$-branch to whom the $B$-branches are attached is referred to as the center of the $D$-branch.
Later, in Lemma~\ref{lemma-B4-root} it will be shown that $D_{z,x}^{4}$- $(x=1, 2, 3, 4)$ branches does not exist in minimal-ABC trees.

\begin{figure}[!ht]
\begin{center}
\includegraphics[scale=0.75]{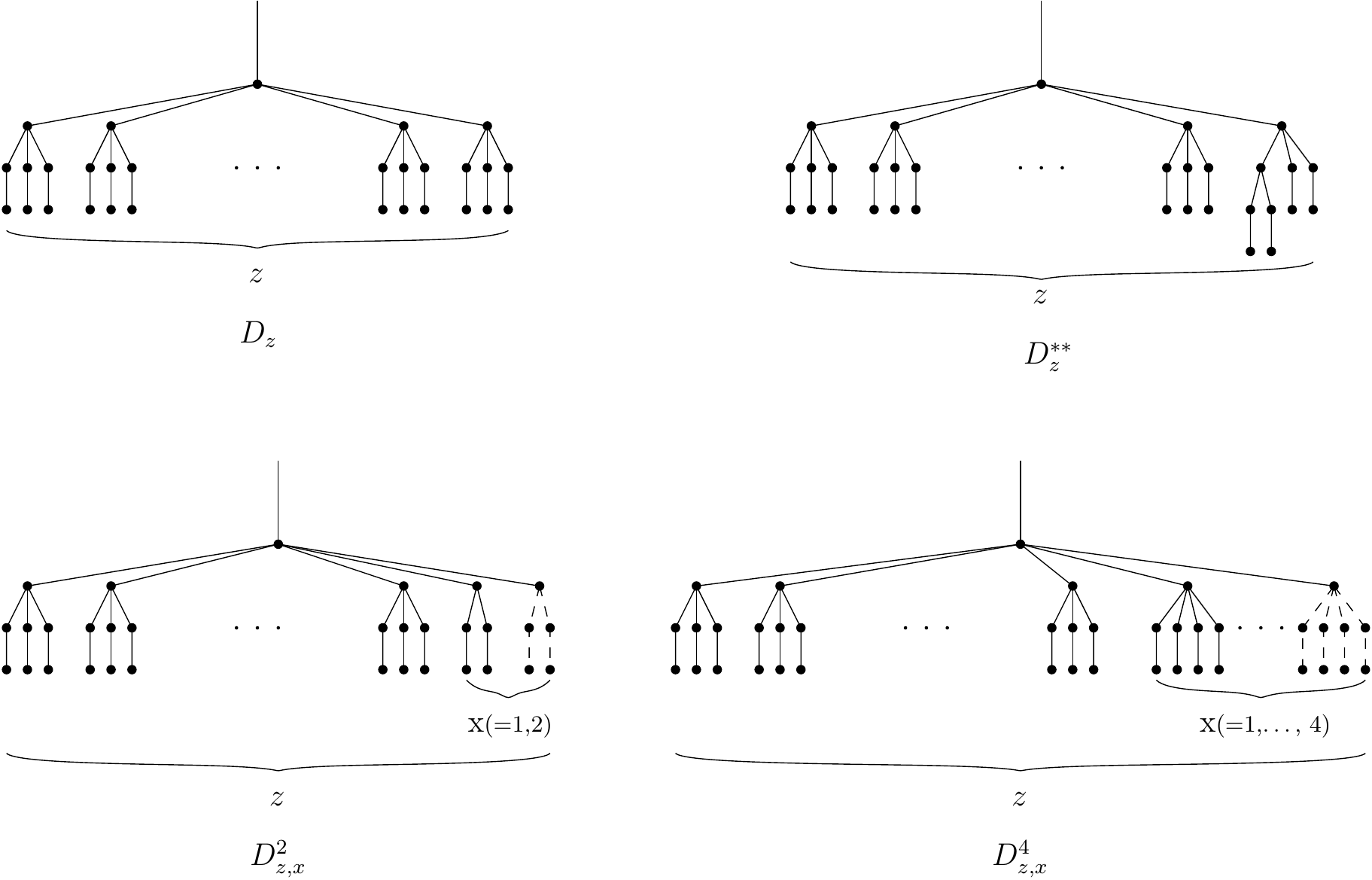}
\caption{$D_z$-, $D_z^{**}$-,  $D_{z,x}^{2}$- $(x=1,2)$ and $D_{z,x}^{4}$- $(x=1, 2, 3, 4)$ branches. The dashed line segments are optional.}
\label{fig-30}
\end{center}
\end{figure}
%


Here, we show that when the minimal-ABC trees are enough large then they do not contain $B_4$, $B_2$, $B_3^*$, $B_3^{**}$ -branches,
and consequently there are no other kinds of $D$-branches except $D_k$ and $D_{k+1}$-branches
(see Corollary~\ref{co-conj2_and_more} and Theorem~\ref{te-combinations-D-branches}).
As a corollary, we have that the radius of large minimal-ABC trees is at most $4$.
The presented theoretical results here reduce further the search space considerably such
 that the minimal-ABC trees of a given arbitrary order can be exactly determined
(see Section~\ref{section-compuation} and Appendix~\ref{appendix-figures}).

\subsection[Some prerequisite theoretical results]{Some prerequisite theoretical results}\label{subsec-2.1}

The following results are needed in some of the proofs later.
More known theoretical and computational results about minimal-ABC trees can be found in \cite{dd-mabctb2-2020, dm-ectmabcir-2018, lcwdh-csltmabci-18}.

\begin{pro}[\cite{d-sptmabci-2014}] \label{pro-10}
Let $h(x,y)=\sqrt{(x +y-2)/ (x y)}$.
The expression $-h(x,y)+h(x+\Delta x,y-\Delta y)$  increases in $x$ and decreases in $y$,  where $x, y, \Delta x, \Delta y  \in \mathbb{R}$
 and $x, y \geq 2$,  $\Delta x \geq 0$,  $0 \leq \Delta y < y$.

\end{pro}

Proposition~\ref{pro-10} can be re-stated in the following manner.

\begin{pro}  \label{pro-20}
The expression $-h(x,y)+h(x-\Delta x,y+\Delta y)$  decreases in $x$ and increases  in $y$,  where $x, y, \Delta x, \Delta y  \in \mathbb{R}$
 and $x, y \geq 2$,  $\Delta y \geq 0$,  $0 \leq \Delta x < x$.
\end{pro}

Lin, Gao, Chen, and Lin
\cite{lgcl-mabcicggds-13}, and Gan, Liu, and You in
\cite{gly-abctgds-12} contributed the so-called {\em switching transformation}.

\begin{lemma}[Switching transformation] \cite{lgcl-mabcicggds-13,gly-abctgds-12} \label{lemma-switching}
Let $pq$ and $rs$ be two edges of a connected graph $G=(V,E)$ such that $ps, qr \notin E$.
Let $G'= G - pq - rs + ps + qr$. If $d(p) \geq d(r)$ and $d(q) \leq d(s)$, then
$\ABC(G') \leq \ABC(G)$. The equality holds if and only if $d(p) = d(r)$ or $d(q) =d(s)$.
\end{lemma}

The following result gives an upper bound on the possible number of $B_4$-branches contained
in a minimal-ABC tree.

\begin{te}[\cite{d-sptmabci-2014}]  \label{thm-20}
There are no more than 4 $B_4$-branches in a minimal-ABC tree.
\end{te}

Some  forbidden combinations of branches are stated in the next theorems.

\begin{te}[\cite{DiDuFo2018}] \label{noB1B4}
A $B_1$-branch and a  $B_4$-branch cannot exist simultaneously in a minimal-ABC tree.
\end{te}

\begin{te}[\cite{DiDuFo2018}] \label{noB2B4}
A $B_2$-branch and a  $B_4$-branch cannot exist simultaneously in a minimal-ABC tree.
\end{te}

The following two conjectures have been arisen in  \cite{adgh-dctmabci-14}.

\begin{conjecture} [\cite{adgh-dctmabci-14}] \label{conjecture1}
The subgraph of a minimal-ABC tree induced by its big vertices is a star.
\end{conjecture}

\begin{conjecture} [\cite{adgh-dctmabci-14}] \label{conjecture2}
After some enough large $n$, besides the big vertices, minimal-ABC trees have only $B_3$-branches.
\end{conjecture}

Moreover, the structure of the minimal-ABC tree depicted in Figure~\ref{fig-dd-4},
for enough large trees, was conjectured in \cite{adgh-dctmabci-14}.

\begin{figure}[!ht]
\begin{center}
\includegraphics[scale=0.75]{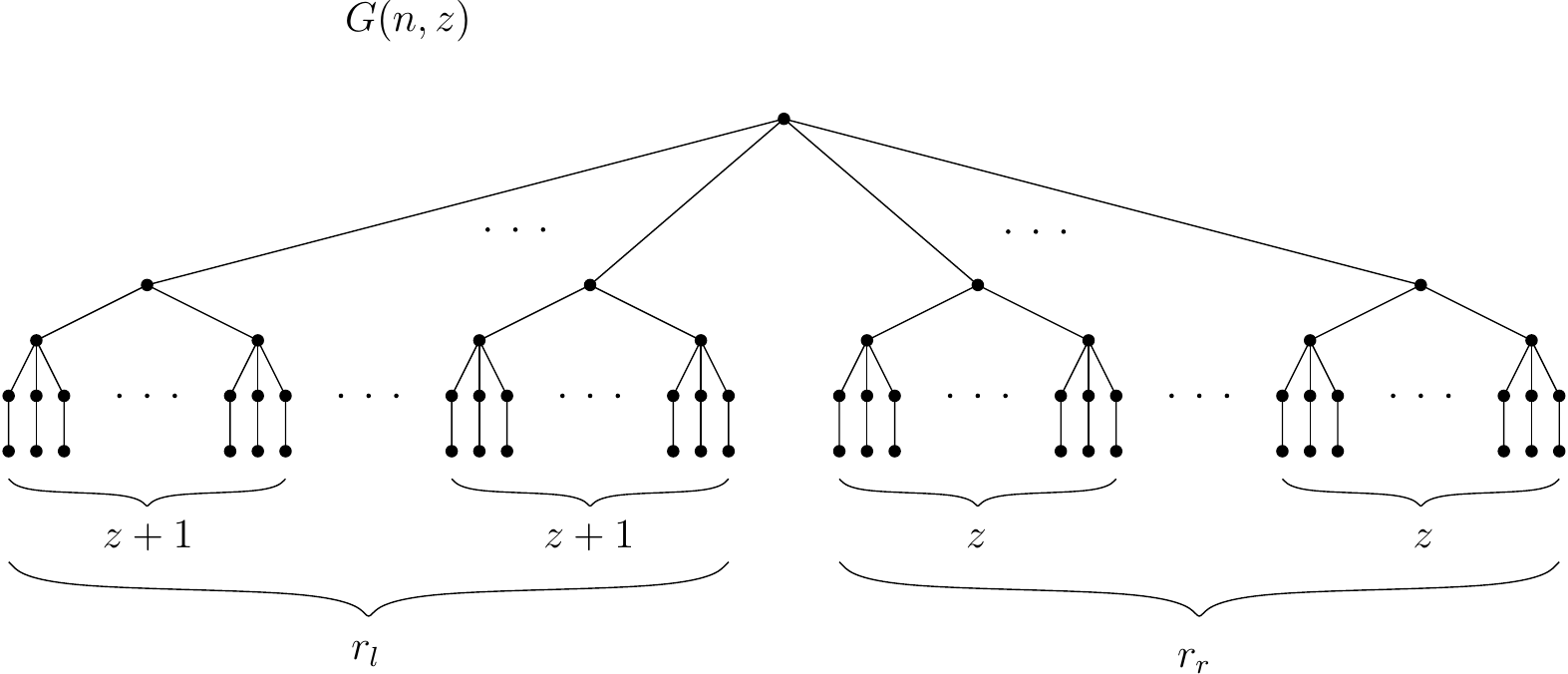}
\caption{The figure from \cite{adgh-dctmabci-14} with the conjectured structure of the minimal-ABC tree.}
\label{fig-dd-4}
\end{center}
\end{figure}

Recently, in \cite{dd-scbvmabct-2020} it was shown that Conjecture~\ref{conjecture1} is true.
Here we give also an affirmative answer to Conjecture~\ref{conjecture2} and show that
the conjectured structure from Figure~\ref{fig-dd-4} is indeed the structure of the minimal-ABC trees,
when the degree of its root is at least $2838$.

The following results give  bounds on the size $z$ of a $D_z$-branch and
the possible combinations of $D$-branches.

\begin{lemma}[\cite{dd-scbvmabct-2020}]   \label{le-Dz-lowerBound}
For $z \le 14$, there is no minimal-ABC tree, which has  a $D_z$-branch.
\end{lemma}

\begin{re} \label{remark}
This result is also valid for other types of $D$-branches, which has been proved in \cite[Section 2]{dd-scbvmabct-2020}. In fact, somewhat better lower bounds are obtained there, but the (weaker) unified lower bound $15$ is good enough for our need throughout this paper.
\end{re}

\begin{lemma} [\cite{dd-scbvmabct-2020}]  \label{le-Dz-UpperBound}
A minimal-ABC tree does not contain a $D_z$-branch, $z \geq 132$.
\end{lemma}

Actually, the proof of Lemma \ref{le-Dz-UpperBound} is still valid for $D_z^{**}$-branches.

\begin{lemma}  \label{le-Dz-UpperBound-star}
A minimal-ABC tree does not contain a $D_z^{**}$-branch, $z \geq 132$.
\end{lemma}

\begin{lemma} \label{le-Dz-B2-UpperBound}
A minimal-ABC tree does not contain a $D_{z,x}^{2}$-branch, $z \geq 132$, for each $x=1, 2$.
\end{lemma}

\begin{proof}
Here, in order to simplify our proof, we recall the transformation (shown in Figure~\ref{fig-Dz_size10}) used to establish Lemma \ref{le-Dz-UpperBound} in \cite{dd-scbvmabct-2020}, and show that such upper bound $131$ is also valid for
$D_{z,x}^{2}$-branches, $x=1, 2$. We assume that  $z \equiv x\, (\bmod\,3)$ in Figure~\ref{fig-Dz_size10}, other two cases are analogous.

\begin{figure}[!ht]
\begin{center}
\includegraphics[scale=0.75]{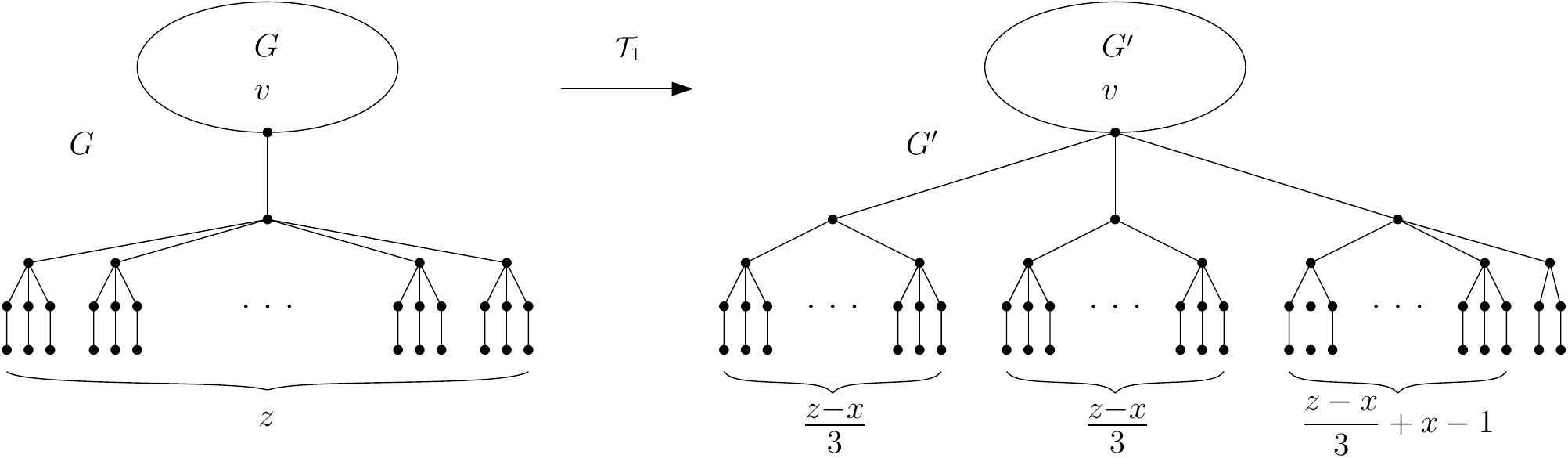}
\caption{The transformation referred in Lemma~\ref{le-Dz-B2-UpperBound}, where $z \equiv x\, (\bmod\,3)$.}
\label{fig-Dz_size10}
\end{center}
\end{figure}

Let $A_1$ denote the change of ABC index after applying the transformation depicted in Figure~\ref{fig-Dz_size10}, in which all of $B$-branches occurring in $G$ outside $\bar{G}$ are $B_3$-branches. If we replace $x$ $B_3$-branch(es) by $B_2$-branch(es), $x=1,2$, then we get a transformation aiming to $D_{z,x}^{2}$-branches, and denote by $A_{2,x}$ the change of ABC index, for $x=1,2$.

In \cite[Lemma 2.5]{dd-scbvmabct-2020}, it is shown that $A_1 < 0$ when $z \geq 132$, so we can conclude the desired result by noting that $A_{2,x} < A_1$ for $x=1,2$.
Indeed, it is easy to obtain that
\beq
A_{2,x} - A_1 & = & x \left( \left( - f(z+1,3) + f \left(\frac{z-x}{3} + 1, 3 \right) \right) -  \left( - f(z+1,4) + f \left(\frac{z-x}{3} + 1, 4 \right) \right) \right) < 0. \nonumber
\eeq
The negativity of the above difference follows by Proposition~\ref{pro-20}.
\end{proof}

\begin{re}
The upper bound $131$ is not optimal. One can improve it by considering the particular structures of $D_{z,x}^{2}$-branches.
Since we do not think that it would affect significantly the subsequent proofs, we omit such a feasible improvement.
\end{re}

\begin{lemma} \label{le-Dz-B4-UpperBound}
A minimal-ABC tree does not contain a $D_{z,x}^{4}$-branch, $z \geq 216$, for each $x=1, 2, 3, 4$.
\end{lemma}

\begin{proof}
We apply the transformation $\mathcal{T}_{2}$, depicted in Figure~\ref{fig-B4_to_root-1}, to prove the above result.
The change of ABC index after applying $\mathcal{T}_{2}$ is
\begin{eqnarray}\label{eq:B4-1-ori}
&&ABC (G') - ABC (G) \nonumber\\
&=& - f (d(v),z+1) + f(d(v),z) + (-f (d(v),d(w)) + f (d(v),d(w) + 1))  + x (-f(z+1,5) + f(z,5))    \nonumber\\
&&  - (z-x) f(z+1,4) + (z-x-1)  f(z,4) + f(d(w) + 1,4) + (d(w) - 1)  (-f(d(w),4) + f(d(w) + 1,4)).  \nonumber
\end{eqnarray}

\begin{figure}[!ht]
\begin{center}
\includegraphics[scale=0.739]{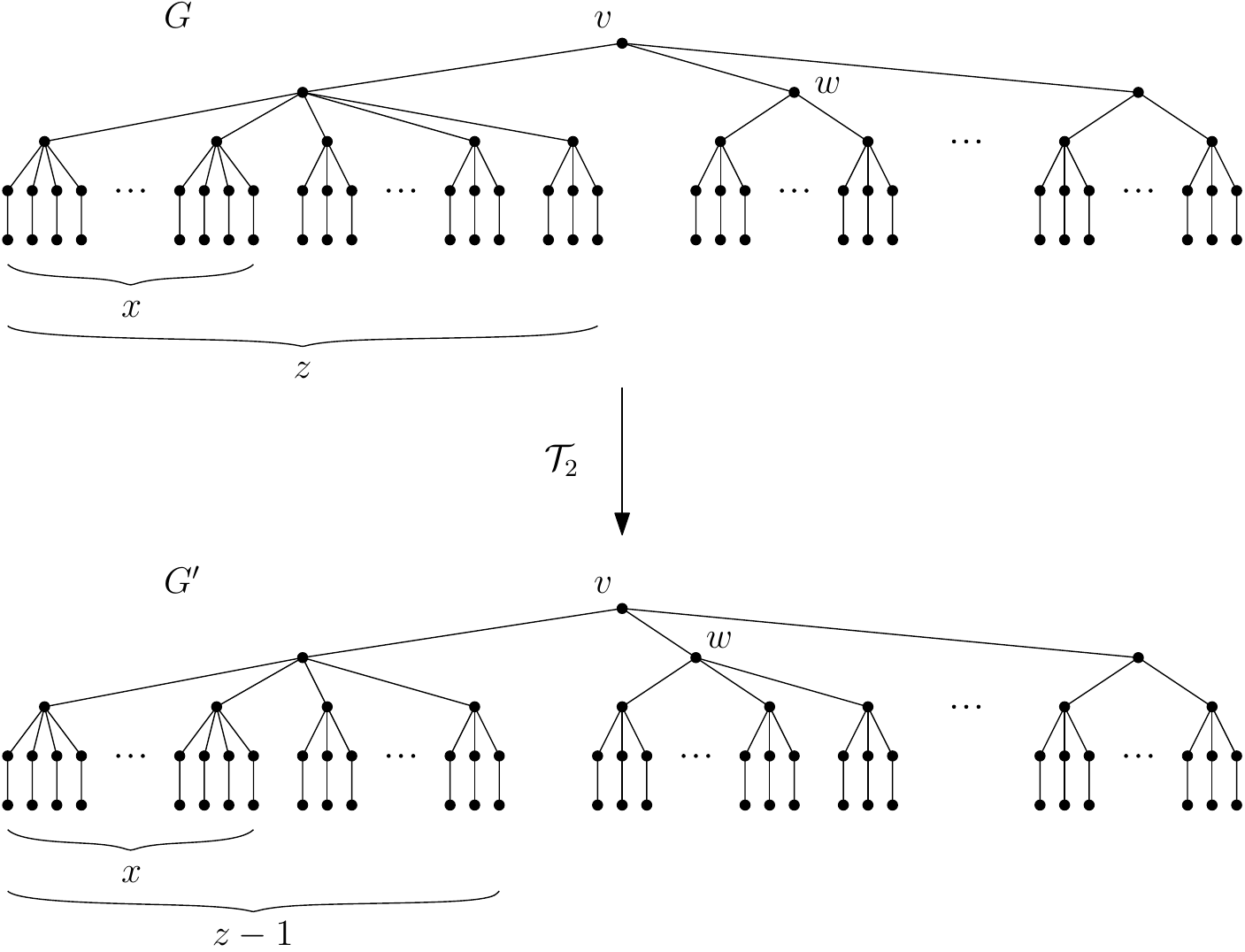}
\caption{Transformation used to derive an upper bound on the size of $D_{z,x}^4$-branches, $x=1,2,3,4$.}
\label{fig-B4_to_root-1}
\end{center}
\end{figure}

From Propositions~\ref{pro-10} and \ref{pro-20}, $- f (d(v),z+1) + f(d(v),z)$ increases in $d(v)$, thus
\begin{eqnarray*}
- f (d(v),z+1) + f(d(v),z) &\le& \lim_{d(z) \rightarrow + \infty}(- f (d(v),z+1) + f(d(v),z) ) \quad = \quad -\sqrt{\frac{1}{z+1}} + \sqrt{\frac{1}{z}},
\end{eqnarray*}
and $-f (d(v),d(w)) + f (d(v),d(w) + 1)$ decreases in $d(v) \ge z+1$, i.e.,
$$
-f (d(v),d(w)) + f (d(v),d(w) + 1) \le -f (z+1,d(w)) + f (z+1,d(w) + 1).
$$
Moreover, it is not hard to verify that $- (z-x) f(z+1,4) + (z-x-1)  f(z,4)$ increases in $z$ for each of fixed $1 \le x \le 4$, and its limit is $- \frac{1}{2}$ when $z$ goes to infinity, i.e.,
$$
- (z-x) f(z+1,4) + (z-x-1)  f(z,4) < - \frac{1}{2}.
$$
Combining all above comments, it leads to an upper bound of $ABC (G') - ABC (G)$:
\begin{eqnarray}\label{eq:B4-1}
ABC (G') - ABC (G) &\le& - \sqrt{\frac{1}{z + 1}} + \sqrt{\frac{1}{z}} + ( -f (z+1,d(w)) + f (z+1,d(w) + 1) ) + x (-f(z+1,5)  \nonumber\\
&& + f(z,5)) - \frac{1}{2} + f(d(w) + 1, 4) + (d(w) - 1) (-f(d(w),4) + f(d(w) + 1, 4)).  \nonumber
\end{eqnarray}
The right-hand side of the above inequality decreases in $z$ and
is negative for $16 \le d (w) \le 132$, $1 \le x \le 4$ and $z = 216$ (and so is for $z \ge 216$).
\end{proof}

\begin{re}
The above unified upper bound of $215$ can be improved by involving particular structural properties
of $D_{z,x}^4$-branches, but the current upper bound suffices for our needs.
\end{re}

\begin{pro}  [\cite{d-sptmabci-4-2017}] \label{pro-diff-Dz}
Let $x$ and $y$ be vertices of a minimal-ABC tree $G$ that have a common parent vertex $z$,  such that $d(x) \geq d(y) \geq 5$.
If $x$ and $y$ have only $B_3$-branches as children, and at most one of them is $B_3^{**}$, then either $d(y) = d(x)$ or  $d(y) = d(x)-1$.
\end{pro}

Next, we show that if a minimal-ABC tree has $B_4$-branches, then they must be adjacent to the root.

\begin{lemma} \label{lemma-B4-root}
Let $G$ be a minimal-ABC tree.
Then all  $B_4$-branches (maximum $4$) are adjacent to the root vertex of $G$.
\end{lemma}
\begin{proof}
Assume the opposite, that  there is at least one $B_4$-branch attached to a child of the root
(i.e., some $D_{z,x}^4$-branch exists). Recall that $15 \le z \le 215$, from Remark \ref{remark} and Lemma \ref{le-Dz-B4-UpperBound}. It is also worth mentioning that the existence of $D_{z,x}^4$-branch allows us to set an assumption that no $B_3$-branch is attached to the root vertex, due to the switching transformation (Lemma~\ref{lemma-switching}).
Then we move all such  $B_4$-branches to the root vertex
(illustrated  in Figure \ref{fig-B4_to_root-2} and denoted as transformation $\mathcal{T}_{3}$).
\begin{figure}[!ht]
\begin{center}
\includegraphics[scale=0.739]{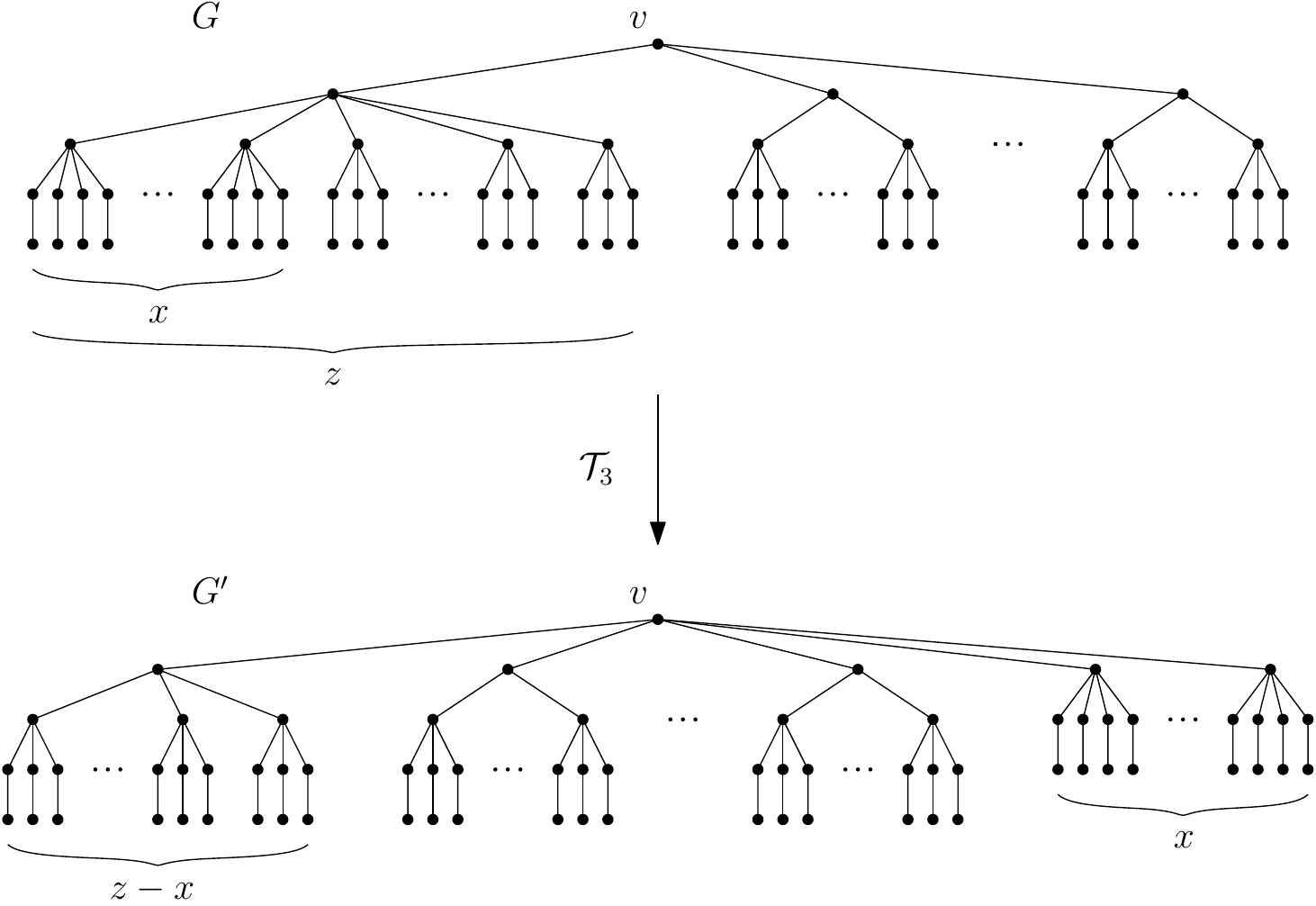}
\caption{Transformation which helps us to show that all  $B_4$-branches are attached to the root vertex of a minimal-ABC tree.}
\label{fig-B4_to_root-2}
\end{center}
\end{figure}
After applying this transformation, the change of ABC index is
\begin{eqnarray}\label{eq:B4-2-ori}
 ABC (G') - ABC (G) &=& -  f(d(v), z+1) + f (d(v) + x, z-x+1) - x \, f(z + 1, 5) - (z-x) f(z+1,4) \nonumber\\
&& + (z-x) f (z-x+1,4) + x \, f(d(v)+x,5) + \sum_{i= 1}^{d(v)-1} (-f(d(v), d(u_i))  \nonumber\\
&& + f(d(v)+x, d(u_i))),
\end{eqnarray}
where $d(u_i)$, $i = 1,2, \dots, d(v) - 1$, are the neighbors of $v$, except the center of $D_{z,x}^4$-branch.
We show that the right-hand side of (\ref{eq:B4-2-ori}) is negative in Appendix~\ref{appendix-B4-to-root}
\footnote{Due to interdependencies of the results in the appendices, at the end of the paper
they do not occur in the same order as they appear in the main text.},
for $1 \le x \le 4$ and $15 \le z \le 215$, which concludes
the proof.
\end{proof}

\begin{co}
A minimal-ABC tree does not contain $D_{z,x}^4$-branches, $x=1,2,3,4$.
\end{co}

\section[The size of $D_z$-branches]{The size of $D_z$-branches}

Here we give an upper bound on the size $z$ of $D_z$-branches depending
on their occurrences in a minimal-ABC tree.

\begin{lemma} \label{le-UpperBoundOnNumberOf_Dz-10}
If a minimal-ABC tree contains at least $261$ $D_z$-branches, then $z \leq 52$.
\end{lemma}
\begin{proof}
Assume that the minimal-ABC tree $G$ has (at least) $x$ $D_z$-branches.
Recall that from Lemma \ref{le-Dz-UpperBound}, we may assume that $z \le 131$.
%
%
We apply the transformation $\mathcal{T}_{4}$ depicted in Figure~{\ref{fig-Dz_size30}}.
\begin{figure}[!ht]
\begin{center}
\includegraphics[scale=0.75]{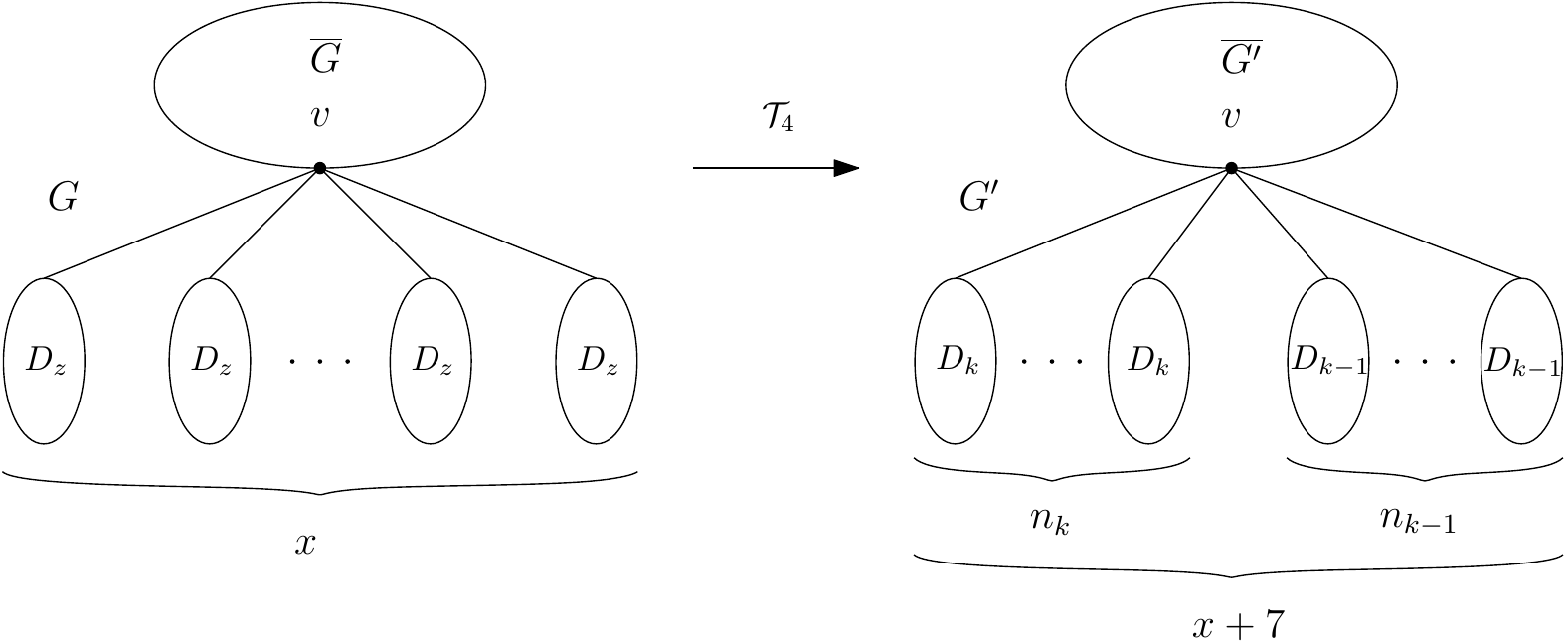}
\caption{The transformation which leads to upper bounds on the sizes of $D_z$-branches depending on their occurrences in minimal-ABC trees.}
\label{fig-Dz_size30}
\end{center}
\end{figure}
The change of the ABC index after applying $\mathcal{T}_{4}$ is
\begin{eqnarray}  \label{eq-Lemma-Dz-30-00}
ABC (G') - ABC (G) &=&-x f(d(v),z+1) -x \, z \, f(z+1,4) - 3 f(4,2) - 3 f(2,1)  \nonumber \\
&&  + n_{k}  f(d(v)+7,k+1)  + n_{k-1}  f(d(v)+7,k) + k n_{k}  f(k+1,4)  +  (k-1) n_{k-1}  f(k,4) \nonumber \\
&&  + \sum_{i=1}^{d(v)-x} ( - f(d(v),d(v_i)) + f(d(v)+7,d(v_i))),
\end{eqnarray}
where $d(v_i)$, $i=1, \dots, d(v)-x$, are the children of $v$, which are in $\overline{G}$ ({$\overline{G'}$}).
For fixed $z$, the parameters $x$, $k$, $n_k$, $n_{k-1}$ must satisfy
\begin{eqnarray}  \label{eq-Lemma-Dz-10-1}
n_{k} + n_{k-1}=x+7  \quad \text{and} \quad k n_{k}  + (k-1) n_{k-1} = x \, z - 1.
\end{eqnarray}
It can be shown that the right-hand side of \eqref{eq-Lemma-Dz-30-00} is always negative, for $53 \le z \le 131$ and $x \geq 261$.
Detailed proofs and results (for every particular $z$) are presented in Appendix~\ref{appendix-upper_bound}.
\end{proof}

\begin{lemma} \label{le-LowerBoundOnNumberOf_Dz}
A minimal-ABC tree with a root of degree at least $3249$ cannot contain more than $364$  $D_{z}$-branches, for $15 \le z \le 51$.
\end{lemma}

\begin{proof}
Consider the transformation $\mathcal{T}_{5}$ depicted in Figure~\ref{fig-Dz_size-LowerBound}.
\begin{figure}[!ht]
\begin{center}
\includegraphics[scale=0.75]{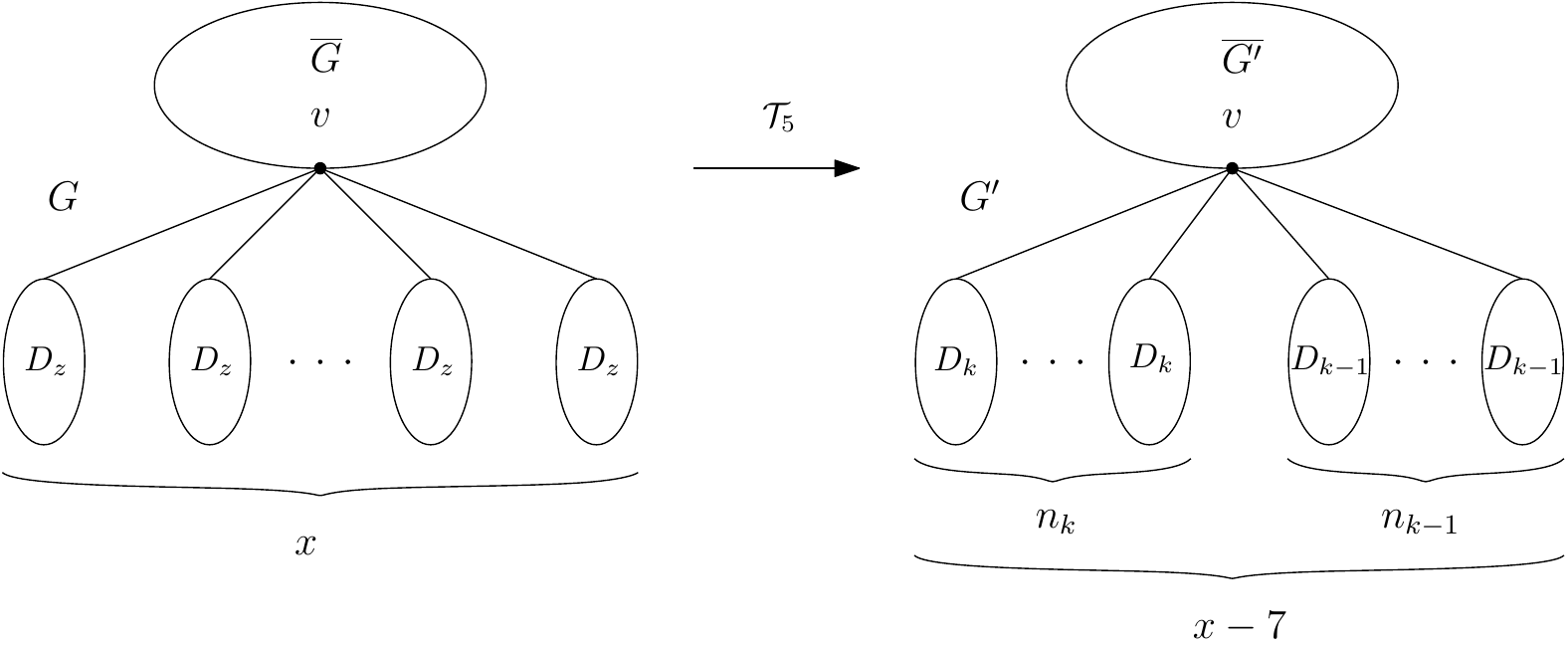}
\caption{The transformation which leads to upper bounds on the sizes of $D_z$-branches depending on their occurrences in minimal-ABC trees
and the degrees of their root vertices.}
\label{fig-Dz_size-LowerBound}
\end{center}
\end{figure}
The change of the ABC index after applying $\mathcal{T}_{5}$ is
\begin{eqnarray}  \label{eq-Lemma-Dz-30-00-second-pre}
%
ABC (G') - ABC (G) &=&-x \, f(d(v),z+1) -x \, z \, f(z+1,4) + 3 f(4,2) + 3 f(2,1) \nonumber \\
&&  + n_{k}  f(d(v)-7,k+1)  + n_{k-1}  f(d(v)-7,k) + k n_{k}  f(k+1,4)   \nonumber \\
&&  +  (k-1) n_{k-1}  f(k,4) + \sum_{i=1}^{d(v)-x} ( - f(d(v),d(v_i)) + f(d(v)-7,d(v_i)))
\end{eqnarray}
with constraints
\begin{eqnarray}  \label{eq-Lemma-Dz-10-12}
n_{k} + n_{k-1}=x-7  \quad \text{and} \quad k n_{k}  + (k-1) n_{k-1} = x \, z + 1,
\end{eqnarray}
where $d(v_i)$, $i=1, \dots, d(v)-x$, are the children of $v$, which are in $\overline{G}$ ({$\overline{G'}$}).
It can be shown that the right-hand side of \eqref{eq-Lemma-Dz-30-00-second-pre} is always negative, for $15 \le z \le 51$, $d(v) \geq  x \geq 365$.
We present detailed proofs and particular results, for every $z$, in Appendix~\ref{appendix-upper_bound-2}.
\end{proof}

\section[$B$-branches adjacent to the root]{$B$-branches adjacent to the root }

By Lemma~\ref{lemma-B4-root}, if there are $B_4$-branches in a minimal-ABC tree, then they must be attached to the root.
Also, the $B_2$-branches cannot be attached to the root, otherwise, we can apply a switching transformation (Lemma \ref{lemma-switching})
between one $B_2$-branch attached to the root and one $B_3$-branch from a $D$-branch.
Using the same argumentation, we can also assume that there is no $B_3^{**}$-branch attached to the root.

\begin{lemma} \label{thm-bound-B-branches-to-root}
The number of $B$-branches  adjacent to the root vertex of a  minimal-ABC tree is at most $919$.
\end{lemma}
\begin{proof}
Assume that there are at least $x$ $B$-branches ($B_3$- and $B_4$-branches) adjacent to the root vertex $v$ (possibly there are other branches adjacent to $v$, besides the $x$ $B$-branches). Recall that there are at most $4$ $B_4$-branches, and assume that all $B_4$-branches are counted in the $x$ $B$-branches.
We apply the transformation $\mathcal{T}_{6}$ depicted in Figure~\ref{fig-t3}.
\begin{figure}[!ht]
\begin{center}
\includegraphics[scale=0.735]{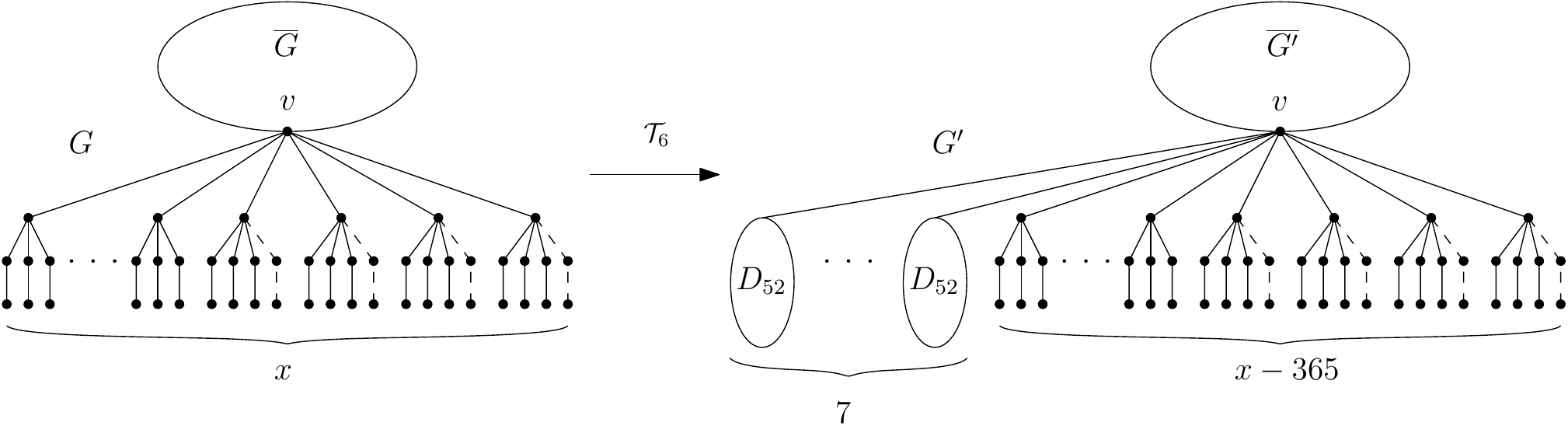}
\caption{The transformation which helps us to obtain an upper bound on the number of $B$-branches attached to the root (the dashed line segments are optional).}
\label{fig-t3}
\end{center}
\end{figure}
After applying $\mathcal{T}_{6}$, the change of the ABC index is
\begin{eqnarray} \label{eq-thm-bound-B3-branches-to-root}
%
ABC (G') - ABC (G) &=& \sum_{i=1}^{d(v)-x}( -f(d(v),d(v_i))+f(d(v'),d(v_i)))
-f(d(v),4)+f(d(v'),53)  \nonumber \\
&&+6(-f(2,1)+f(d(v'),53))  +364(-f(d(v),4)+f(53,4))  \nonumber \\
&& +(x-365-n_4)(-f(d(v),4)+f(d(v'),4)) +n_4(-f(d(v),5)+f(d(v'),5)), \qquad
\end{eqnarray}
where $d(v')=d(v)-358$, $n_4$ is the number of $B_4$-branches attached to $v$, and $v_i$, $i=1, \dots, d(v)-x$, are the  vertices adjacent to $v$ in $\overline{G}$ ({$\overline{G'}$}). It is worth mentioning that $v_i$ is the center of either a $D$-branch or a $B_3$-branch.
When $x \ge 920$, it can be verified that the right-hand side of \eqref{eq-thm-bound-B3-branches-to-root} is always negative. We present the detailed proofs and particular results, for every possible combination of $D$-branches, in Appendix~\ref{appendix-bound-B-branches-to-root}.
\end{proof}


\subsection[Existence of $B_4$-branches]{Existence of $B_4$-branches}

Here, we show that if a minimal-ABC tree is large enough, then it does not contain $B_4$-branches.

\begin{lemma} \label{lemma-no-B4-branches}
A minimal-ABC tree, whose root has a degree of at least $1228$, does not contain $B_4$-branches.
\end{lemma}

\begin{proof}
Assume that there exists at least one $B_4$-branch in a minimal-ABC tree $G$. The existence of $B_4$-branch means that $B_1$- or $B_2$-branches would not occur, from Theorems \ref{noB1B4} and \ref{noB2B4}, so each $D$-branch occurring in $G$ should be a $D_z$-branch for some $z$.

Let us consider first the case when $G$ contains a combination of $D_{49}$- and $D_{48}$-branches.
Since  $d(v) \geq 1228$, by Lemma~\ref{le-LowerBoundOnNumberOf_Dz} (Table \ref{table-small-z}), there can be at most $121$ $D_{49}$-branches
and $91$ $D_{48}$-branches. By Lemma~\ref{thm-bound-B-branches-to-root},
there cannot be more than $919$ $B$-branches attached to the root vertex-
actually in the case with a combination of $D_{49}$- and $D_{48}$-branches there cannot be more than $824$ $B$-branches attached to the root vertex (see Table \ref{table-combin-small-z}).
Thus, in this case it holds that $d(v) \leq 121 + 91 + 824 = 1036 < 1228$, which is a contradiction to the initial assumption that $d(v) \ge 1228$.
Similarly, we can show that the cases when $G$ contains a combination $D_{z}$- and $D_{z-1}$-branches, for $15 \leq z \leq 48$ and $55 \leq z \leq 131$,
are not possible.

It remains to prove the claim of the lemma when $G$ contains $D_{49}$- and $D_{50}$-branches; $D_{50}$- and $D_{51}$-branches; $D_{51}$- and $D_{52}$-branches;
$D_{52}$- and $D_{53}$-branches; and $D_{53}$- and $D_{54}$-branches.
Consider the transformation from Figure~\ref{fig-no-B4}.
Recall that there can be at most four $B_4$-branches and they are attached to the root of the minimal-ABC tree, from Theorem~\ref{thm-20} and Lemma~\ref{lemma-B4-root}, so in $\overline{G}$ there can be at most three $B_4$-branches, $B_3$-branches, and $D$-branches.

Since a tree may contain a combination of $D_{z}$- and $D_{z+1}$-branches, with $49 \le z \le 53$,
the transformation $\mathcal{T}_{7}$ to be feasible, there must be at least
$2z-3$ $D_{z}$-branches or $2(z+1)-3$ $D_{z+1}$-branches.
Observe that when this is not the case, then $d(v) < 2z-3 + 2(z+1)-3 + 919 < 1228$,
and thus the transformation $\mathcal{T}_{7}$ can be applied.
Therefore,  to be consistent with the notation in Figure~\ref{fig-no-B4},
we further assume that $G$ contains (at least) $2z-3$ $D_{z}$-branches, $49 \le z \le 54$,
and in addition to possibly either $D_{z-1}$-branches or $D_{z+1}$-branches.

\begin{figure}[!ht]
\begin{center}
\includegraphics[scale=0.75]{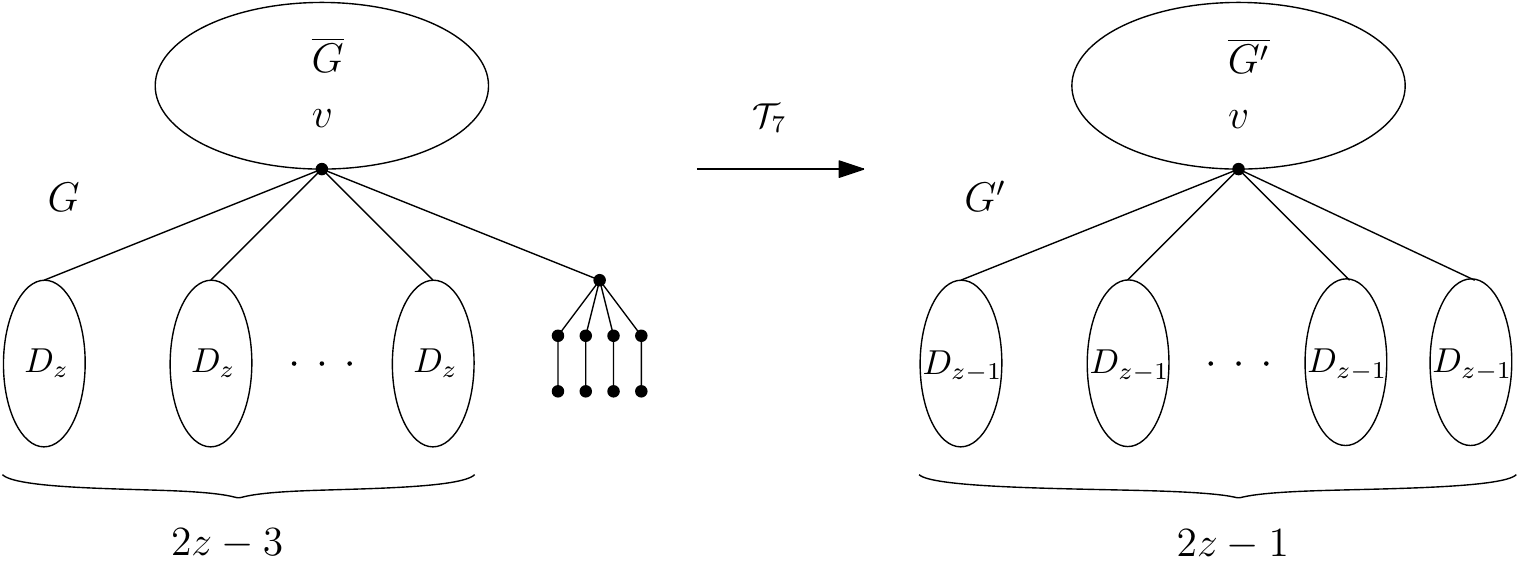}
\caption{Suggested transformation, which shows that there are no $B_4$-branches, when a minimal-ABC tree is large enough. }
\label{fig-no-B4}
\end{center}
\end{figure}
After applying $\mathcal{T}_{7}$, the change of the ABC index is
\begin{eqnarray}  \label{eq-lemma-no-B4-10}
ABC (G') - ABC (G) &=& -(2z-3) f(d(v),z+1)  -(2z-3) z \, f(z+1,4) -  f(d(v),5) - f(4,2) - f(2,1)  \nonumber \\
&&  + (2z-1) f(d(v)+1,z)  + (2z-1) (z-1) f(z,4)  \nonumber \\
&& + n_3 (-f(d(v),4)+f(d(v)+1,4)) + (n_4 - 1) (-f(d(v),5)+f(d(v)+1,5))  \nonumber \\
&& + \sum_{i=1}^{d(v)-2z+3 -n_3 -n_4} (- f(d(v),d(v_i)) + f(d(v) + 1, d(v_i))),
\end{eqnarray}
where $n_3$ is the number
of $B_3$-branches attached to the root, $n_4$ is the number of $B_4$-branches (attached to the root),
and $d(v_i)$, $i=1, \dots, d(v)-2z+3-n_3-n_4$, are the children of $v$, which are centers of $D$-branches in $\overline{G}$ ({$\overline{G'}$}). By Theorem~\ref{thm-20}, $n_4$ can be at most $4$, and $n_3 + n_4 \le 919$ by Lemma~\ref{thm-bound-B-branches-to-root}.

Verification shows that the right-hand side of (\ref{eq-lemma-no-B4-10}) is negative.
The detailed deductions are presented  in Appendix \ref{appendix-exist-B4}.
\end{proof}
%

\subsection[$B$-branches attached to the root]
{$B$-branches attached to the root}

\begin{lemma} \label{lemma-no-B-branches-10}
If the root vertex of a minimal-ABC tree has a degree of at least $2956$,
then there are no $B$-branches attached to the root.
\end{lemma}

\begin{proof}
Assume that the claim of the lemma is not true and there are some $B$-branches attached to the root.
We apply the transformation $\mathcal{T}_{8}$ depicted in Figure~\ref{fig-no-B-branches}.
Since $d(v) \geq 2956$, by Lemma~\ref{lemma-no-B4-branches} there are no $B_4$-branches.
Here similarly as in Lemma~\ref{lemma-no-B4-branches},
due to Lemmas~\ref{le-UpperBoundOnNumberOf_Dz-10}~and~\ref{le-LowerBoundOnNumberOf_Dz} (Tables \ref{table-large-z} and \ref{table-small-z}) and the fact that $d(v) \geq 2956$,
the only possible combinations of $D$-branches are:
$D_{50}$- and $D_{51}$-branches; $D_{51}$- and $D_{52}$-branches; and $D_{52}$- and $D_{53}$-branches, together with a possible $D$-branch containing $B_2$- or $B_3^{**}$-branch.
\begin{figure}[!ht]
\begin{center}
\includegraphics[scale=0.75]{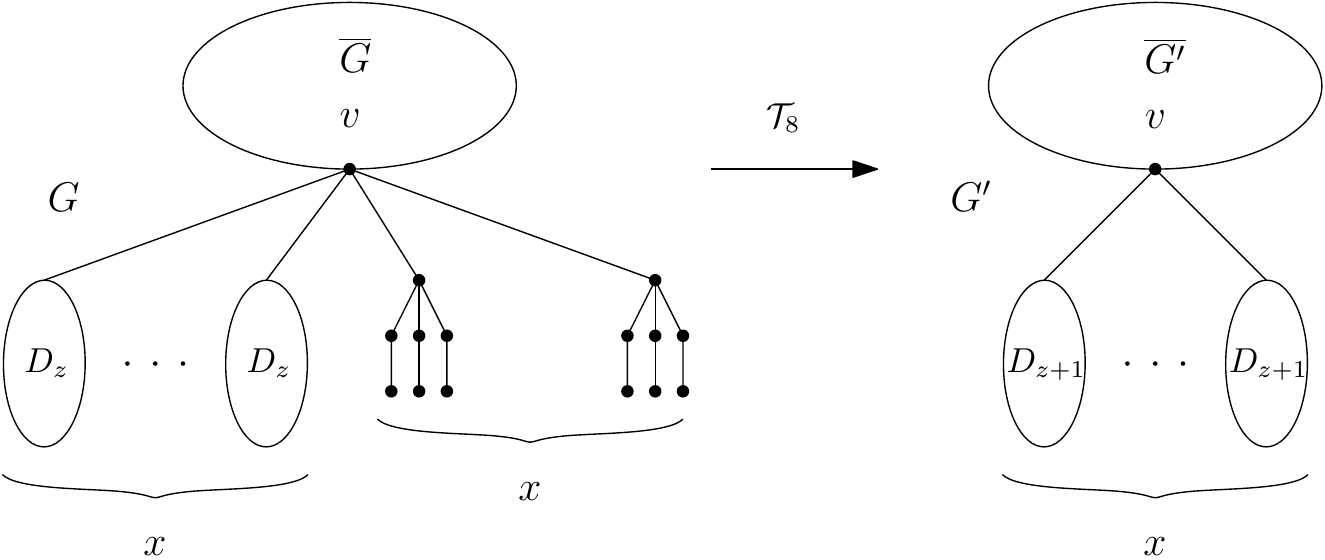}
\caption{The transformation $\mathcal{T}_{8}$, which shows that there are no
$B$-branches attached to the root, when the minimal-ABC tree is sufficiently large. }
\label{fig-no-B-branches}
\end{center}
\end{figure}

The change of the ABC index after applying $\mathcal{T}_{8}$ is
\begin{eqnarray}  \label{eq-lemma-no-B-branches-to-root}
ABC (G') - ABC (G) &=& x( -f(d(v),z+1)+f(d(v)-x,z+2)) +x z (-f(z+1,4)+f(z+2,4))  \nonumber \\
&&  +x( -f(d(v),4)+f(z+2,4)) \nonumber \\
&& + \sum_{i = 1}^{d(v)-2x} (- f(d(v),d(v_i))+f(d(v)-x,d(v_i))),
\end{eqnarray}
where $v_i$, $i = 1,\dots, d(v) - 2x$, are the children of $v$ in $\overline{G}$ ({$\overline{G'}$}).
A straightforward verification shows that the right-hand side of (\ref{eq-lemma-no-B-branches-to-root}) is negative. The detailed deductions can refer to Appendix \ref{appendix-no-B-branches-to-root}.
\end{proof}


\section[On the size and existence of $D_{z,1}^2$-, $D_{z,2}^2$- and $D_{z}^{**}$-branches]
{On the size and existence of $D_{z,1}^2$-, $D_{z,2}^2$- and $D_{z}^{**}$-branches}

\subsection[$D_{z,1}^2$- and $D_{z,2}^2$-branches]
{$D_{z,1}^2$- and $D_{z,2}^2$-branches}

\begin{pro} \label{pro-Dz-B2-10}
If a minimal-ABC tree contains a $D_{z,2}^2$-branch, then $ 25 \leq z \leq 50$.
If a minimal-ABC tree contains a $D_{z,1}^2$-branch, then $19 \leq z \leq 97$.
\end{pro}

\begin{proof}
The above upper and lower bounds on $z$ for a $D_{z,2}^2$-branch were presented in \cite[Lemma 14]{dd-mabctb2-2020}, and the lower bound on $z$ for
 a $D_{z,1}^2$-branch was given in \cite[Lemma 15]{dd-mabctb2-2020}. Here, we need only to show that if a minimal-ABC tree contains a $D_{z,1}^2$-branch, then $z \leq 97$.

Assume that there is a $D_{z,1}^2$-branch in a minimal-ABC tree, with $z \geq 98$. Recall that $z \le 131$ from Lemma~\ref{le-Dz-B2-UpperBound}.
Also, assume that $z$ is odd (the case with even $z$ can be proved analogously).
Now, apply the transformation $\mathcal{T}_{9}$ depicted in Figure~\ref{fig-Dz*_size20}.
\begin{figure}[!ht]
\begin{center}
\includegraphics[scale=0.75]{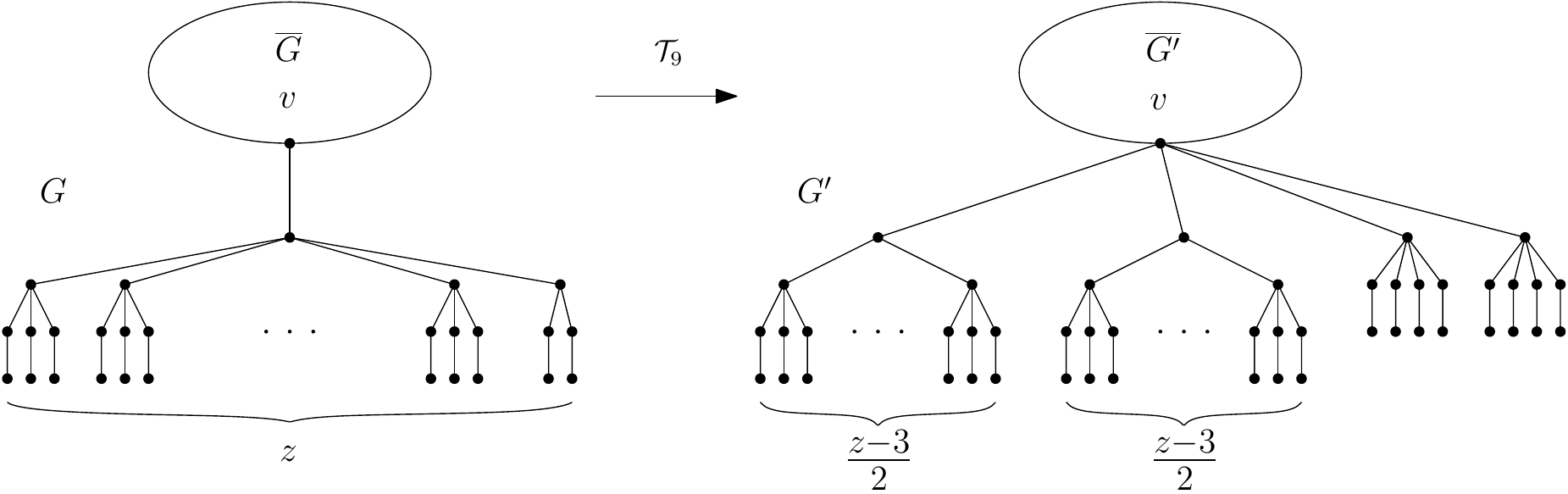}
\caption{The transformation applied to obtain an upper bound on the size of a $D_{z,1}^2$-branch.}
\label{fig-Dz*_size20}
\end{center}
\end{figure}
After applying $\mathcal{T}_{9}$, the change of the ABC index is
\begin{eqnarray} \label{eq-Lemma-Dz*-10-00}
%
ABC (G') - ABC (G) &=& \sum_{x v \in E(\bar{G})} ( - f(d(v), d(x)) + f (d(v) + 3, d(x))) - f (d(v),z+1)  \nonumber \\
&& + 2 f \left( d(v) + 3, \frac{z-1}{2} \right)   - f (z+1,3) + 2 f ( d(v) + 3, 5) \nonumber \\
&& - (z-1) f (z+1,4) + (z-3) f \left( \frac{z-1}{2}, 4 \right).
\end{eqnarray}
We can show that the right-hand side of \eqref{eq-Lemma-Dz*-10-00} is always negative when $99 \le z \ge 131$ (note that $z$ is odd here), in which the details are given in
Appendix \ref{appendix-nonexist-B2-1}.
\end{proof}


Next, we consider the difference between $D$-branches.
The following proposition is a modification of Proposition~\ref{pro-diff-Dz} and includes only the fact
that the vertex $y$ may have a $B_2$-branch as a child.
\begin{pro} \label{pro-P3-20}
Let $x$ and $y$ be vertices of a minimal-ABC tree $G$ that have a common parent vertex $z$,
such that $d(x) \geq d(y) \geq 7$. If $x$ has only $B_3$-branches as children and $y$ has
$B_3$-branches and $one$ $B_2$-branch as children, then $d(x)  \leq d(y) +5$.
If $y$ has  $B_3$-branches and $two$ $B_2$-branches as children, then $d(x)  \leq d(y) + 9$.
\end{pro}

\begin{proof}

First, assume that $y$ has one $B_2$-branch as a child.
We estimate the differences between $d(x)$ and $d(y)$ by applying transformation $\mathcal{T}_{10}$ illustrated in Figure~\ref{fig-tz}.
\begin{figure}[!ht]
\begin{center}
\includegraphics[scale=0.75]{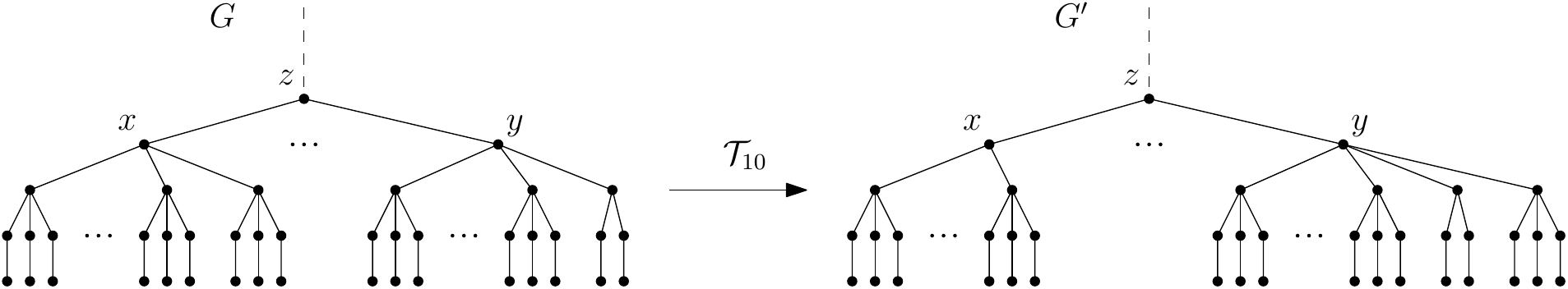}
\caption{Suggested transformation, which shows that the differences between the sizes of $D$-branches are small.}
\label{fig-tz}
\end{center}
\end{figure}

After applying $\mathcal{T}_{10}$, the degree of the vertex $x$ decreases by $1$,
while the degree of $y$ increases by $1$.
The rest of the vertices do not change their degrees.
The change of the ABC index is then
\begin{eqnarray} \label{eq-pro-P3-10-00}
ABC (G') - ABC (G) &=&-f(d(z),d(x))+f(d(z),d(x)-1)  -f(d(z),d(y))+f(d(z),d(y)+1)  \nonumber \\
&& +(d(x)-2)(-f(d(x),4)+f(d(x)-1,4)) -f(d(x),4)+f(d(y)+1,4)  \nonumber \\
&& +(d(y)-2)(-f(d(y),4)+f(d(y)+1,4))  -f(d(y),3)+f(d(y)+1,3).
\end{eqnarray}
By Proposition~\ref{pro-10}, the expression
$
-f(d(z),d(x))+f(d(z),d(x)-1)
$ increases in $d(z)$,
i.e.,
\begin{eqnarray*}
-f(d(z),d(x))+f(d(z),d(x)-1) &\le& \lim_{d(z) \rightarrow + \infty}(-f(d(z),d(x))+f(d(z),d(x)-1) ) \\
&=& -\sqrt{\frac{1}{d(x)}} + \sqrt{\frac{1}{d(x)-1}},
\end{eqnarray*}
while by Proposition~\ref{pro-20},
$
-f(d(z),d(y))+f(d(z),d(y)+1)
$
decreases in $d(z)$.
Thus,
\begin{eqnarray} \label{eq-pro-P3-10-00-1}
ABC (G') - ABC (G)  &\le&-\sqrt{\frac{1}{d(x)}} + \sqrt{\frac{1}{d(x)-1}} -f(d(z),d(y))+f(d(z),d(y)+1)  \nonumber \\
&&+(d(x)-2)(-f(d(x),4)+f(d(x)-1,4)) -f(d(x),4)+f(d(y)+1,4)  \nonumber \\
&& +(d(y)-2)(-f(d(y),4)+f(d(y)+1,4)) -f(d(y),3)+f(d(y)+1,3),
\end{eqnarray}
and it decreases in $d(z)$.

By Lemmas~\ref{le-Dz-lowerBound} and \ref{le-Dz-UpperBound}, we have that $16 \le d(x) \le 132$,
and by Proposition~\ref{pro-Dz-B2-10} that $20 \le d(y) \le 98$.
When $d(z) \ge 10000$ and $d(x) \ge d(y) + 6$, it can be verified
that the right-hand side of \eqref{eq-pro-P3-10-00-1} is negative, for any $16 \le d(x) \le 132$ and $20 \le d(y) \le 98$
(observe that since the right-hand side of \eqref{eq-pro-P3-10-00-1} decreases in $d(z)$, it suffices to verify the negativity of \eqref{eq-pro-P3-10-00-1} when
$d(z)=10000$).
As to $d(x) \le d(z) < 10000$ and $d(x) \ge d(y) + 6$, it can be verified that the right-hand side of \eqref{eq-pro-P3-10-00}
is always negative, from direct calculations, for any combination of $16 \le d(x) \le 132$ and $20 \le d(y) \le 98$.

The case when there are two $B_2$-branches attached to $y$ can be analogously considered as the case with one
$B_2$-branch attached to $y$, and therefore, to avoid the repetitions of same arguments, we omit the proof in this case.
%
%
\end{proof}

\begin{lemma}\label{lemma-NO-D_k,2_2}
A minimal-ABC tree with at least $65$ $D_z$-branches and with the root of degree at least $146$
does not contain a $D_{k,2}^{2}$-branch.
\end{lemma}

\begin{proof}
Assume that there are at least $65$ $D_z$-branches, the degree of the root vertex is at least $146$,
and the minimal-ABC tree does contain a $D_{k,2}^{2}$-branch.
By Lemmas~\ref{le-UpperBoundOnNumberOf_Dz-10} and \ref{le-LowerBoundOnNumberOf_Dz} (Tables \ref{table-large-z} and \ref{table-small-z}),
we may assume that $43 \leq z \leq 56$, and by Proposition~\ref{pro-Dz-B2-10} that $25 \leq k \leq 50$.
It holds also that $k \leq z$, and by Proposition~\ref{pro-P3-20}, $z \leq k +9$.
So we further assume that $34 \le k \le 50$.
Here we apply the transformation $\mathcal{T}_{11}$ depicted in Figure~\ref{fig-no-Dz,2_2-v2}.

\begin{figure}[!ht]
\begin{center}
\includegraphics[scale=0.75]{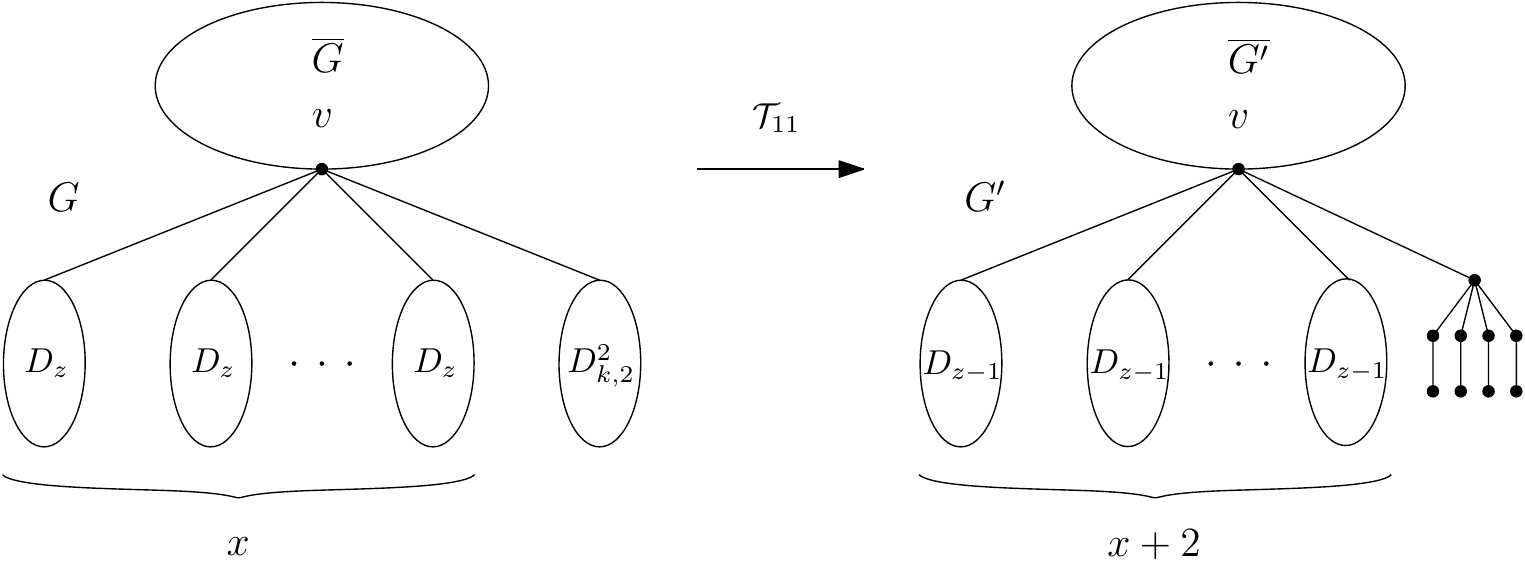}
\caption{The transformation which shows that under certain  conditions there is no $D_{k,2}^2$-branch.}
\label{fig-no-Dz,2_2-v2}
\end{center}
\end{figure}

The change of the ABC index after applying $\mathcal{T}_{11}$ is
\begin{eqnarray}  \label{eq-no-D_2,2^2-10}
%
ABC (G') - ABC (G)&=& \sum_{y v \in E(\bar{G})} ( - f(d(v), d(y)) + f (d(v) + 2, d(y))) -x \, f(d(v),z+1) -x \, z \, f(z+1,4) \nonumber \\
&& - f(d(v),k+1)   -(k-2) f(k+1,4) -2f(k+1,3) + (x+2) f(d(v)+2,z)   \nonumber \\
&& + (x+2)(z-1) f(z,4) +f(d(v)+2, 5).
\end{eqnarray}
The number of $B_3$-branches in $G$ (outside $\bar{G}$) is equal to the number of $B_3$-branches in $G'$ (outside {$\overline{G'}$}).
Thus, it follows that $x z + k -2= (x+2)(z-1)$, or equivalently, $x = 2z-k$.
It can be verified that
for each $z \in [43,56]$ with constraints $k \in [z-9, \min\{z, 50\}]$
and $x = 2z-k$ ,
the change of the ABC index (the right-hand side of (\ref{eq-no-D_2,2^2-10})) is always negative.
The technical details are presented in Appendix \ref{appendix-nonexist-B2-2}.
Notice that with the above given possible values of $z$ and $k$,
$x$ is at most $65$, which guarantees that the transformation $\mathcal{T}_{11}$ is feasible.
\end{proof}

\begin{lemma}\label{lemma-D_k,1_2}
A minimal-ABC tree with at least $261$ $D_z$-branches and with the root of degree at least $1228$
does not contain a $D_{k,1}^{2}$-branch.
\end{lemma}

\begin{proof}
Assume that there are at least $261$ $D_z$-branches, the root vertex has a degree of at least $1228$, and the minimal-ABC tree does contain a $D_{k,1}^{2}$-branch.
By Proposition \ref{pro-Dz-B2-10}, it holds that $19 \le k \le 97$.
Since there are (at least) $261$ $D_z$-branches and  the root of degree at least $1228$, by Lemmas~\ref{le-UpperBoundOnNumberOf_Dz-10}
and \ref{le-LowerBoundOnNumberOf_Dz}  (Tables \ref{table-large-z} and \ref{table-small-z}),
it follows that $z \in \{50, 51, 52 \}$.
By Proposition~\ref{pro-P3-20}, $z \leq k +5$. Also it holds that $k \leq z$.

Assume that $z = 50$. Clearly, it remains to argue that the claim  is true for $1228 \leq d(v) < 1358$. Otherwise,
$n_{50} \le 182$ from Table \ref{table-small-z}, which is a contradiction to the hypothesis that $n_{50} \ge 261$.
Recall that in the proof of Appendix \ref{appendix-exist-B4}, we have shown that $n_{50} \ge 222$
is impossible when $1228 \le d(v) < 1358$ by using (\ref{eq-Lemma-Dz-30-20-second})
(here the only difference is $n_4 = 0$ instead of $n_4 = 1$ in Appendix \ref{appendix-exist-B4},
 but it does not affect the negativity of the change of ABC index). It implies that $n_{50} < 222$, which is a contradiction to the hypothesis that $n_{50} \ge 261$.

For $z \in \{51, 52 \}$, we apply the transformation $\mathcal{T}_{12}$ depicted in Figure~\ref{fig-no-Dz,1_2}.
\begin{figure}[!ht]
\begin{center}
\includegraphics[scale=0.75]{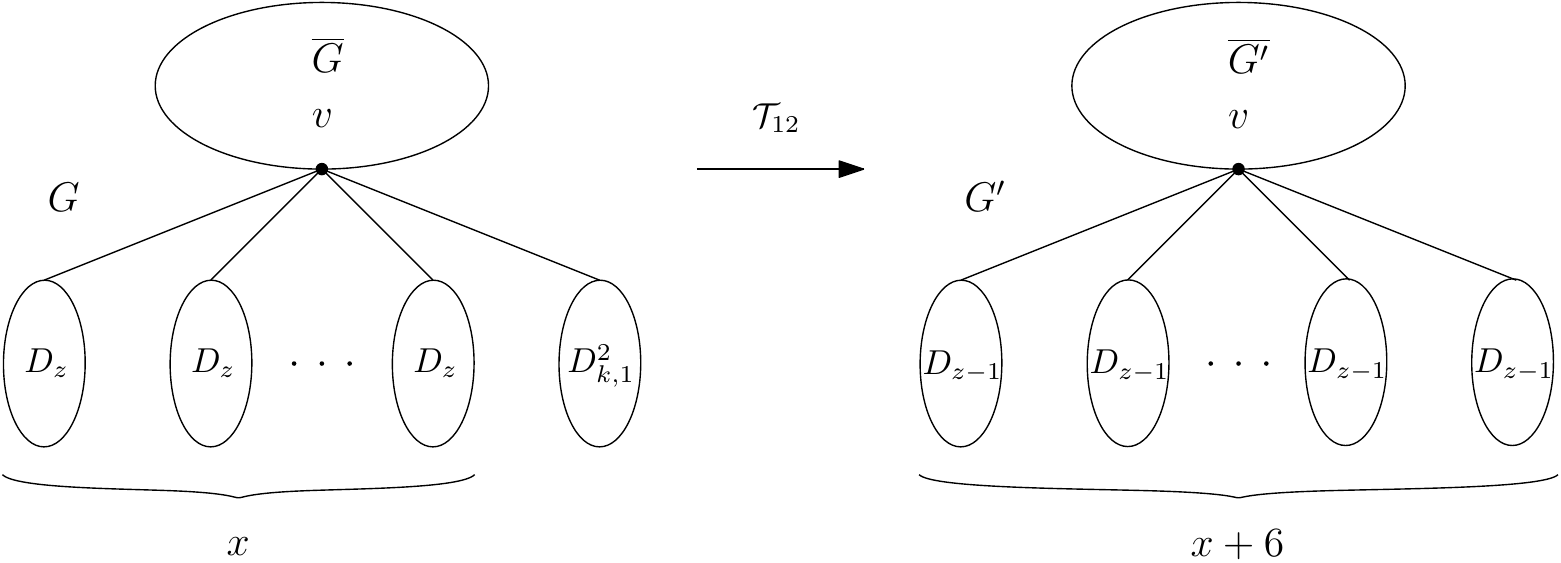}
\caption{The transformation which shows that under certain  conditions there is no $D_{k,1}^2$-branch.}
\label{fig-no-Dz,1_2}
\end{center}
\end{figure}
The change of the ABC index after applying $\mathcal{T}_{12}$ is
\begin{eqnarray}  \label{eq-no-Dz,1_2-10}
%
ABC (G') - ABC (G) &=& \sum_{y v \in E(\bar{G})} ( - f(d(v), d(y)) + f (d(v) + 5, d(y))) -x \, f(d(v),z+1) -x \, z \, f(z+1,4)  \nonumber \\
&& - f(d(v),k+1) -(k-1) f(k+1,4) -f(k+1,3) - 2 f(3,2) - 2 f(2,1)  \nonumber \\
&&  + (x+6) f(d(v)+5,z)  + (x+6)(z-1) f(z,4).
\end{eqnarray}
The number of $B_3$-branches in $G$
 is equal to the number of $B_3$-branches in $G'$ (the ones outside $\overline{G}$ ({$\overline{G'}$})),
i.e., $x z + k -1= (x+6)(z-1)$,  or equivalently, $x = 6z-k -5$.
From $z \in \{51, 52 \}$, $k \in [z-5,z]$
and $x = 6z - k - 5$, it follows that $x \leq 260$ and thus, the transformation $\mathcal{T}_{12}$ is feasible.
A verification shows that the right-hand side of (\ref{eq-no-Dz,1_2-10}) is negative under the given parameters.
The details are presented in Appendix \ref{appendix-nonexist-B2-3}.
\end{proof}

\begin{re}
The above upper bound $261$ of $D_z$-branches can be improved by involving particular structural properties
of $D_{k,1}^2$-branches, but the current upper bound suffices for the proofs of the main and computational results in Section~\ref{main-results}.
\end{re}

\subsection[$D_{z}^{**}$-branch]
{$D_{z}^{**}$-branch}

\begin{pro} \label{le-Dz**-UpperBound-10}
If a minimal-ABC tree contains a $D_z^{**}$-branch, then $47 \leq z \leq 74$.
\end{pro}
\begin{proof}
Assume that there is a $D_z^{**}$-branch in a minimal-ABC tree, with $z \geq 75$. Recall that $z \le 131$ from Lemma \ref{le-Dz-UpperBound-star}.
Assume, also that $z$ is odd (the case when $z$ is even can be handled analogously).
Then apply the transformation $\mathcal{T}_{13}$ depicted in Figure~\ref{fig-Dz**_size10-2}.
\begin{figure}[!ht]
\begin{center}
\includegraphics[scale=0.75]{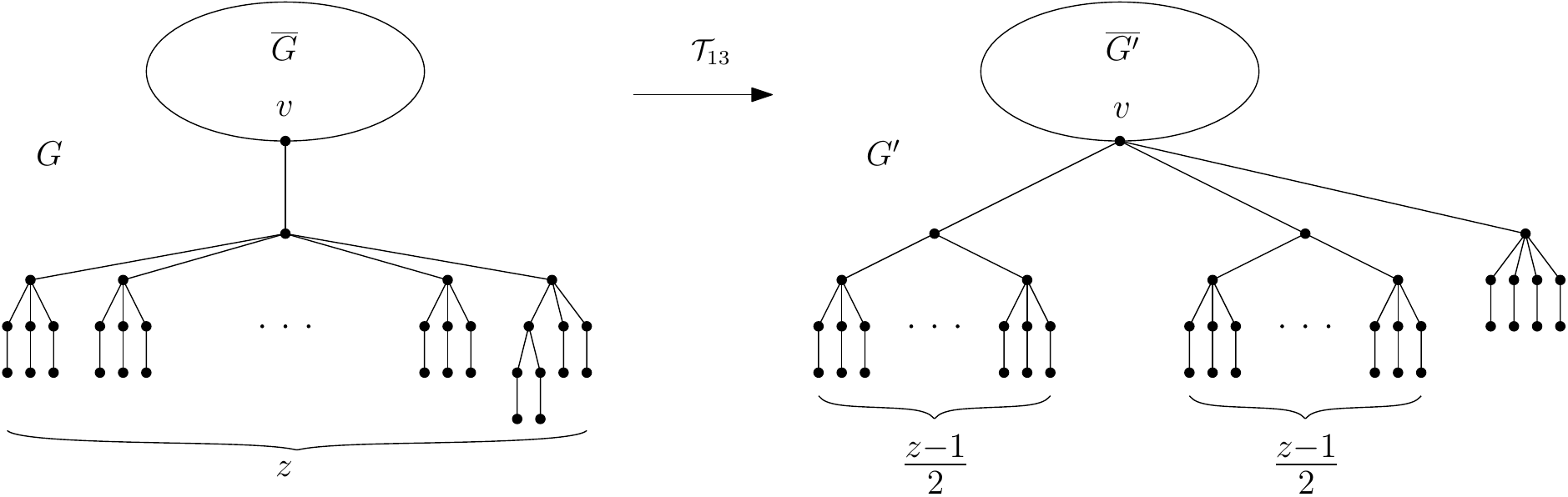}
\caption{The transformation applied to obtain an upper bound on the size of a $D_{z}^{**}$-branch.}
\label{fig-Dz**_size10-2}
\end{center}
\end{figure}
After applying $\mathcal{T}_{13}$, the change of the ABC index is
\begin{eqnarray} \label{eq-Lemma-Dz**-10}
&& ABC (G') - ABC (G) \nonumber \\
&=& \sum_{x v \in E(\bar{G})} ( - f(d(v), d(x)) + f (d(v) + 2, d(x))) - f (d(v),z+1) + 2 f \left( d(v) + 2, \frac{z+1}{2} \right)  \nonumber \\
&& - z \, f (z+1,4) + (z-1) f \left( \frac{z+1}{2}, 4 \right) - f(4,3) + f (v+2,5).
\end{eqnarray}
In Appendix~\ref{appendix-upper-D**}, we show that the right-hand side of \eqref{eq-Lemma-Dz**-10} is negative for $75 \le z \le 131$.

For the lower bound on the size $z$ of a $D_z^{**}$-branch, consider the transformation $\mathcal{T}_{14}$ depicted in Figure~\ref{fig-Dz**_size20}.
\begin{figure}[!ht]
\begin{center}
\includegraphics[scale=0.75]{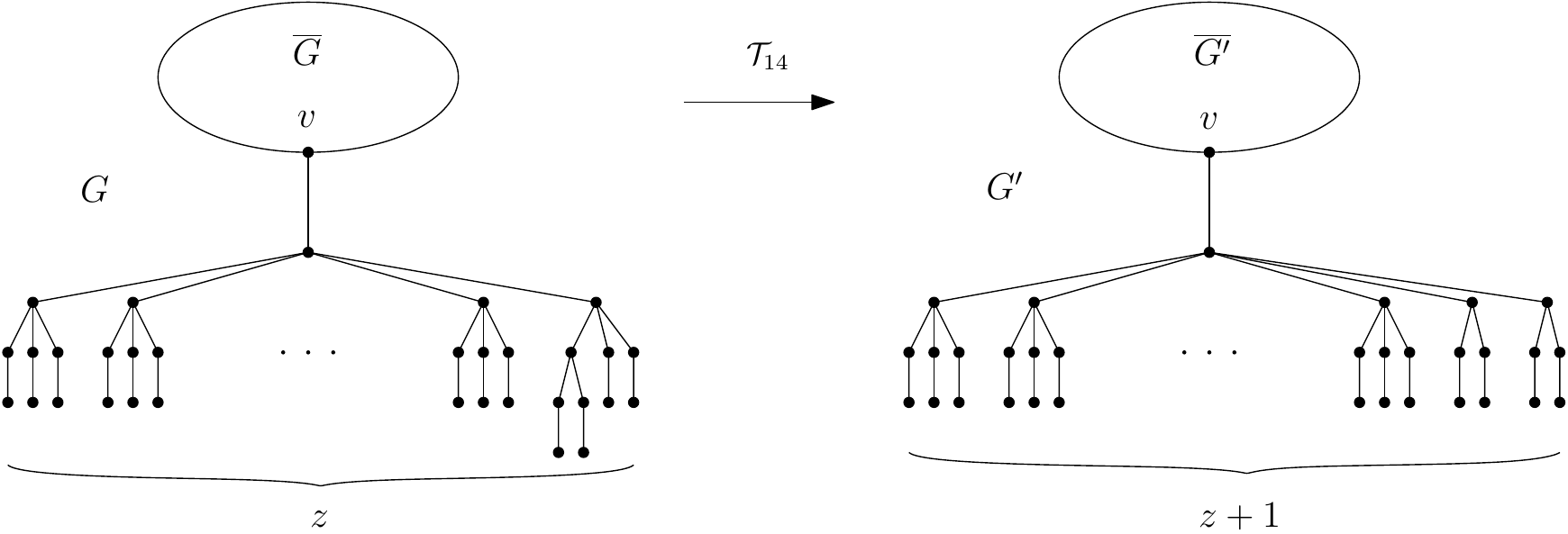}
\caption{The transformation applied to obtain a lower bound on the size of a $D_{z}^{**}$-branch.}
\label{fig-Dz**_size20}
\end{center}
\end{figure}
The change of the ABC index after applying $\mathcal{T}_{14}$ is
\begin{eqnarray*}
%
ABC (G') - ABC (G)  &=& -f(d(v),z+1)+f(d(v),z+2)   + (z-1) (-f(z+1,4)+f(z+2,4) )   \\
&&  -f(z+1,4) - f(4,3) + 2 f(z+2,3).
\end{eqnarray*}
By Proposition~\ref{pro-20}, $-f(d(v),z+1)+f(d(v),z+2)$ decreases in $d(v) \ge z+1$. Thus we have the following upper bound on $ABC (G') - ABC (G)$:
\begin{eqnarray*}
%
ABC (G') - ABC (G) &\le& -f(z+1,z+1)+f(z+1,z+2)   + (z-1) (-f(z+1,4)+f(z+2,4) )  \\
&&  -f(z+1,4) - f(4,3) + 2 f(z+2,3),
\end{eqnarray*}
which is negative for $15 \le z \leq 46$.
\end{proof}

\begin{lemma} \label{lemma-no-D_B3**-branches-v2}
If there are at least $56$ $D_z$-branches, then the minimal-ABC tree does not contain a $D_{k}^{**}$-branch.
\end{lemma}
\begin{proof}
Assume that there are at least $56$ $D_z$-branches and the minimal-ABC tree does contain a $D_{k}^{**}$-branch.
By Lemma~\ref{le-Dz**-UpperBound-10}, it holds that $47 \le k \le 74$.
Since there are at least $56$ $D_z$-branches, by Lemma~\ref{le-UpperBoundOnNumberOf_Dz-10}
(see the results in Table \ref{table-large-z})  and  Lemma~\ref{le-LowerBoundOnNumberOf_Dz}
(see Table \ref{table-small-z}),
it follows that $38 \leq z \leq 57$.
By Proposition~\ref{pro-diff-Dz}, $k=z-1,$ $z$ or $z+1$. Thus, we may assume that $47 \le k \le 58$ and $46 \le z \le 57$.

Next, we apply the transformation $\mathcal{T}_{15}$ depicted in Figure~\ref{fig-no-Dz_**-v3}.
\begin{figure}[!ht]
\begin{center}
\includegraphics[scale=0.75]{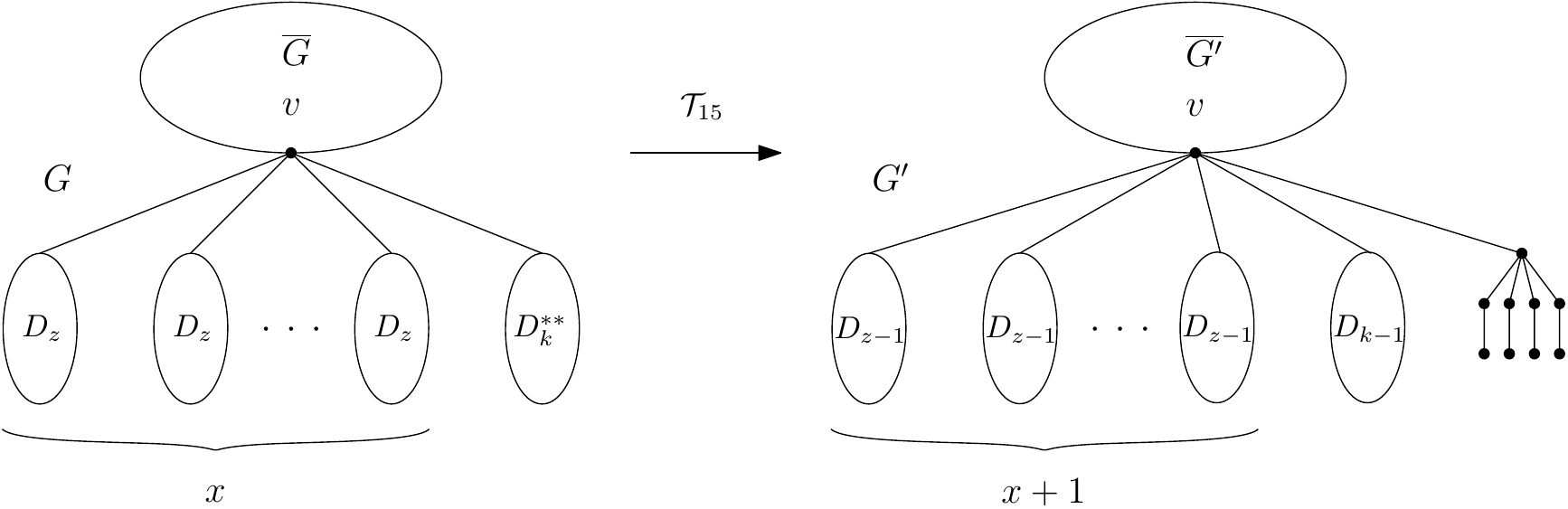}
\caption{The transformation which shows that under certain  conditions there is no $D_{z}^{**}$-branch.}
\label{fig-no-Dz_**-v3}
\end{center}
\end{figure}
The change of the ABC index after applying $\mathcal{T}_{15}$ is
\begin{eqnarray}  \label{eq-no-Dz_**-v2}
ABC (G') - ABC (G) &=& \sum_{y v \in E(\bar{G})} ( - f(d(v), d(y)) + f (d(v) + 2, d(y))) -x \, f(d(v),z+1) -x \,  z \,  f(z+1,4) \nonumber \\
&&  - f(d(v),k+1) - k \, f(k+1,4) - f(4,3) + (x+1) f(d(v)+2,z)  \nonumber \\
&&   + (x+1)(z-1) f(z,4) +f(d(v)+2,k)  + (k-1) f(k,4)  +f(d(v)+2, 5).
\end{eqnarray}
The number of $B_3$-branches in $G$ (outside $\bar{G}$) is equal to the number of $B_3$-branches in $G'$ (outside {$\overline{G'}$}),
i.e., $x \, z  + k - 1= (x+1)(z-1) + k - 1$, that is, $x = z-1$.
From $z \in [46,57]$, $k \in [\max\{47,z - 1\},z+1]$
and $x = z-1$, it follows that $x \le 56$, and therefore, the transformation $\mathcal{T}_{15}$ is feasible.
For these values of $z$, $k$, and $x$, it can be shown that the right-hand side of (\ref{eq-no-Dz_**-v2}) is negative.
The details are presented in Appendix \ref{appendix-nonexist-B3**}.
\end{proof}

\section[Main results]{Main results}\label{main-results}

By Proposition~\ref{pro-diff-Dz} and
Lemmas~\ref{lemma-no-B4-branches}, \ref{lemma-NO-D_k,2_2}, \ref{lemma-D_k,1_2}
and \ref{lemma-no-D_B3**-branches-v2} we obtain the
following result, which is an affirmative answer to Conjecture~\ref{conjecture2}.

\begin{te} \label{co-withB3_to_root}
A minimal-ABC tree with the root vertex of degree at least $1228$ is comprised only of a root and $D_z$- and $D_{z+1}$-branches,
where $z \in \{50, 51, 52\}$
and maybe of $B_3$-branches and at most one of $D_{t,1}^2$-branch, $z - 5 \le t \le z$, attached to the root.
\end{te}
\begin{proof}
In general, if the degree of the root vertex is at least $1228$, besides the $B$-branches, there are at least
$1228-919=309$ $D$-branches (since there are at most $919$ $B$-branches attached to the root vertex
from Lemma \ref{thm-bound-B-branches-to-root}), and thus,  by Lemmas~\ref{lemma-NO-D_k,2_2}
and \ref{lemma-no-D_B3**-branches-v2}, it follows that
there is no $D_{z,2}^2$-branch and no $D_{z}^{**}$-branch.
In addition, by Lemma~\ref{lemma-no-B4-branches}, when $d(v) \geq 1228$ there are no $B_4$-branches. Observe that a branch with one $B_2$-branch, say $D_{t,1}^2$-branch for some $t$ ($z - 5 \le t \le z$ from Proposition~\ref{pro-P3-20}), may exist, there is at most one of such branch with $B_2$-branch.

As in the proof of Lemma \ref{lemma-no-B4-branches}, we have shown that when the root vertex has degree at least $1228$,
the cases when a minimal-ABC tree contains a combination of $D_{z}$- and $D_{z+1}$-branches (maybe with a possible $D_{t,1}^2$-branch), for $15 \leq z \leq 48$ and $54 \leq z \leq 131$,
are not possible. So we are remaining to eliminate the two cases:
the combination of $D_{49}$- and $D_{50}$-branches, and the combination of $D_{53}$- and $D_{54}$-branches.

The same as in Appendix \ref{appendix-exist-B4}, with the help of (\ref{eq-Lemma-Dz-30-20-second}),
the case that $G$ contains a combination of $D_{49}$- and $D_{50}$-branches (whether the $D_{t,1}^2$-branch exists or not) is also impossible (here $n_4 = 0$ and the worst case is when $n_3$ is also equal to $0$).

Now, consider the last remaining case when $G$ contains a combination $D_{53}$- and $D_{54}$-branches, maybe with a possible $D_{t,1}^2$-branch.
Here, the minimal-ABC tree may contain at most $882$ $B_3$-branches (see Tables~\ref{tbl-maxB-2} and \ref{tbl-maxB-2-1}), since no $B_4$-branch occurs when $d(v) \geq 1228$ from Lemma~\ref{lemma-no-B4-branches}.

\smallskip
\noindent
{\em Case 1: There is no $D_{t,1}^2$-branch for any $t$}.

There can be in addition $260$ $D_{53}$-branches and $130$ $D_{54}$-branches (see Table~\ref{table-large-z}).
Since there are at most $882$ $B_3$-branches, the number of $D_{53}$- and $D_{54}$-branches must be at least
$1228 - 882 = 346$.
Here, we can apply the transformation $\mathcal{T}_{4}$ and obtain a new upper bound on the change of the ABC index with $z = 53$,
which is an improvement of the upper bound (\ref{eq-Lemma-Dz-30-00-1}) presented in Appendix~\ref{appendix-upper_bound}:
\begin{eqnarray}  \label{eq-Lemma-Dz-30-00-1-v2}
ABC (G') - ABC (G)  &\le&-x \, f(d(v),z+1) -x \, z \, f(z+1,4) - 3 f(4,2) - 3 f(2,1)  + n_{k}  f(d(v)+7,k+1)   \nonumber \\
&&  + n_{k-1}  f(d(v)+7,k)   + k \, n_{k}  f(k+1,4)  +  (k-1) n_{k-1}  f(k,4)  \nonumber \\
&& + n_3 ( - f(d(v),4) + f(d(v)+7,4))   \nonumber \\
&& + (d(v) - x - n_3) ( - f(d(v),z+1) + f(d(v)+7,z+1)).  \nonumber
\end{eqnarray}
Clearly, the worst case  is when $n_3$ as large as possible, i.e., $n_3 = 882$.
Recall that $n_{53} + n_{54} \ge 346$ and $n_{54} \le 130$, thus $n_{53} \ge 346 - 130 = 216$. So by setting $z = 53$, $x = 216$, $k = 52$ ($n_{52} = 74$ and $n_{51} = 149$), $n_3 = 882$, the right-hand side of last inequality above is negative for $1228 \leq d(v) \leq 882 + 260 + 130 = 1272$.

\smallskip
\noindent
{\em Case 2: There is $D_{t,1}^2$-branch for some $t$}.

In this case, the upper bound on the change of ABC index after applying transformation $\mathcal{T}_{4}$ with $z = 53$ is:
\begin{eqnarray}  \label{eq-Lemma-Dz-30-00-1-v2}
ABC (G') - ABC (G) &\le&-x \, f(d(v),z+1) -x \, z \, f(z+1,4) - 3 f(4,2) - 3 f(2,1)  + n_{k}  f(d(v)+7,k+1)   \nonumber \\
&&  + n_{k-1}  f(d(v)+7,k)   + k \, n_{k}  f(k+1,4)  +  (k-1) n_{k-1}  f(k,4) \nonumber \\
&& + n_3 ( - f(d(v),4) + f(d(v)+7,4))  + ( - f(d(v),t+1) + f(d(v)+7,t+1)) \nonumber \\
&& + (d(v) - x - n_3 - 1) ( - f(d(v),z+1) + f(d(v)+7,z+1)).  \nonumber
\end{eqnarray}

When $d(v) \ge 1229$, $n_{53} + n_{54} \ge 1229 - 882 - 1 = 346$ (we minus one here because of the occurrence of $D_{t,1}^2$-branch), and thus $n_{53} \ge 346 - 130 = 216$.
As above, by setting $z = 53$, $x = 216$, $k = 52$ ($n_{52} = 74$ and $n_{51} = 149$), $n_3 = 882$, $t = 48$ (because $t \ge 53 - 5 = 48$ from Proposition~\ref{pro-P3-20}), the right-hand side of last inequality above is negative for $1229 \leq d(v) \leq 882 + 260 + 130 + 1= 1273$.

When $d(v) = 1228$, $n_{53} + n_{54} \ge 345$. In order to get $n_{53} \ge 216$ as the case $d(v) \ge 1229$, we resort to the right-hand side of (\ref{eq-Lemma-Dz-30-00-1}), by setting $d(v) = 1228$, $z = 54$, $x = 130$, $k = 52$ ($n_{51} = 105$ and $n_{52} = 32$), leading to a negative value for the right-hand side of (\ref{eq-Lemma-Dz-30-00-1}). It implies that $n_{54} \le 129$ when $d(v) = 1228$, so $n_{53} \ge 345 - 129 = 216$ follows. The remaining parts are the same as the case $d(v) \ge 1229$.

The proof is completed.
\end{proof}

If we want to eliminate the existence of $D_{t,1}^2$-branch in Theorem \ref{co-withB3_to_root}, Lemma \ref{lemma-D_k,1_2} could help us a lot, when the root vertex has a somewhat larger degree than $1228$.

\begin{co} \label{co-conj2_and_more-add}
A minimal-ABC tree with the root vertex of degree at least $1441$ is comprised only of a root and $D_z$- and $D_{z+1}$-branches,
where $z \in \{50, 51, 52\}$
and maybe of $B_3$-branches attached to the root.
\end{co}

\begin{proof}
From Lemma \ref{lemma-D_k,1_2}, we need at least $261$ $D_z$-branches to guarantee the nonexistence of $D_{t,1}^2$-branch. When $d(v) \ge 1441$, besides the $B$-branches and one possible $D$-branch with $B_2$-branch, there are at least
$1441-919-1=521$ $D$-branches (recall that there are at most $919$ $B$-branches attached to the root vertex
from Lemma \ref{thm-bound-B-branches-to-root}). Assume that such $D$-branches are $D_z$ or $D_{z+1}$-branches. Clearly, one of $n_{z}$ or $n_{z+1}$ is at least $\lceil \frac{521}{2} \rceil = 261$.
\end{proof}

By Lemma~\ref{lemma-no-B-branches-10}, when the degree of the root is at least $2956$,
then there are also no $B_3$-branches attached to the root.

\begin{co} \label{co-conj2_and_more}
A minimal-ABC tree with the root vertex of degree at least $2956$ is comprised only of a root and $D_z$- and $D_{z+1}$-branches,
where $z \in \{50, 51, 52\}$.
\end{co}

This confirmed the conjectured structure of the minimal-ABC tree from \cite{adgh-dctmabci-14}  depicted in Figure~\ref{fig-dd-4}.
More details about the structure of minimal-ABC trees,
 when the degree of the root is at least $2956$,
are presented in the following theorem.

\begin{te} \label{te-combinations-D-branches}
Let $v$ be the root of a minimal-ABC tree $G$.
\begin{itemize}
\item
If   $2956 \leq d(v) \leq 3241$, then $G$ must contain $D_{51}$-branches and in addition
\begin{itemize}
\item   either $D_{52}$-branches,
\item or at most $123$ $D_{50}$-branches.
\end{itemize}
In particular, when $2956 \le d(v) \le 3185$, $G$ contains at most $357$ $D_{52}$-branches.
\item
If  $d(v) \ge 3242$, then $G$ must contain $D_{52}$-branches and
\begin{itemize}
\item  either at most $718$ $D_{51}$-branches.
\item  or at most $178$ $D_{53}$-branches.
\end{itemize}
In particular, when $d(v) \geq 3249$,  $G$ contains at most $364$ $D_{51}$-branches.
\end{itemize}
\end{te}

\begin{proof}
We partition our proof into two parts with respect to $d(v)$,
namely the cases when $2956 \leq d(v) \leq 3241$  and $d(v) \ge 3242$.

\smallskip
\noindent
{\em Case 1: $2956 \leq d(v) \leq 3241$}.

We claim that $G$ must contain $D_{51}$-branches in this case. Otherwise, from Corollary \ref{co-conj2_and_more}, $G$ contains
a combination of $D_{52}$- and $D_{53}$-branches, or $G$ contains only $D_{50}$-branches. The latter is clearly impossible, recall from Table \ref{table-small-z}, $n_{50} \le 182$, and thus $d(v) = n_{50} \le 182$ (since there is no $B$-branch attached to $v$ from Lemma \ref{lemma-no-B-branches-10}).

Assume that $G$ contains a combination of $D_{52}$- and $D_{53}$-branches. By virtue of the transformation $\mathcal{T}_{4}$ depicted in Figure~{\ref{fig-Dz_size30}}, we can show that $n_{52} \le 357$. Observe that now there is no $B$-branch attached to the root, instead of the
term $(d(v)-x) ( - f(d(v),4) + f(d(v)+7,4) )$  in (\ref{eq-Lemma-Dz-30-00-1}), we have the term
$(d(v)-x) ( - f(d(v),z + 1) + f(d(v)+7,z + 1) )$ with $z =52$. We can rewrite (\ref{eq-Lemma-Dz-30-00-1}) as:
\begin{eqnarray} \label{eq-Lemma-Dz-30-00-v2}
%
ABC (G') - ABC (G)  &\le& -x f(d(v),z+1) -x \, z \, f(z+1,4) - 3 f(4,2) - 3 f(2,1)   \nonumber \\
&&  + n_{k}  f(d(v)+7,k+1)  + n_{k-1}  f(d(v)+7,k) + k n_{k}  f(k+1,4)   \nonumber \\
&&  +  (k-1) n_{k-1}  f(k,4) + (d(v)-x) ( - f(d(v),z + 1) + f(d(v)+7,z + 1)). \quad
\end{eqnarray}
A straightforward verification shows that for $z=52$, $x = 358$, $k = 52$ (thus $n_{51}=365$ and $n_{52}=0$),
the expression (\ref{eq-Lemma-Dz-30-00-v2}) is negative for $2956 \leq d(v) \leq 3241$. Together with $n_{53} \le 260$ from Table \ref{table-large-z}, we obtain $d(v) = n_{52} + n_{53} \le 357 + 260 = 617 < 2956$, a contradiction.

Besides the $D_{51}$-branches, $G$ can only contain $D_{50}$- or $D_{52}$-branches. Moreover, $n_{50} \le 123$ by using the transformation $\mathcal{T}_{5}$, more precisely, the right-hand side of (\ref{eq-Lemma-Dz-30-00-second}) is negative when $2956 \leq d(v) \leq 3241$, $z = 50$, $x = 124$, $k = 53$ ($n_{53} = 117$ and $n_{52} = 0$). And $n_{52} \le 357$ when $2956 \leq d(v) \leq 3185$ from the transformation $\mathcal{T}_{4}$ again, here the difference from (\ref{eq-Lemma-Dz-30-00-v2}) is that we need to replace the term $(d(v)-x) ( - f(d(v),z + 1) + f(d(v)+7,z + 1) )$ by $(d(v)-x) ( - f(d(v),z) + f(d(v)+7,z) )$ with $z = 52$, i.e.,
\begin{eqnarray}  \label{eq-Lemma-Dz-30-00-v2-1}
%
ABC (G') - ABC (G) &\le& -x \, f(d(v),z+1) -x \, z \, f(z+1,4) - 3 f(4,2) - 3 f(2,1)   \nonumber \\
&&  + n_{k}  f(d(v)+7,k+1)  + n_{k-1}  f(d(v)+7,k) + k n_{k}  f(k+1,4)   \nonumber \\
&&  +  (k-1) n_{k-1}  f(k,4) + (d(v)-x) ( - f(d(v),z) + f(d(v)+7,z) ).
\end{eqnarray}
By the same setting that $z=52$, $x = 358$, $k = 52$ (thus $n_{51}=365$ and $n_{52}=0$) as above, the right-hand side of (\ref{eq-Lemma-Dz-30-00-v2-1}) is negative for $2956 \leq d(v) \leq 3185$.

\smallskip
\noindent
{\em Case 2: $d(v) \ge 3242$}.

From Corollary  \ref{co-conj2_and_more}, the possible $D$-branches occurring in $G$ are $D_{50}$-, $D_{51}$-, $D_{52}$-, $D_{53}$-branches.
From Tables \ref{table-large-z} and \ref{table-small-z}, $n_{50} \le 182$, $n_{53} \le 260$. By using the transformation $\mathcal{T}_{5}$, we have $n_{51} \le 718$, which is obtained from the fact that the right-hand side of (\ref{eq-Lemma-Dz-30-00-second}) is negative when $d(v) \geq 3242$, $z = 51$, $x = 719$, $k = 52$ ($n_{52} = 358$ and $n_{51} = 354$). It is worth mentioning that the upper bound of $n_{51}$ can be improved into $n_{51} \le 364$ when $d(v) \ge 3249$, from Table \ref{table-small-z}. These upper bounds imply that $d(v) < 2956$ unless $D_{52}$-branches occur.

In addition to $D_{52}$-branches, $G$ can contain either $D_{51}$- or $D_{53}$-branches. We have known that $n_{51} \le 718$. As to $n_{53}$, Table \ref{table-large-z} reflects that $n_{53} \le 260$, actually when we involve in the hypothesis that $d(v) \geq 3242$, the upper bound would be improved into $n_{53} \le 178$ by using the transformation $\mathcal{T}_{4}$, that is to say, the right-hand side of (\ref{eq-Lemma-Dz-30-00-v2-1}) is negative when $d(v) \ge 3242$, $z = 53$, $x = 179$, $k = 52$ ($n_{52} = 0$ and $n_{51} = 186$). The detailed deduction about the negativity of right-hand side of (\ref{eq-Lemma-Dz-30-00-v2-1}) when $z = 53$ is similar to that in Appendix \ref{appendix-upper_bound}. In a rough aspect, the right-hand side of (\ref{eq-Lemma-Dz-30-00-2}) in Appendix \ref{appendix-upper_bound} is negative when $d(v) = 700000$, so is $d(v) \ge 700000$ (since it decreases in $d(v)$), and for $3242 \le d(v) < 700000$, the right-hand side of (\ref{eq-Lemma-Dz-30-00-v2-1}) is negative from direct calculations.
\end{proof}

\subsection{The optimization function and  \\determining the parameters of the minimal-ABC trees}\label{section-compuation}

The obtained theoretical results have significantly reduced the search space. By applying them the
parameters that fully characterize the  minimal-ABC trees  can be efficiently computed.
Denote by $\mathcal{T}$ all trees of a given order $n$ that are comprised of:
a combination of $D_z$- and $D_{z+1}$-branches, $15 \leq z \leq 131$; at most one $D_{z}^{**}$-, $D_{k,1}^2$-  or $D_{k,2}^2$-branch;
at most $4$ $B_4$-branches and at most $919$ $B$-branches attached to the root vertex.
{\footnote{Notice that we may assume that the $D^{**}$-branch is of size $z$.
Otherwise, we can apply the switching transformation, switching one $B_3$-branch from a $D_z$-branch (if it exist) with the $B_3^{**}$-branch.
If the minimal ABC tree does not contain $D_z$-branches, but only $D_{z+1}$-branches, then this case is considered when we increase $z$ by one.}

A minimal-ABC tree of order at least $1000$ is one of the trees from $\mathcal{T}$
(the minimal-ABC trees of orders at most $1100$ were reported in \cite{lcwdh-csltmabci-18}).
Let $T \in \mathcal{T}$.
Then
\begin{flalign}  \label{eq-n-minABC-00}
&ABC (T) = n_z f(d(v), z+1) +  n_{z+1} f(d(v), z+2) +  n_z z f(z+1,4) +  n_{z+1} (z+1) f(z+2,4)     \nonumber \\
  &  + 6 n_z z \frac{\sqrt{2}}{2} + 6 n_{z+1} (z+1)  \frac{\sqrt{2}}{2}  + b_3\left( n_3 f(d(v),4)  + 6 n_3  \frac{\sqrt{2}}{2} \right)  + b_4 \left(n_4f(d(v),5) + 8 n_4  \frac{\sqrt{2}}{2} \right)  \nonumber \\ 
 &+ b_1 \left(  f(d(v), k_1+1)  + (k_1-1)f(k_1+1,4)  + f(k_1+1,3)  + 6(k_1-1)\frac{\sqrt{2}}{2} +4 \frac{\sqrt{2}}{2} \right)  \nonumber \\
 &+  b_2 \left(  f(d(v), k_2+1)  + (k_2-2)f(k_2+1,4)  + 2 f(k_2+1,3) + 6(k_2-2)\frac{\sqrt{2}}{2} +2 \cdot 4 \cdot  \frac{\sqrt{2}}{2} \right)  \nonumber \\
 &+ b_{*} \left(f(d(v), z+1) + z f(z+1,4)+(6 z +2)\frac{\sqrt{2}}{2} + f(4,3) \right),   \nonumber
\end{flalign}
where $n_z$ and $n_{z+1}$ are, respectively, the numbers of  $D_z$- and $D_{z+1}$-branches,
$n_3$ and $n_4$ are, respectively, the numbers of $B_3$- and $B_4$-branches attached to the root,
$k_1$ is the size of the $D_{k,1}^2$-branch ($(k-1)$ $B_3$-branches and one $B_2$-branch, which comprise a $D_{k,1}^2$-branch),
$k_2$ is the size of the $D_{k,2}^2$-branch ($(k-2)$ $B_3$-branches and two $B_2$-branches, which comprise a $D_{k,2}^2$-branch),
and parameters $b_*, b_1, b_2, b_3, b_4 \in \{ 0, 1\}$ are related to the existences of  $D_{z}^{**}$-, $D_{k,1}^2$-, $D_{k,2}^2$-,
$B_3$- and  $B_4$-branches attached to the root vertex, respectively - for example $b_*=0/1$ means that $T$ {\it does not contain/ does contain} $D_{z}^{**}$-branch.

The parameters of a minimal-ABC tree of order $n$ can be determined by solving the following optimization problem:

\smallskip
\noindent
\begin{eqnarray}  \label{eq-n-minABC-10}
\hspace*{-10cm}  min_{T \in \mathcal{T}} ABC(T)  \nonumber
\end{eqnarray}

\noindent
subject to
\begin{itemize}
\item[(1)] $d(v) = n_z + n_{z+1} + b_* + b_1 + b_2 + b_3 \cdot n_3 + b_4 \cdot n_4$
\item[(2)] $n = 1+ n_z (7 z +1) + n_{z+1} (7(z+1)+1)  + b_1(7(k_1-1) + 5 +1)  + b_2 (7(k_2-2) + 10 +1) + b_3 \cdot n_3 \cdot 7 + b_4 \cdot n_4 \cdot 9 +b_*\cdot (7z+4)$
\item[(3)] $15 \leq z \leq 131$      \hspace*{\fill} (by Lemmas~\ref{le-Dz-lowerBound} and  \ref{le-Dz-UpperBound})
\item[(4)] $z-5 \leq k_1 \leq z$     \hspace*{\fill} (by Proposition \ref{pro-P3-20})
\item[(5)] $z-9 \leq k_2 \leq z$   \hspace*{\fill} (by Proposition \ref{pro-P3-20})
\item[(6)] $0 \leq n_4 \leq 4$      \hspace*{\fill} (by Theorem \ref{thm-20})
\item[(7)] $0 \leq n_3 + n_4 \leq 919$  \hspace*{\fill} (by Lemma \ref{thm-bound-B-branches-to-root})
\item[(8)] $b_*, b_1, b_2, b_3, b_4 \in \{ 0, 1\}$
\item[(9)] at most one of $b_*, b_1, b_2, b_4$ can be $1$  \hspace*{\fill} (by Theorems \ref{noB1B4} and \ref{noB2B4}).
\end{itemize}

\bigskip

With the growth of the order of the trees, the above optimization problem becomes simpler.

When $d(v) \geq 1228$, by Theorem~\ref{co-withB3_to_root}, it follows that $n_4=0$, $b_* = b_1 = b_2 = b_4 =0$, and $50 \leq z \leq 52$.
Further, when $d(v) \geq 2956$, by Corollary~\ref{co-conj2_and_more}, it follows that $n_3 = n_4=0$, $b_* = b_1 = b_2 = b_3= b_4 =0$ and $50 \leq z \leq 52$.
Moreover, by Theorem~\ref{te-combinations-D-branches}, when $d(v) \geq 3249$, we have $0 \leq n_{51} \leq 364$ and $0 \leq n_{53} \leq 178$.
That allows us to compute the parameters of the minimal-ABC tree for an arbitrary $n$ in constant time.

We have implemented the solution to the above optimization problem in C++.
The minimal-ABC trees and their structures are given in Appendix  \ref{appendix-figures}.
The minimal-ABC trees of orders up to $1100$ were already known \cite{lcwdh-csltmabci-18}.

The computational results lead to the following observations. All bounds are sharp and are attained by some minimal-ABC tree.

\begin{observation}
A minimal-ABC tree of order larger than $1072$ can contain at most  two $B_4$-branches.
\end{observation}

\begin{observation}
A minimal-ABC tree of order larger than $8507$ can contain at most one $B_4$-branch.
\end{observation}

\begin{observation} \label{observation-3}
A minimal-ABC tree of order larger than $19040$ does not contain $B_4$-branches.
A minimal-ABC tree whose root vertex is of degree at least $368$ does not contain $B_4$-branches attached to the root vertex.
\end{observation}

\begin{observation} \label{observation-4}
A minimal-ABC tree of order larger than $1017676$ does not contain $B_3$-branches attached to the root vertex.
A minimal-ABC tree whose root vertex is of degree at least $2838$ does not contain $B_3$-branches attached to the root vertex.
\end{observation}

\begin{observation} \label{observation-5}
A minimal-ABC tree may have at most $738$ $B_3$-branches attached to the root vertex.
\end{observation}

\begin{observation}
A minimal-ABC tree of order larger than $999$ does not contain a $B_3^{**}$-branch,
and thus any such minimal-ABC tree cannot have a $D_{k}^{**}$-branch.
\end{observation}

\begin{observation}
A minimal-ABC tree of order larger than $305$ does not contain more than one $B_2$-branch,
and thus any such minimal-ABC tree cannot have a $D_{k,2}^2$-branch.
\end{observation}

\begin{observation}
A minimal-ABC tree of order larger than $6739$ does not contain a $B_2$-branch, and therefore, it does not contain a $D_{k,1}^2$-branch.
\end{observation}

\begin{observation}
A minimal-ABC tree of order larger than $1142737$ does not contain a $D_{50}$-branch.
A minimal-ABC tree whose root vertex is of degree larger than $3192$ does not contain $D_{50}$-branches.
\end{observation}

For most of the above observations, we have related theoretical results.
Some of them are close to the optimal ones obtained by the computations.
Anyway, the theoretically proven bounds reduce enough the search space such that we can compute
the minimal-ABC trees up to some  given order, forced by the limitations of the numerical representations of the numbers
(we can obtain numerically consistent results for trees of orders up to $13290000000000000$).
However, formal proofs of some of the above observations, without the use of a computer, could be done in the future, although
that will not add any new value regarding the understanding of the structure of the minimal-ABC trees.

\bigskip

\noindent
{\bf Acknowledgements.}
This work was partially supported by the Slovenian research agency
ARRS (Program No. P1-0383), the bilateral cooperation between China and Slovenia
(Project No. Bi-CN-18-20-015), and the National Natural Science Foundation of China (Grant No. 11701505).


%
%

\bibliographystyle{siam}
\bibliography{abc-mp}

\appendix

\section{Appendix}

\subsection{The size of $D_z$-branches--the negativity of right-hand side of (\ref{eq-Lemma-Dz-30-00})} \label{appendix-upper_bound}

The relations presented in \eqref{eq-Lemma-Dz-10-1} can be regarded as a system of two equations with unknowns $n_k$ and $n_{k-1}$ when parameters $z$, $x$, $k$ are given. The solution is
\begin{eqnarray} \label{eq-nk-nk-1}
n_k = x(z - k + 1) - 7k + 6  \quad \text{and} \quad n_{k-1} = x(k - z) + 7k + 1,
\end{eqnarray}
which means that $n_k$ and $n_{k-1}$ are determined by the parameters $z$, $x$, $k$.
It can be shown that
the right-hand side of (\ref{eq-Lemma-Dz-30-00}) is negative for the following values of the above parameters:

\begin{table}
\begin{center}
\caption{Some values of the parameters for which the right-hand side of (\ref{eq-Lemma-Dz-30-00}) is negative.} \label{table-large-z}
\begin{tabular}{cclccl} 
\hline
$z$ & $x$ &  & $z$ & $x$ &  \\
\hline
 $53$ & $261$ & ($k=52$, $n_{51}=104$, $n_{52}=164$) & $54$ & $131$ &  ($k=52$, $n_{51}=103$, $n_{52}=35$) \\
 \hline
 $55$ & $90$ & ($k=52$, $n_{51}=95$, $n_{52}=2$) &    $56$ & $72$ & ($k=52$, $n_{51}=77$, $n_{52}=2$) \\
  \hline
 $57$ & $60$ & ($k=52$, $n_{51}=65$, $n_{52}=2$) &    $58$ & $52$ & ($k=52$, $n_{51}=53$, $n_{52}=6$) \\
 \hline
 $59$ & $45$ & ($k=52$, $n_{51}=50$, $n_{52}=2$) &    $60$ & $40$ & ($k=52$, $n_{51}=45$, $n_{52}=2$)\\
 \hline
 $61$ & $36$ & ($k=52$, $n_{51}=41$, $n_{52}=2$) &    $62$ & $33$ & ($k=52$, $n_{51}=35$, $n_{52}=5$) \\
 \hline
 $63$ & $30$ & ($k=52$, $n_{51}=35$, $n_{52}=2$) &    $64$ & $28$ & ($k=52$, $n_{51}=29$, $n_{52}=6$) \\
 \hline
 $65$ & $26$ & ($k=52$, $n_{51}=27$, $n_{52}=6$) &    $66$ & $25$ & ($k=52$, $n_{51}=15$, $n_{52}=17$) \\
 \hline
 $67$ & $23$ & ($k=52$, $n_{51}=20$, $n_{52}=10$) &  $68$ & $22$ & ($k=52$, $n_{51}=13$, $n_{52}=16$) \\
 \hline
 $69$ & $21$ & ($k=52$, $n_{51}=8$, $n_{52}=20$) &  $70$ & $21$ & ($k=53$, $n_{52}=15$, $n_{53}=13$) \\
 \hline
 $71$ & $20$ & ($k=53$, $n_{52}=12$, $n_{53}=15$) & $72$ & $19$ &  ($k=53$, $n_{52}=11$, $n_{53}=15$) \\
 \hline
 $73$ & $19$ & ($k=54$, $n_{53}=18$, $n_{54}=8$) &    $74$ & $18$ & ($k=54$,  $n_{53}=19$, $n_{54}=6$) \\
 \hline
 $75$ & $17$ & ($k=54$,  $n_{53}=22$, $n_{54}=2$) &    $76$ & $17$ & ($k=54$,  $n_{53}=5$, $n_{54}=19$) \\
 \hline
 $77$ & $17$ & ($k=55$,  $n_{54}=12$, $n_{55}=12$) &    $78$ & $16$ & ($k=55$,  $n_{54}=18$, $n_{55}=5$)\\
 \hline
 $79$ & $16$ & ($k=55$, $n_{54}=2$, $n_{55}=21$) &    $80$ & $16$ & ($k=56$, $n_{55}=9$, $n_{56}=14$) \\
 \hline
 $81$ & $15$ & ($k=56$, $n_{55}=18$, $n_{56}=4$) &    $82$ & $15$ & ($k=56$, $n_{55}=3$, $n_{56}=19$) \\
 \hline
 $83$ & $15$ & ($k=57$, $n_{56}=10$, $n_{57}=12$) &    $84$ & $15$ & ($k=58$, $n_{57}=17$, $n_{58}=5$) \\
 \hline
 $85$ & $15$ & ($k=58$, $n_{57}=2$, $n_{58}=20$) &  $86$ & $14$ & ($k=58$, $n_{57}=15$, $n_{58}=6$) \\
 \hline
 $87$ & $14$ & ($k=58$, $n_{57}=1$, $n_{58}=20$) &  $88$ & $13$ & ($k=58$, $n_{57}=17$, $n_{58}=3$) \\
 \hline
 $89$ & $13$ & ($k=58$, $n_{57}=4$, $n_{58}=16$) &  $90$ & $13$ &  ($k=59$, $n_{58}=11$, $n_{59}=9$) \\
 \hline
 $91$ & $13$ &($k=60$, $n_{59}=18$, $n_{60}=2$) &  $92$ & $13$ & ($k=60$,  $n_{59}=5$, $n_{60}=15$) \\
 \hline
 $93$ & $13$ &($k=61$,  $n_{60}=12$, $n_{61}=8$) &  $94$ & $13$ & ($k=62$,  $n_{61}=19$, $n_{62}=1$) \\
 \hline
 $95$ & $13$ & ($k=62$,  $n_{61}=6$, $n_{62}=14$) &  $96$ & $13$ & ($k=63$,  $n_{62}=13$, $n_{63}=7$) \\
 \hline
 $97$ & $13$ & ($k=63$,  $n_{62}=0$, $n_{63}=20$) &  $98$ & $13$ & ($k=64$,  $n_{63}=7$, $n_{64}=13$) \\
 \hline
 $99$ & $13$ & ($k=65$,  $n_{64}=14$, $n_{65}=6$) &  $100$ & $13$ &  ($k=65$,  $n_{64}=1$, $n_{65}=19$) \\
 \hline
 $101$ & $13$ & ($k=66$,  $n_{65}=8$, $n_{66}=12$) &  $102$ & $13$ &  ($k=67$,  $n_{66}=15$, $n_{67}=5$) \\
 \hline
 $103$ & $13$ & ($k=67$, $n_{66}=2$, $n_{67}=18$) &  $104$ & $13$ &  ($k=68$, $n_{67}=9$, $n_{68}=11$) \\
 \hline
 $105$ & $13$ & ($k=69$, $n_{68}=16$, $n_{69}=4$) &  $106$ & $13$ &  ($k=69$, $n_{68}=3$, $n_{69}=17$) \\
 \hline
 $107$ & $13$ & ($k=70$, $n_{69}=10$, $n_{70}=10$) &  $108$ & $13$ &  ($k=71$,  $n_{70}=17$, $n_{71}=3$) \\
 \hline
 $109$ & $13$ & ($k=71$,  $n_{70}=4$, $n_{71}=16$) &  $110$ & $13$ &  ($k=72$,  $n_{71}=11$, $n_{72}=9$) \\
 \hline
 $111$ & $13$ & ($k=73$,  $n_{72}=18$, $n_{73}=2$) &  $112$ & $13$ &  ($k=73$,  $n_{72}=5$, $n_{73}=15$) \\
 \hline
 $113$ & $13$ & ($k=74$,  $n_{73}=12$, $n_{74}=8$) &  $114$ & $13$ &  ($k=75$,  $n_{74}=19$, $n_{75}=1$) \\
 \hline
 $115$ & $13$ & ($k=75$,  $n_{74}=6$, $n_{75}=14$) &  $116$ & $13$ &  ($k=76$,  $n_{75}=13$, $n_{76}=7$) \\
 \hline
 $117$ & $13$ & ($k=76$,  $n_{75}=0$, $n_{76}=20$) &  $118$ & $13$ &  ($k=77$,  $n_{76}=7$, $n_{77}=13$) \\
 \hline
 $119$ & $13$ & ($k=78$,  $n_{77}=14$, $n_{78}=6$) &  $120$ & $13$ &  ($k=78$,  $n_{77}=1$, $n_{78}=19$) \\
 \hline
 $121$ & $13$ & ($k=79$,  $n_{78}=8$, $n_{79}=12$) &  $122$ & $13$ &  ($k=80$,  $n_{79}=15$, $n_{80}=5$) \\
 \hline
 $123$ & $13$ & ($k=80$,  $n_{79}=2$, $n_{80}=18$) &  $124$ & $13$ &  ($k=81$,  $n_{80}=9$, $n_{81}=11$) \\
 \hline
 $125$ & $13$ & ($k=82$,  $n_{81}=16$, $n_{82}=4$) &  $126$ & $13$ &  ($k=82$,  $n_{81}=3$, $n_{82}=17$) \\
 \hline
 $127$ & $13$ & ($k=83$,  $n_{82}=10$, $n_{83}=10$) &  $128$ & $13$ &  ($k=84$,  $n_{83}=17$, $n_{84}=3$) \\
 \hline
 $129$ & $13$ & ($k=84$,  $n_{83}=4$, $n_{84}=16$) &  $130$ & $13$ &  ($k=85$,  $n_{84}=11$, $n_{85}=9$) \\
 \hline
 $131$ & $13$ & ($k=86$,  $n_{85}=18$, $n_{86}=2$) & \multicolumn{1}{l}{} &  \multicolumn{1}{l}{} &   \multicolumn{1}{l}{} \\
 \cline{1-3} 
  \end{tabular}
  \end{center}
  \end{table}


By Proposition~\ref{pro-10},  $-f(d(v),d(v_i))+f(d(v)+7,d(v_i))$ decreases in $d(v_i)$. Thus, the right-hand side of (\ref{eq-Lemma-Dz-30-00})
has its maximum value when $d(v_i)=4$, $i=1, \dots d(v)-x$, i.e.,
\begin{eqnarray}  \label{eq-Lemma-Dz-30-00-1}
ABC (G') - ABC (G) &\le&-x \, f(d(v),z+1) -x \, z \, f(z+1,4) - 3 f(4,2) - 3 f(2,1)  + n_{k}  f(d(v)+7,k+1)    \nonumber \\
&& + n_{k-1}  f(d(v)+7,k)  + k n_{k}  f(k+1,4)  +  (k-1) n_{k-1}  f(k,4) \nonumber \\
&& + ( d(v)-x ) ( - f(d(v),4) + f(d(v)+7,4)).
\end{eqnarray}

Next, we show that there exists an upper bound on \eqref{eq-Lemma-Dz-30-00-1} decreasing in $d(v)$.
Observe that the terms related to $d(v)$ in the right-hand side of (\ref{eq-Lemma-Dz-30-00-1}) are
\begin{eqnarray*}
&&-x \, f(d(v),z+1) + n_{k}  f(d(v)+7,k+1)  + n_{k-1}  f(d(v)+7,k)  + ( d(v)-x ) ( - f(d(v),4) + f(d(v)+7,4))  \\
&=&( d(v) -x + 7) f(d(v) + 7, 4) - (d(v) - x) f (d(v),4) + n_k ( - f (d(v) + 7, k) + f (d(v) + 7, k + 1))  \\
&&  + (x + 7) ( - f (d(v) + 7, 4) + f (d(v) + 7, k) ) + x ( - f (d(v), z + 1) + f (d(v) + 7, 4) ).
\end{eqnarray*}

It is not hard to verify that $( d(v) -x + 7) f(d(v) + 7, 4) - (d(v) - x) f (d(v),4)$ decreases in $d (v) \ge x > 0$, by verifying the negativity of its derivative.
From Propositions \ref{pro-10} and \ref{pro-20}, we have
\begin{itemize}
\item
$- f (d(v) + 7, k) + f (d(v) + 7, k + 1)$ decreases in $d (v)$;

\item
$- f (d(v) + 7, 4) + f (d(v) + 7, k)$ decreases in $d (v)$;

\item
$- f (d(v), z + 1) + f (d(v) + 7, 4)$ increases in $d (v)$, i.e.,
\begin{eqnarray*}
- f (d(v), z + 1) + f (d(v) + 7, 4) &\le&   \lim_{d(v) \rightarrow + \infty}(- f (d(v), z + 1) + f (d(v) + 7, 4))   \\
 &=&  - \sqrt{\frac{1}{z+1}} + \sqrt{\frac{1}{4}}.
\end{eqnarray*}

\end{itemize}

Thus, the right-hand side of (\ref{eq-Lemma-Dz-30-00-1}) has another (weaker) upper bound:
\begin{eqnarray}  \label{eq-Lemma-Dz-30-00-2}
%
ABC (G') - ABC (G) &\le&( d(v) -x + 7) f(d(v) + 7, 4) - (d(v) - x) f (d(v),4)  + n_k ( - f (d(v) + 7, k)  \nonumber  \\
&&+ f (d(v) + 7, k + 1))   + (x + 7) ( - f (d(v) + 7, 4) + f (d(v) + 7, k) )    \nonumber  \\
&& + x \left( - \sqrt{\frac{1}{z+1}} + \sqrt{\frac{1}{4}} \right)  -x \, z \, f(z+1,4) - 3 f(4,2) - 3 f(2,1)   \nonumber  \\
&& + k n_{k}  f(k+1,4)  +  (k-1) n_{k-1}  f(k,4),
\end{eqnarray}
which decreases in $d (v)$.

Let us consider the case $z=53$ and $x=261$ as a representative (since it is the worst case with the largest valid $x$) to illustrate how to show that the right-hand side of (\ref{eq-Lemma-Dz-30-00}) is negative. With the assumption that $z=53$ and $x=261$, we set $k=52$, which implies that $n_{51}=104$ and  $n_{52}=164$ from \eqref{eq-nk-nk-1}.
Now, let us substitute $z=53$, $x=261$, $k=52$, $n_{51}=104$, $n_{52}=164$ into (\ref{eq-Lemma-Dz-30-00-2}), we can see that the right-hand side of  (\ref{eq-Lemma-Dz-30-00-2})
is negative when $d(v) = 259226$, and so is it for $d(v) \ge 259226$.
As to $x = 261 \le d(v) < 259226$, we resort to the right-hand side of (\ref{eq-Lemma-Dz-30-00-1}), which is always negative by direct calculations.
Thus, we have verified that a minimal-ABC tree cannot contain (at least) $x = 261$ $D_z$-branches with $z = 53$.

Similarly, it can be verified that for the rest of the combinations of the listed parameters in Table \ref{table-large-z}, the change of the ABC index is negative.

\subsection{The size of $D_z$-branches--the negativity of right-hand side of (\ref{eq-Lemma-Dz-30-00-second-pre})} \label{appendix-upper_bound-2}

The relations from \eqref{eq-Lemma-Dz-10-12} can be reformulated as
\begin{eqnarray} \label{eq-nk-nk-2}
n_k = x(z - k + 1) + 7k - 6  \quad \text{and} \quad n_{k-1} = x(k - z) - 7k - 1,
\end{eqnarray}
which means that $n_k$ and $n_{k-1}$ are determined by the parameters $x$, $z$, and $k$.
It can be easily verified that for the following values of the above parameters (including a lower bound of $d(v)$), the right-hand side of (\ref{eq-Lemma-Dz-30-00-second-pre}) is negative:

\begin{table}
\begin{center}
\caption{ Some values of the parameters for which the right-hand side of (\ref{eq-Lemma-Dz-30-00-second-pre}) is negative.} \label{table-small-z}
\begin{tabular}{ccc ccc}
\hline
$z$ & $x$ & $d(v)$   & $z$ & $x$ & $d(v)$    \\
\hline
 $51$ & $365$ &  $\geq 3249$  &  $50$ & $183$ &  $\geq 1358$   \\
   \multicolumn{3}{c}{($k=52$, $n_{51}=0$, $n_{52}=358$)}  &  \multicolumn{3}{c}{($k=52$, $n_{51}=0$, $n_{52}=358$)}  \\
 \hline
 $49$ & $122$ &  $\geq 825$ &   $48$ & $92$ &  $\geq 572$    \\
  \multicolumn{3}{c}{($k=52$, $n_{51}=1$, $n_{52}=114$) }  &  \multicolumn{3}{c}{($k=52$, $n_{51}=3$, $n_{52}=82$)}  \\
 \hline
  $47$ & $73$ &  $\geq 430$ &    $46$ & $61$ &  $\geq 333$   \\
    \multicolumn{3}{c}{($k=52$, $n_{51}=0$, $n_{52}=66$)}  &  \multicolumn{3}{c}{($k=52$, $n_{51}=1$, $n_{52}=53$)}  \\
   \hline
    $45$ & $53$ &  $\geq 263$ &    $44$ & $46$ &  $\geq 215$    \\
   \multicolumn{3}{c}{($k=52$, $n_{51}=6$, $n_{52}=40$)}  &  \multicolumn{3}{c}{($k=52$, $n_{51}=3$, $n_{52}=36$)}  \\
    \hline
    $43$ & $41$ &  $\geq 177$ &    $42$ & $37$ &  $\geq 146$    \\
    \multicolumn{3}{c}{($k=52$, $n_{51}=4$, $n_{52}=30$)}  &  \multicolumn{3}{c}{($k=52$, $n_{51}=5$, $n_{52}=25$)}  \\
    \hline
     $41$ & $34$ &  $\geq 120$  &  $40$ & $31$ &  $\geq 100$  \\
     \multicolumn{3}{c}{($k=52$, $n_{51}=9$, $n_{52}=18$)}  &  \multicolumn{3}{c}{($k=52$, $n_{51}=7$, $n_{52}=17$) }  \\
     \hline
     $39$ & $29$ &  $\geq 80$ &    $38$ & $27$ &  $\geq 64$   \\
      \multicolumn{3}{c}{($k=52$, $n_{51}=12$, $n_{52}=10$)}  &  \multicolumn{3}{c}{($k=52$, $n_{51}=13$, $n_{52}=7$) }  \\
     \hline
     $37$ & $25$ &  $\geq 47$ &     $36$ & $23$ &  $\geq z$    \\
     \multicolumn{3}{c}{($k=52$, $n_{51}=10$, $n_{52}=8$)}  &  \multicolumn{3}{c}{($k=52$, $n_{51}=3$, $n_{52}=13$) }  \\
     \hline
     $35$ & $22$ &  $\geq z $ &     $34$ & $21$ &  $\geq z$    \\
     \multicolumn{3}{c}{($k=52$, $n_{51}=9$, $n_{52}=6$)}  &  \multicolumn{3}{c}{ ($k=52$, $n_{51}=13$, $n_{52}=1$)}  \\
     \hline
     $33$ & $20$ &  $\geq z $  &  $32$ & $19$ &  $\geq z$    \\
     \multicolumn{3}{c}{($k=51$, $n_{50}=2$, $n_{51}=11$) }  &  \multicolumn{3}{c}{($k=51$, $n_{50}=3$, $n_{51}=9$) }  \\
     \hline
      $31$ & $18$ &  $\geq z $ &  $30$ & $18$ &  $\geq z$    \\
      \multicolumn{3}{c}{ ($k=51$, $n_{50}=2$, $n_{51}=9$) }  &  \multicolumn{3}{c}{ ($k=50$, $n_{49}=9$, $n_{50}=2$)  }  \\
      \hline
       $29$ & $17$ &  $\geq z $ &  $28$ & $16$ &  $\geq z$    \\
        \multicolumn{3}{c}{  ($k=50$, $n_{49}=6$, $n_{50}=4$) }  &  \multicolumn{3}{c}{  ($k=50$, $n_{49}=1$, $n_{50}=8$)  }  \\
      \hline
       $27$ & $16$ &  $\geq z $    &  $26$ & $15$ &  $\geq z$     \\
       \multicolumn{3}{c}{  ($k=49$, $n_{48}=8$, $n_{49}=1$) }  &  \multicolumn{3}{c}{  ($k=49$, $n_{48}=1$, $n_{49}=7$)  }  \\
      \hline
      $25$ & $15$ &  $\geq z $    &  $24$ & $15$ &  $\geq z$     \\
      \multicolumn{3}{c}{ ($k=48$, $n_{47}=8$, $n_{48}=0$) }  &  \multicolumn{3}{c}{  ($k=46$, $n_{45}=7$, $n_{46}=1$)  }  \\
     \hline
      $23$ & $15$ &  $\geq z $    &  $22$ & $15$ &  $\geq z$     \\
       \multicolumn{3}{c}{ ($k=44$, $n_{43}=6$, $n_{44}=2$) }  &  \multicolumn{3}{c}{ ($k=42$, $n_{41}=5$, $n_{42}=3$) }  \\
     \hline
     $21$ & $15$ &  $\geq z $ &     $20$ & $15$ &  $\geq z$    \\
     \multicolumn{3}{c}{ ($k=40$, $n_{39}=4$, $n_{40}=4$) }  &  \multicolumn{3}{c}{ ($k=38$, $n_{37}=3$, $n_{38}=5$)  }  \\
    \hline
     $19$ & $15$ &  $\geq z$ &     $18$ & $15$ &  $\geq z$    \\
     \multicolumn{3}{c}{ ($k=36$, $n_{35}=2$, $n_{36}=6$) }  &  \multicolumn{3}{c}{  ($k=34$, $n_{33}=1$, $n_{34}=7$) }  \\
    \hline
     $17$ & $15$ &  $\geq z $    &  $16$ & $15$ &  $\geq z$    \\
     \multicolumn{3}{c}{ ($k=32$, $n_{31}=0$, $n_{32}=8$) }  &  \multicolumn{3}{c}{ ($k=31$, $n_{30}=7$, $n_{31}=1$) }  \\
     \hline
     $15$ & $14$ &  $\geq z$ &   &   &      \\
      \multicolumn{3}{c}{($k=31$, $n_{30}=6$, $n_{31}=1$) }  &  \multicolumn{3}{c}{  }  \\
      \cline{1-3}
  \end{tabular}
  \end{center}
  \end{table}

The above difference of $ABC (G') - ABC (G)$, in the right-hand side of \eqref{eq-Lemma-Dz-30-00-second-pre}, is bounded from above by
\begin{eqnarray}  \label{eq-Lemma-Dz-30-00-second}
%
ABC (G') - ABC (G)  &\le&-x \, f(d(v),z+1) -x \, z \, f(z+1,4) + 3 f(4,2) + 3 f(2,1)  \nonumber \\
&&  + n_{k}  f(d(v)-7,k+1)  + n_{k-1}  f(d(v)-7,k) + k n_{k}  f(k+1,4)    \nonumber \\
&& +  (k-1) n_{k-1}  f(k,4)   + (d(v)-x)(- f(d(v), z+2) + f(d(v)-7, z+2)), \qquad \;
\end{eqnarray}
by Proposition~\ref{pro-20} and the fact that $d(v_i) \le z + 2$ from Proposition~\ref{pro-diff-Dz}, $i=1, \dots, d(v)-x$.
The terms related to $d(v)$ in the right-hand side of (\ref{eq-Lemma-Dz-30-00-second}) are
\beq
&-&x f(d(v),z+1) + n_{k}  f(d(v)-7,k+1)  + n_{k-1}  f(d(v)-7,k)  + ( d(v)-x ) ( - f(d(v),z+2) + f(d(v)-7,z+2) ), \nonumber \\ 
\eeq
which is equal to
\begin{eqnarray*}
&&( d(v) -x - 7) f(d(v) - 7, z+2) - (d(v) - x) f (d(v),z+2)   + n_k ( - f (d(v) - 7, k) + f (d(v) - 7, k + 1)) \\
&& + x ( - f (d(v), z+1) + f (d(v) - 7, k) ) + 7 ( - f (d(v) - 7, k) + f (d(v) - 7, z+2) ).
\end{eqnarray*}

For each combination of $z$ and $x$ listed in Table \ref{table-small-z}, it is not hard to verify that the first derivative of
$( d(v) -x - 7) f(d(v) - 7, z+2) - (d(v) - x) f (d(v),z+2)$ with respect to $d(v)$ is positive,
and thus
this expression increases in $d (v)$, so
\begin{eqnarray*}
&&( d(v) -x - 7) f(d(v) - 7, z+2) - (d(v) - x) f (d(v),z+2) \\
&\le& \lim_{d(v) \rightarrow + \infty}(( d(v) -x - 7) f(d(v) - 7, z+2) - (d(v) - x) f (d(v),z+2)) \quad  = \quad - \frac{7} {\sqrt{z+2}}.
\end{eqnarray*}
By Propositions \ref{pro-10} and \ref{pro-20}, we have
\begin{itemize}
\item
$- f (d(v) - 7, k) + f (d(v) - 7, k + 1)$
decreases in $d (v)$;

\item
$- f (d(v), z+1) + f (d(v) - 7, k)$
decreases in $d (v)$, since $k \ge z + 1$;

\item
$- f (d(v) - 7, k) + f (d(v) - 7, z+2)$
decreases in $d (v)$ when $z+2 > k$ (only when $z = 51$ and $k = 52$), is equal to $0$ when $z + 2 = k$ (only when $z = 50$ and $k = 52$), 
and increases in $d (v)$, so
\begin{eqnarray*}
 - f (d(v) - 7, k) + f (d(v) - 7, z+2)  \;\; \le  \;\;  \lim_{d(v) \rightarrow + \infty}(- f (d(v) - 7, k) + f (d(v) - 7, z+2)) \;\;  = \;\;  - \sqrt{\frac{1} {k}} + \sqrt{\frac{1} {z+2}}, \;\; \\ 
\end{eqnarray*}
when $z + 2 < k$ (i.e., $15 \le z \le 49$).

\end{itemize}

Therefore, we get a weaker upper bound on $ABC (G') - ABC (G)$ than the right-hand side of \eqref{eq-Lemma-Dz-30-00-second}:
\begin{itemize}

\item
when $z + 2 > k$ (more precisely, when $z = 51$ and $k = 52$),
\begin{eqnarray}  \label{eq-Lemma-Dz-30-00-third-1}
ABC (G') - ABC (G) &\le& - \frac{7} {\sqrt{z+2}}  + n_k ( - f (d(v) - 7, k) + f (d(v) - 7, k + 1))   \nonumber\\
&& + x ( - f (d(v), z+1) + f (d(v) - 7, k) )   \nonumber \\
&& + 7 \left( - f (d(v) - 7, k) + f (d(v) - 7, z+2) \right)  -x \, z \, f(z+1,4) \nonumber \\
&& + 3 f(4,2) + 3 f(2,1)  + k n_{k}  f(k+1,4)  +  (k-1) n_{k-1}  f(k,4),
\end{eqnarray}
and it  decreases in $d (v)$.

\item
When $z + 2 \le k$ (i.e., $15 \le z \le 50$),
\begin{eqnarray}  \label{eq-Lemma-Dz-30-00-third-2}
ABC (G') - ABC (G)  &\le& - \frac{7} {\sqrt{z+2}}  + n_k ( - f (d(v) - 7, k) + f (d(v) - 7, k + 1))  \nonumber\\
&&+ x ( - f (d(v), z+1) + f (d(v) - 7, k) )   + 7 \left(- \sqrt{\frac{1} {k}} + \sqrt{\frac{1} {z+2}} \right)  \nonumber\\
&&  -x \, z \, f(z+1,4) + 3 f(4,2) + 3 f(2,1) + k n_{k}  f(k+1,4) \nonumber \\
&&  +  (k-1) n_{k-1}  f(k,4),   \nonumber
\end{eqnarray}
and it  decreases in $d (v)$, too.
\end{itemize}

We take the case $z=51$, $x=365$, and $d(v) \geq 3249$ as a representative to illustrate how to show that the right-hand side of (\ref{eq-Lemma-Dz-30-00-second-pre}) is negative. Observe that as previously we are considering the worst case with the largest valid $x$.  With the assumption that $z=51$, $x=365$, and $d(v) \geq 3249$, we set $k=52$, which implies that $n_{51}=0$ and  $n_{52}=358$ from \eqref{eq-nk-nk-2}.
After substituting $z=51$, $x=365$, $k=52$, $n_{51}=0$, $n_{52}=358$ into the right-hand side of (\ref{eq-Lemma-Dz-30-00-third-1}), we get that the right-hand side of  (\ref{eq-Lemma-Dz-30-00-third-1}) is negative when $d(v) =3863$, and thus also for $d(v) \ge 3863$.
When $3249 \le d(v) < 3863$, we resort to the right-hand side of (\ref{eq-Lemma-Dz-30-00-second}), which is always negative by direct calculations.
Thus, we have verified that a minimal-ABC tree cannot contain (at least) $x = 365$ $D_z$-branches with $z = 51$ when $d(v)\geq 3249$.

For the rest of the combinations of the listed parameters in Table \ref{table-small-z}, it can be verified similarly that the change of the ABC index is negative.

\subsection{The number of $B$-branches adjacent to the root--the negativity of right-hand side of (\ref{eq-thm-bound-B3-branches-to-root})} \label{appendix-bound-B-branches-to-root}

Since $d(v')=d(v)-358$, by Proposition~\ref{pro-20} it follows that
$-f(d(v),4)+f(d(v'),4) < -f(d(v),5)+f(d(v'),5)$. Thus,
we may assume that $n_4 = 4$ in right-hand side of \eqref{eq-thm-bound-B3-branches-to-root}, which is the worst case.


Assume that the maximal size of $D$-branches in $G$ is $k$, where $15 \le k \le 131$.
Notice that the expression $-f(d(v),d(v_i))+f(d(v'),d(v_i))$ increases in $d(v_i)$, in which the maximal value for $d(v_i)$ is $k + 1$.
Therefore, $\sum_{i=1}^{d(v)-x}( -f(d(v),d(v_i))+f(d(v'),d(v_i)))$ is bounded from above by
$(d(v)-x) ( -f(d(v),k+1)+f(d(v'),k+1))$.
So the right-hand side of (\ref{eq-thm-bound-B3-branches-to-root}) has an upper bound:
\begin{eqnarray} \label{eq-thm-bound-B3-branches-to-root-11}
ABC (G') - ABC (G) &\le& (d(v)-x) ( -f(d(v),k+1)+f(d(v'),k+1)) -f(d(v),4)+f(d(v'),53) \nonumber \\
&& +6(-f(2,1)+f(d(v'),53)) +364(-f(d(v),4)+f(53,4)) \nonumber \\
&& +(x-369)(-f(d(v),4)+f(d(v'),4)) +4(-f(d(v),5)+f(d(v'),5)).
\end{eqnarray}

First, we claim that the right-hand side of (\ref{eq-thm-bound-B3-branches-to-root-11}) decreases in $x$.
Because the coefficients of $x$ in the right-hand side of (\ref{eq-thm-bound-B3-branches-to-root-11}) are
$$
(-f(d(v),4)+f(d(v'),4)) - ( -f(d(v),k+1)+f(d(v'),k+1)),
$$
which is negative
from Proposition \ref{pro-20}, it implies that the right-hand side of (\ref{eq-thm-bound-B3-branches-to-root-11}) decreases in $x$.

Next, we get an upper bound on $ABC (G') - ABC (G)$ which decreases in $d(v)$.
After rearranging the terms in the right-hand side of (\ref{eq-thm-bound-B3-branches-to-root-11}), we have
 \begin{eqnarray} \label{eq-thm-bound-B3-branches-to-root-2}
ABC (G') - ABC (G) &\le& (d(v') - x) f (d(v'),k+1) - (d(v)-x) f (d(v),k+1) \nonumber \\
&& + 365 ( - f(d(v),4) + f (d(v'),53))  +  358( -f(d(v'),53)+f(d(v'),k+1))   \nonumber \\
&& +(x-369)(-f(d(v),4)+f(d(v'),4)) +4(-f(d(v),5)+f(d(v'),5))  \nonumber \\ 
&& - 6 f(2,1) + 364 f(53,4). \nonumber
\end{eqnarray}

It is not hard to check that $(d(v') - x) f (d(v'),k+1) - (d(v)-x) f (d(v),k+1)$ increases in $d(v) \ge x$, for $15 \le k \le 131$ and $x \ge 400$,  further
\begin{eqnarray*}
&&(d(v') - x) f (d(v'),k+1) - (d(v)-x) f (d(v),k+1) \\
&\le& \lim_{d(v) \rightarrow + \infty}((d(v') - x) f (d(v'),k+1) - (d(v)-x) f (d(v),k+1))  \quad  =  \quad - \frac{358}{\sqrt{k+1}}. \\
\end{eqnarray*}
And from Proposition~\ref{pro-20}, all the following terms decrease in $d(v)$:
\begin{itemize}
\item
$- f(d(v),4) + f (d(v'),53)$,
\item
$-f(d(v),4)+f(d(v'),4)$,
\item
$-f(d(v),5)+f(d(v'),5)$.
\end{itemize}
Again by Propositions~\ref{pro-10} and \ref{pro-20}, the exclusive term $-f(d(v'),53)+f(d(v'),k+1)$ decreases in $d(v)$ if $k > 52$, increases in $d(v)$ if $k < 52$, and is equal to $0$ if $k = 52$. Moreover, when $k < 52$,
\begin{eqnarray*}
&&-f(d(v'),53)+f(d(v'),k+1)\quad \le \quad  \lim_{d(v) \rightarrow + \infty}(-f(d(v'),53)+f(d(v'),k+1)) \quad =  \quad  - \frac{1}{\sqrt{53}} + \frac{1}{\sqrt{k+1}}.
\end{eqnarray*}

Thus, we obtain a weaker upper bound on $ABC (G') - ABC (G)$:
When $k \ge 52$,
 \begin{eqnarray} \label{eq-thm-bound-B3-branches-to-root-further}
ABC (G') - ABC (G) &\le& - \frac{358}{\sqrt{k+1}} + 365 ( - f(d(v),4) + f (d(v'),53))  \nonumber \\
&& +  358( -f(d(v'),53)+f(d(v'),k+1)) +(x-369)(-f(d(v),4)+f(d(v'),4))  \nonumber \\
&& + 4(-f(d(v),5)+f(d(v'),5)) - 6 f(2,1) + 364 f(53,4),
\end{eqnarray}
which decreases in $d(v)$, and when $k < 52$,
 \begin{eqnarray} \label{eq-thm-bound-B3-branches-to-root-further-1}
ABC (G') - ABC (G) &\le& - \frac{358}{\sqrt{k+1}} + 365 ( - f(d(v),4) + f (d(v'),53))  +  358 \left( - \frac{1}{\sqrt{53}} + \frac{1}{\sqrt{k+1}} \right)  \nonumber \\
&&  +(x-369)(-f(d(v),4) +f(d(v'),4)) +4(-f(d(v),5)+f(d(v'),5)) \nonumber \\
&& - 6 f(2,1) + 364 f(53,4) \nonumber\\
&=&  365 ( - f(d(v),4) + f (d(v'),53))  -  \frac{358}{\sqrt{53}} +(x-369)(-f(d(v),4) +f(d(v'),4)) \nonumber \\
&&  +4(-f(d(v),5)+f(d(v'),5)) - 6 f(2,1) + 364 f(53,4),
\end{eqnarray}
which also decreases in $d(v)$.

Set $d(v) = 5300$, $52 \le k \le 131$ and $x = 400$ in the right-hand side of \eqref{eq-thm-bound-B3-branches-to-root-further},
and $d(v) = 3400$, $15 \le k \le 51$ and $x = 400$ in the right-hand side of \eqref{eq-thm-bound-B3-branches-to-root-further-1},
the corresponding upper bound on $ABC (G') - ABC (G)$ is always negative, so are when $d(v) \ge 5300$ and $d(v) \ge 3400$, respectively.
It also implies that the right-hand side of \eqref{eq-thm-bound-B3-branches-to-root-11} is negative for:
\begin{itemize}
\item
$d(v) \ge 5300$, $52 \le k \le 131$ and $x \ge 400$,
\item
$d(v) \ge 3400$, $15 \le k \le 51$ and $x \ge 400$.
\end{itemize}

%

When $k=52$,
we assume that there are only $D$-branches of size $52$ (which is the worst case).
For $x \leq d(v) < 5300$ and $x = 917$,
by direct calculations we obtain that the right-hand side of (\ref{eq-thm-bound-B3-branches-to-root-11}) is negative.

Similarly as for $k=52$, in the cases $k=51, 50$,
we may assume that there are only $D$-branches of size $51$ and $D$-branches of size $50$, respectively.
Then we obtain in those cases that the right-hand side of (\ref{eq-thm-bound-B3-branches-to-root-11}) is negative
for $x = 909$ and $x = 901$, respectively, and  $x \leq d(v) < 3400$.

As to the remaining cases with small $d(v)$, in order to get more precise results, we partition our proofs into two parts: each $D$-branch is a $D_z$-branch for some $z$; there exists some $D$-branch containing $B_2$-, or $B_3^{**}$-branch.

\noindent
{\em Case 1: Each $D$-branch is a $D_z$-branch for some $z$}.

In this case, besides $D_k$-branches, it can only contain $D_{k-1}$-branches, by Proposition \ref{pro-diff-Dz}.
From Lemmas \ref{le-UpperBoundOnNumberOf_Dz-10} and \ref{le-LowerBoundOnNumberOf_Dz}
(Tables \ref{table-large-z} and \ref{table-small-z}), we can find the upper bound on the number of $D_{k}$-branches,
denoted here by $n_k^{max}$ (e.g., $n_{53}^{max} = 260$, the value of the corresponding $x$ in Table \ref{table-large-z} minus $1$).
It is worth mentioning that $n_{52}^{max}$ does not exist.

Let us introduce a refined version of \eqref{eq-thm-bound-B3-branches-to-root-11},
by noting that the number of $D_k$-branches is bounded from above by $n_k^{max}$
(see Tables \ref{table-large-z} and \ref{table-small-z}), unless $k = 52$.
Thus, when $k \in \{ 15, \dots, 131 \} \setminus \{52,53\}$,
\begin{eqnarray} \label{eq-thm-bound-B3-branches-to-root-11-refine}
ABC (G') - ABC (G)  &=& (d(v)-x-n_k-n_{k-1}) ( -f(d(v),4)+f(d(v'),4) )  \nonumber \\
&&+ n_{k-1} ( -f(d(v),k)+f(d(v'),k) ) + n_k ( -f(d(v),k + 1)+f(d(v'),k+1) ) \nonumber \\
&& -f(d(v),4)+f(d(v'),53)  +6(-f(2,1)+f(d(v'),53)) \nonumber \\
&& +364(-f(d(v),4)+f(53,4))  +(x-369)(-f(d(v),4)+f(d(v'),4)) \nonumber \\ 
&& +4(-f(d(v),5)+f(d(v'),5)) \nonumber \\
&=& (d(v)-n_k-n_{k-1}) ( -f(d(v),4)+f(d(v'),4) )  \nonumber \\
&&+ n_{k-1} ( -f(d(v),k)+f(d(v'),k) ) + n_k ( -f(d(v),k + 1)+f(d(v'),k+1) ) \nonumber \\
&&  -f(d(v),4)+f(d(v'),53) +6(-f(2,1)+f(d(v'),53)) \nonumber \\
&& +364(-f(d(v),4)+f(53,4))  -369(-f(d(v),4)+f(d(v'),4)) \nonumber \\
&& +4(-f(d(v),5)+f(d(v'),5)),
\end{eqnarray}
where $n_i$ is the number of $D_i$-branches, $0 \leq n_i \leq n_i^{max}$, $i \in \{ 15, \dots, 131 \} \setminus 52$. In particular,
when $k = 53$,
\begin{eqnarray} \label{eq-thm-bound-B3-branches-to-root-11-refine-1}
ABC (G') - ABC (G)  &\le& (d(v)-x-n_k) ( -f(d(v),k)+f(d(v'),k) ) \nonumber \\
&&+ n_k ( -f(d(v),k + 1)+f(d(v'),k+1) ) -f(d(v),4)+f(d(v'),53) \nonumber \\
&& +6(-f(2,1)+f(d(v'),53)) +364(-f(d(v),4)+f(53,4)) \nonumber \\
&& +(x-369)(-f(d(v),4)+f(d(v'),4)) +4(-f(d(v),5)+f(d(v'),5)) \nonumber \\
&=& (d(v)-x-n_{53}) ( -f(d(v),53)+f(d(v'),53) ) \nonumber \\
&&+ n_{53} ( -f(d(v),54)+f(d(v'),54) ) -f(d(v),4)+f(d(v'),53) \nonumber \\
&& +6(-f(2,1)+f(d(v'),53)) +364(-f(d(v),4)+f(53,4)) \nonumber \\
&& +(x-369)(-f(d(v),4)+f(d(v'),4)) +4(-f(d(v),5)+f(d(v'),5)). \quad \quad \quad
\end{eqnarray}

Let us consider the case when $k = 53$, i.e., $G$ contains $D_{53}$-branches, possibly together with $D_{52}$-branches.
With the constraints $x + n_{53} \leq d(v) < 5300$ and $0 \leq n_{53} \leq n_{53}^{max} = 260$ (Table \ref{table-large-z}),
set $x = 920$,
it can be verified by direct calculations that the right-hand side of (\ref{eq-thm-bound-B3-branches-to-root-11-refine-1}) is negative.
And with $x = 919$, it is possible that the change is positive. It means that
the smallest value of $x$ for which we obtain the negative value of (\ref{eq-thm-bound-B3-branches-to-root-11-refine-1}) is $920$.

Now, it remains to show that
we obtain a negative change of the ABC index
for $920 \leq x + n_k + n_{k-1} \leq d(v) < 5300$ when $54 \leq k \leq 131$,
and for $920 \leq x + n_k + n_{k-1} \leq d(v) < 3400$ when $15 \leq k \leq 49$.
Here we always use the right-hand side of (\ref{eq-thm-bound-B3-branches-to-root-11-refine}).

When $k = 54$, i.e., $G$ may contain $D_{54}$-branches and $D_{53}$-branches.
For  $0 \leq n_{54} \leq 130$, $0 \leq n_{53} \leq 260$ (Table \ref{table-large-z}), and
 $x = 883$,  $x + n_{54} + n_{53} \leq d(v)< 5300$, it can be verified by direct calculations that
 the right-hand side of (\ref{eq-thm-bound-B3-branches-to-root-11-refine}) is negative.
 Similarly, as in the case  $k = 54$, we can argue for the case $55 \leq k \leq 131$.
 The corresponding values of the maximum numbers of $B$-branches attached to the root are given in Table~\ref{tbl-maxB-2}.

\begin{table}[ht]
\begin{center}
\caption{ The possible combinations of $D_k$- and $D_{k-1}$-branches, $53 \leq k \leq 131$, and the corresponding upper bounds on the  numbers of $B$-branches are presented. By obtaining the above results, the occurrence of $D_k$-branches is necessary, while the $D_{k-1}$-branches may or may not occur in a minimal-ABC tree.} \label{tbl-maxB-2}
\begin{tabular}{cccccc}
\hline
combination  & upper & combination  & upper   & combination  & upper  \\
                     &  bound &                     & bound   &                      & bound  \\
\hline
 $D_{52}$ and $D_{53}$ &  $919$ & $D_{53}$ and $D_{54}$ &  $882$ & $D_{54}$ and $D_{55}$ & $798$  \\
 $D_{55}$ and $D_{56}$ &  $749$  &  $D_{56}$ and $D_{57}$ &  $717$ & $D_{57}$ and $D_{58}$ & $693$  \\
 $D_{58}$ and $D_{59}$ &  $672$  & $D_{59}$ and $D_{60}$ &  $654$ & $D_{60}$ and $D_{61}$ & $639$   \\
 $D_{61}$ and $D_{62}$ &  $627$ & $D_{62}$ and $D_{63}$ &  $615$ & $D_{63}$ and $D_{64}$ & $605$  \\
 $D_{64}$ and $D_{65}$ &  $597$  & $D_{65}$ and $D_{66}$ &  $591$ & $D_{66}$ and $D_{67}$ & $584$  \\
 $D_{67}$ and $D_{68}$ &  $577$  & $D_{68}$ and $D_{69}$ &  $573$ & $D_{69}$ and $D_{70}$ & $571$  \\
 $D_{70}$ and $D_{71}$ &  $569$  & $D_{71}$ and $D_{72}$ &  $564$ & $D_{72}$ and $D_{73}$ & $562$  \\
 $D_{73}$-and $D_{74}$ &  $560$  & $D_{74}$ and $D_{75}$ &  $554$ & $D_{75}$ and $D_{76}$ &  $552$  \\
 $D_{76}$ and $D_{77}$ &  $553$  & $D_{77}$ and $D_{78}$ &  $550$ & $D_{78}$ and $D_{79}$ &  $547$  \\
 $D_{79}$ and $D_{80}$ &  $548$  & $D_{80}$ and $D_{81}$ &  $545$ & $D_{81}$ and $D_{82}$ &  $542$  \\
 $D_{82}$ and $D_{83}$ &  $543$  & $D_{83}$ and $D_{84}$ &  $544$ & $D_{84}$ and $D_{85}$ &  $545$  \\
 $D_{85}$- and $D_{86}$ &  $541$  &  $D_{86}$ and $D_{87}$ &  $538$ & $D_{87}$ and $D_{88}$ &  $535$ \\
 $D_{88}$ and $D_{89}$ &  $531$  & $D_{89}$ and $D_{90}$ &  $532$ & $D_{90}$ and $D_{91}$ &  $533$  \\
 $D_{91}$ and $D_{92}$ &  $533$  & $D_{92}$ and $D_{93}$ &  $534$ & $D_{93}$ and $D_{94}$ &  $535$  \\
 $D_{94}$ and $D_{95}$ &  $535$  & $D_{95}$ and $D_{96}$ &  $536$ & $D_{96}$ and $D_{97}$ &  $536$  \\
 $D_{97}$ and $D_{98}$ &  $537$  & $D_{98}$ and $D_{99}$ &  $538$ & $D_{99}$ and $D_{100}$ &  $538$  \\
 $D_{100}$ and $D_{101}$ &  $539$  & $D_{101}$ and $D_{102}$ &  $539$ & $D_{102}$ and $D_{103}$ &  $540$  \\
 $D_{103}$ and $D_{104}$ &  $540$  & $D_{104}$ and $D_{105}$ &  $541$ & $D_{105}$ and $D_{106}$ &  $541$  \\
 $D_{106}$ and $D_{107}$ &  $542$  & $D_{107}$ and $D_{108}$ &  $543$ & $D_{108}$ and $D_{109}$ &  $543$  \\
 $D_{109}$ and $D_{110}$ &  $544$  & $D_{110}$ and $D_{111}$ &   $544$ & $D_{111}$ and $D_{112}$ &  $545$  \\
 $D_{112}$ and $D_{113}$ &  $545$  & $D_{113}$ and $D_{114}$ &   $546$ & $D_{114}$ and $D_{115}$ &  $546$  \\
 $D_{115}$ and $D_{116}$ &  $547$  & $D_{116}$ and $D_{117}$ &   $547$ & $D_{117}$ and $D_{118}$ &  $548$  \\
 $D_{118}$ and $D_{119}$ &  $548$  & $D_{119}$ and $D_{120}$ &   $549$ & $D_{120}$ and $D_{121}$ &  $549$  \\
 $D_{121}$ and $D_{122}$ &  $550$  & $D_{122}$ and $D_{123}$ &   $550$ & $D_{123}$ and $D_{124}$ &  $551$  \\
 $D_{124}$ and $D_{125}$ &  $551$  & $D_{125}$ and $D_{126}$ &   $552$ & $D_{126}$ and $D_{127}$ &  $552$  \\
 $D_{127}$ and $D_{128}$ &  $553$  & $D_{128}$ and $D_{129}$ &   $553$ & $D_{129}$ and $D_{130}$ &  $553$  \\
  $D_{130}$ and $D_{131}$ &  $554$   &  &  & & \\
 \hline
  \end{tabular}
  \end{center}
  \end{table}

 For the cases $15 \leq k \leq 49$, we proceed as in the cases $k = 49$ and $k = 48$, which we elaborate next. When $k = 49$, recall that $0 \leq n_{49} \leq n_{49}^{max} = 121$ when $d(v) \ge 825$, $0 \leq n_{48} \leq n_{48}^{max} = 91$ when $d(v) \ge 572$ (Table \ref{table-small-z}), the right-hand side of (\ref{eq-thm-bound-B3-branches-to-root-11-refine}) is negative from direct calculations, for $x + n_{49} + n_{48} \leq d(v) < 3400$ and $x = 825$.
 When $k = 48$, with the following constraints: $0 \leq n_{48} \leq n_{48}^{max} = 91$ when $d(v) \ge 572$, $0 \leq n_{47} \leq n_{47}^{max} = 72$ when $d(v) \ge 430$ (Table \ref{table-small-z}), and
 $x = 724$, $x + n_{48} + n_{47} \leq d(v) < 3400$, it can be verified by direct calculations that
 the right-hand side of (\ref{eq-thm-bound-B3-branches-to-root-11-refine}) is negative.
The corresponding values of the maximum numbers of $B$-branches attached to the root, when
$15 \leq k \leq 52$, are given in Table~\ref{table-combin-small-z}.

\begin{table}[ht]
\begin{center}
\caption{ The possible combinations of $D_k$- and $D_{k-1}$-branches, $15 \leq k \leq 52$, and the corresponding
 upper bounds on the  numbers of $B$-branches are presented. By obtaining the above results, the occurrence of $D_k$-branches
 is necessary, while the $D_{k-1}$-branches may or may not occur in a minimal-ABC tree.
 } \label{table-combin-small-z}
\begin{tabular}{cccccc}
\hline
combination  & upper & combination  & upper   & combination  & upper  \\
                     &  bound &                     & bound   &                      & bound  \\
\hline
 $D_{51}$ and $D_{52}$ & $916$  & $D_{50}$ and $D_{51}$ &  $908$  & $D_{49}$ and $D_{50}$ & $900$  \\
 $D_{48}$ and $D_{49}$ &  $824$  &  $D_{47}$ and $D_{48}$ &  $723$  & $D_{46}$ and $D_{47}$ & $688$  \\
 $D_{45}$ and $D_{46}$ &  $661$ &$D_{44}$ and $D_{45}$ &  $638$   & $D_{43}$ and $D_{44}$ & $618$  \\
 $D_{42}$ and $D_{43}$ &  $601$ & $D_{41}$ and $D_{42}$ &  $587$   & $D_{40}$ and $D_{41}$ & $574$  \\
 $D_{39}$ and $D_{40}$ &  $563$ & $D_{38}$ and $D_{39}$ &  $553$  & $D_{37}$ and $D_{38}$  & $543$ \\
$D_{36}$ and $D_{37}$ &   $533$ & $D_{35}$ and $D_{36}$ &  $525$  & $D_{34}$ and $D_{35}$  & $519$ \\
$D_{33}$ and $D_{34}$ &  $513$  & $D_{32}$ and $D_{33}$ &  $507$  & $D_{31}$ and $D_{32}$  &  $500$ \\
$D_{30}$ and $D_{31}$ &  $496$  & $D_{29}$ and $D_{30}$ &  $492$  & $D_{28}$ and $D_{29}$  &  $486$ \\
$D_{27}$ and $D_{28}$ &  $482$  & $D_{26}$ and $D_{27}$ &  $478$  & $D_{25}$ and $D_{26}$  &  $474$ \\
$D_{24}$ and $D_{25}$ &  $472$  & $D_{23}$ and $D_{24}$ &  $469$  & $D_{22}$ and $D_{23}$  &  $467$ \\
$D_{21}$ and $D_{22}$ &  $465$  & $D_{20}$ and $D_{21}$ &  $462$  & $D_{19}$ and $D_{20}$  &  $460$ \\
$D_{18}$ and $D_{19}$ &  $457$  & $D_{17}$ and $D_{18}$ &  $454$  & $D_{16}$ and $D_{17}$  &  $451$ \\
$D_{15}$ and $D_{16}$ & $449$ & $D_{15}$  & $426$  & & \\
 \hline
  \end{tabular}
  \end{center}
  \end{table}

\noindent
{\em Case 2: There exists some $D$-branch containing $B_2$-, or $B_3^{**}$-branch}.

In this case, there is exactly one $D$-branch containing one or two $B_2$-branches, or one $D$-branch containing a $B_3^{**}$-branch, other $D$-branches are $D_z$-branches for some $z$. And $n_4 = 0$, from Theorems \ref{noB1B4} and \ref{noB2B4}. As above, we can get a upper bound of right-hand side of (\ref{eq-thm-bound-B3-branches-to-root}), which can be regarded as a revised version of right-hand sides of (\ref{eq-thm-bound-B3-branches-to-root-11-refine}) and (\ref{eq-thm-bound-B3-branches-to-root-11-refine-1}):
When $k \in \{ 15, \dots, 131 \} \setminus \{52,53\}$,

\begin{eqnarray} \label{eq-thm-bound-B3-branches-to-root-11-refine-add}
ABC (G') - ABC (G) &\le& (d(v)-x - 1 -n_k-n_{k-1}) ( -f(d(v),4)+f(d(v'),4) ) -f(d(v),k + 1) \nonumber \\
&&+f(d(v'),k + 1)  + n_{k-1} ( -f(d(v),k)+f(d(v'),k) ) \nonumber \\
&&+ n_k ( -f(d(v),k + 1)+f(d(v'),k+1) ) -f(d(v),4)+f(d(v'),53)  \nonumber \\
&&+6(-f(2,1)+f(d(v'),53))  +364(-f(d(v),4)+f(53,4)) \nonumber \\ 
&& +(x-365)(-f(d(v),4)+f(d(v'),4))  \nonumber \\
&=& (d(v) - 1 -n_k-n_{k-1}) ( -f(d(v),4)+f(d(v'),4) ) -f(d(v),k + 1)\nonumber \\
&&+f(d(v'),k + 1) ) + n_{k-1} ( -f(d(v),k)+f(d(v'),k) )  \nonumber \\
&&+ n_k ( -f(d(v),k + 1)+f(d(v'),k+1) ) -f(d(v),4)+f(d(v'),53)  \nonumber \\
&&+6(-f(2,1)+f(d(v'),53)) +364(-f(d(v),4)+f(53,4)) \nonumber \\
&& -365(-f(d(v),4)+f(d(v'),4)),
\end{eqnarray}
where $n_i$ is the number of $D_i$-branches, $0 \leq n_i \leq n_i^{max}$, $i \in \{ 15, \dots, 131 \} \setminus 52$. In particular,
when $k = 53$,

\begin{eqnarray} \label{eq-thm-bound-B3-branches-to-root-11-refine-1-add}
ABC (G') - ABC (G) &\le& (d(v)-x- 1 -n_k) ( -f(d(v),k)+f(d(v'),k) ) -f(d(v),k + 1)  \nonumber \\
&&+f(d(v'),k+1) + n_k ( -f(d(v),k + 1)+f(d(v'),k+1) ) -f(d(v),4) \nonumber \\
&& +f(d(v'),53) +6(-f(2,1)+f(d(v'),53)) +364(-f(d(v),4)+f(53,4)) \nonumber \\
&& +(x-365)(-f(d(v),4)+f(d(v'),4)) \nonumber \\
&=& (d(v)-x-1-n_{53}) ( -f(d(v),53)+f(d(v'),53) ) \nonumber \\
&&+ (n_{53} + 1) ( -f(d(v),54)+f(d(v'),54) ) -f(d(v),4)+f(d(v'),53) \nonumber \\
&& +6(-f(2,1)+f(d(v'),53)) +364(-f(d(v),4)+f(53,4)) \nonumber \\
&& +(x-365)(-f(d(v),4)+f(d(v'),4)).
\end{eqnarray}
Analogous to the arguments mentioning in Case $1$, we can get the corresponding values of the maximum numbers of $B$-branches attached to the root, which is listed in Tables \ref{tbl-maxB-2-1} and \ref{table-combin-small-z-1}.
%
\begin{table}[ht]
\begin{center}
 \caption{ The corresponding upper bounds on the  numbers of $B$-branches, when the maximal size of $D$-branches in $G$ is $k$, $53 \leq k \leq 131$, and $G$ contains some $B_2$-, or $B_3^{**}$-branch.} \label{tbl-maxB-2-1}
\begin{tabular}{cccccc}
\hline
$k$  & upper bound & $k$   &   upper bound  & $k$   & upper bound  \\
\hline
 $53$ &  $919$ & $54$ &  $882$ & $55$ & $798$  \\
 $56$ &  $749$  &  $57$ &  $718$ & $58$ & $694$  \\
 $59$ &  $673$  & $60$ &  $655$ & $61$ & $640$   \\
 $62$ &  $628$ & $63$ &  $617$ & $64$ & $607$  \\
 $65$ &  $599$  & $66$ &  $593$ & $67$ & $586$  \\
 $68$ &  $579$  & $69$ &  $575$ & $70$ & $573$  \\
 $71$ &  $571$  & $72$ &  $566$ & $73$ & $564$  \\
 $74$ &  $562$  & $75$ &  $557$ & $76$ &  $554$  \\
 $77$ &  $555$  & $78$ &  $553$ & $79$ &  $550$  \\
 $80$ &  $551$  & $81$ &  $548$ & $82$ &  $545$  \\
 $83$ &  $546$  & $84$ &  $547$ & $85$ &  $548$  \\
 $86$ &  $544$  &  $87$ &  $541$ & $88$ &  $538$ \\
 $89$ &  $535$  & $90$ &  $535$ & $91$ &  $536$  \\
 $92$ &  $537$  & $93$ &  $537$ & $94$ &  $538$  \\
 $95$ &  $538$  & $96$ &  $539$ & $97$ &  $540$  \\
 $98$ &  $540$  & $99$ &  $541$ & $100$ &  $542$  \\
 $101$ &  $542$  & $102$ &  $543$ & $103$ &  $543$  \\
 $104$ &  $544$  & $105$ &  $544$ & $106$ &  $545$  \\
 $107$ &  $546$  & $108$ &  $546$ & $109$ &  $547$  \\
 $110$ &  $547$  & $111$ &   $548$ & $112$ &  $548$  \\
 $113$ &  $549$  & $114$ &   $549$ & $115$ &  $550$  \\
 $116$ &  $550$  & $117$ &   $551$ & $118$ &  $551$  \\
 $119$ &  $552$  & $120$ &   $553$ & $121$ &  $553$  \\
 $122$ &  $554$  & $123$ &   $554$ & $124$ &  $554$  \\
 $125$ &  $555$  & $126$ &   $555$ & $127$ &  $556$  \\
 $128$ &  $556$  & $129$ &   $557$ & $130$ &  $557$  \\
  $131$ &  $558$   &  &  & & \\
 \hline
  \end{tabular}
  \end{center}
  \end{table}
  %
\begin{table}[ht]
\begin{center}
 \caption{ The corresponding
 upper bounds on the  numbers of $B$-branches, when the maximal size of $D$-branches in $G$ is $k$, $15 \leq k \leq 52$, and $G$ contains some $B_2$-, or $B_3^{**}$-branch.
 }  \label{table-combin-small-z-1}
 \begin{tabular}{cccccc}
\hline
$k$  & upper bound & $k$   &   upper bound  & $k$   & upper bound  \\
\hline
 $52$ & $916$  & $51$ &  $908$  & $50$ & $900$  \\
 $49$ &  $824$  &  $48$ &  $724$  & $47$ & $688$  \\
 $46$ &  $661$ &$45$ &  $639$   & $44$ & $618$  \\
 $43$ &  $602$ & $42$ &  $588$   & $41$ & $575$  \\
 $40$ &  $564$ & $39$ &  $554$  & $38$  & $544$ \\
$37$ &   $534$ & $36$ &  $526$  & $35$  & $520$ \\
$34$ &  $514$  & $33$ &  $508$  & $32$  &  $501$ \\
$31$ &  $497$  & $30$ &  $493$  & $29$  &  $487$ \\
$28$ &  $483$  & $27$ &  $479$  & $26$  &  $475$ \\
$25$ &  $472$  & $24$ &  $470$  & $23$  &  $468$ \\
$22$ &  $465$  & $21$ &  $463$  & $20$  &  $460$ \\
$19$ &  $457$  & $18$ &  $454$  & $17$  &  $451$ \\
$16$ & $448$ & $15$  & $426$  & & \\
 \hline
  \end{tabular}
  \end{center}
  \end{table}

\subsection{The existence of $B_4$-branches--the negativity of right-hand side of (\ref{eq-lemma-no-B4-10})} \label{appendix-exist-B4}

Let $n_z$ be the number of $D_z$-branches adjacent to the root vertex $v$, for $49 \le z \le 54$.
Further evaluating the right-hand side of (\ref{eq-lemma-no-B4-10}), we get the following upper bound on $ABC (G') - ABC (G)$:
If $G$ contains a combination of $D_{z}$- and $D_{z+1}$-branches, then
\begin{eqnarray}  \label{eq-lemma-no-B4-10-1-all-add-1}
ABC (G') - ABC (G) &\le& -(2z-3) f(d(v),z+1)  -(2z-3) z \, f(z+1,4) -  f(d(v),5) - f(4,2)     \nonumber \\
&&  - f(2,1)   + (2z-1) f(d(v)+1,z) + (2z-1) (z-1) f(z,4)   \nonumber \\
&& + n_3 (-f(d(v),4)+f(d(v)+1,4)) + (n_4 - 1) (-f(d(v),5)+f(d(v)+1,5))   \nonumber \\
&&  + (d(v) - 2z + 3 - n_3 - n_4) (- f(d(v),z+1) + f(d(v) + 1, z+1)),
\end{eqnarray}
and
if $G$ contains a combination of $D_{z-1}$- and $D_{z}$-branches, then
\begin{eqnarray}  \label{eq-lemma-no-B4-10-1-all-add-2}
ABC (G') - ABC (G)  &\le& -(2z-3) f(d(v),z+1)  -(2z-3) z \, f(z+1,4) -  f(d(v),5) - f(4,2) \nonumber \\
&&    - f(2,1)  + (2z-1) f(d(v)+1,z)   + (2z-1) (z-1) f(z,4)  \nonumber \\
&&  + n_3 (-f(d(v),4)+f(d(v)+1,4)) + (n_4 - 1) (-f(d(v),5)+f(d(v)+1,5))  \nonumber \\
&& + (d(v) - 2z + 3 - n_3 - n_4) (- f(d(v),z) + f(d(v) + 1, z)).
\end{eqnarray}

Recall from Tables \ref{tbl-maxB-2} and \ref{table-combin-small-z}, $n_3 + n_4$ is bounded from above, say $n_3 + n_4 \le M_{z,z+1}$ when $G$ contains a combination of $D_{z}$- and $D_{z+1}$-branches (e.g., $M_{52,53} = 919$). Observe that the coefficient of $n_3$ in both (\ref{eq-lemma-no-B4-10-1-all-add-1}) and (\ref{eq-lemma-no-B4-10-1-all-add-2}) is positive, i.e.,
\beq
-f(d(v),4)+f(d(v)+1,4) &>& - f(d(v),z) + f(d(v) + 1, z) \nonumber \\ 
                                  & > & - f(d(v),z+1) + f(d(v) + 1, z+1)  \nonumber
\eeq
from Proposition~\ref{pro-10}, which implies that both inequalities increase in $n_3$. Substitute $n_3 = M_{z,z+1} - n_4$ in (\ref{eq-lemma-no-B4-10-1-all-add-1}) and $n_3 = M_{z-1,z} - n_4$ in (\ref{eq-lemma-no-B4-10-1-all-add-2}) (the upper bounds of $n_3$), we get: If $G$ contains a combination of $D_{z}$- and $D_{z+1}$-branches, then
\begin{eqnarray}  \label{eq-lemma-no-B4-10-1-all-add-3-1}
ABC (G') - ABC (G)  &\le& -(2z-3) f(d(v),z+1)  -(2z-3) z \, f(z+1,4) -  f(d(v),5) - f(4,2)   \nonumber \\
&&   - f(2,1)  + (2z-1) f(d(v)+1,z)  + (2z-1) (z-1) f(z,4)   \nonumber \\
&& + (M_{z,z+1} - n_4) (-f(d(v),4)+f(d(v)+1,4))   \nonumber \\
&& + (n_4 - 1) (-f(d(v),5)+f(d(v)+1,5))  \nonumber \\
&& + (d(v) - 2z + 3 - M_{z,z+1}) (- f(d(v),z+1) + f(d(v) + 1, z+1)),
\end{eqnarray}
and
if $G$ contains a combination of $D_{z-1}$- and $D_{z}$-branches, then
\begin{eqnarray}  \label{eq-lemma-no-B4-10-1-all-add-4-1}
ABC (G') - ABC (G) &\le& -(2z-3) f(d(v),z+1)  -(2z-3) z \, f(z+1,4) -  f(d(v),5) - f(4,2)   \nonumber \\
&&  - f(2,1)  + (2z-1) f(d(v)+1,z)   + (2z-1) (z-1) f(z,4)  \nonumber \\
&&  + (M_{z-1,z} - n_4) (-f(d(v),4)+f(d(v)+1,4))  \nonumber \\
&& + (n_4 - 1) (-f(d(v),5)+f(d(v)+1,5))   \nonumber \\
&& + (d(v) - 2z + 3 - M_{z-1,z}) (- f(d(v),z) + f(d(v) + 1, z)).
\end{eqnarray}
Further, the coefficient of $n_4$ in both (\ref{eq-lemma-no-B4-10-1-all-add-3-1}) and (\ref{eq-lemma-no-B4-10-1-all-add-4-1}) is negative, i.e.,
$$
-f(d(v),5)+f(d(v)+1,5) < -f(d(v),4)+f(d(v)+1,4)
$$
from Proposition~\ref{pro-10}, which implies that both inequalities decrease in $n_4$. Set $n_4 = 1$ (the lower bound of $n_4$), it follows that: If $G$ contains a combination of $D_{z}$- and $D_{z+1}$-branches, then
\begin{eqnarray}  \label{eq-lemma-no-B4-10-1-all-add-3}
ABC (G') - ABC (G) &\le& -(2z-3) f(d(v),z+1)  -(2z-3) z \, f(z+1,4) -  f(d(v),5) - f(4,2)   \nonumber \\
&&  - f(2,1)  + (2z-1) f(d(v)+1,z)   + (2z-1) (z-1) f(z,4)   \nonumber \\
&& + (M_{z,z+1} - 1) (-f(d(v),4)+f(d(v)+1,4))  \nonumber \\
&& + (d(v) - 2z + 3 - M_{z,z+1}) (- f(d(v),z+1) + f(d(v) + 1, z+1)),
\end{eqnarray}
and
if $G$ contains a combination of $D_{z-1}$- and $D_{z}$-branches, then
\begin{eqnarray}  \label{eq-lemma-no-B4-10-1-all-add-4}
ABC (G') - ABC (G) &\le& -(2z-3) f(d(v),z+1)  -(2z-3) z \, f(z+1,4) -  f(d(v),5) - f(4,2)   \nonumber \\
&&   - f(2,1)  + (2z-1) f(d(v)+1,z)  + (2z-1) (z-1) f(z,4)     \nonumber \\
&& + (M_{z-1,z} - 1) (-f(d(v),4)+f(d(v)+1,4))  \nonumber \\ 
&& + (d(v) - 2z + 3 - M_{z-1,z}) (- f(d(v),z) + f(d(v) + 1, z)).
\end{eqnarray}

In the sequel, we consider the five possible cases separately.

\smallskip

\noindent
$\bullet$ {\it $G$ contains a combination of $D_{53}$- and $D_{54}$-branches}.

In this case it holds that  $n_3 + n_4 \leq 882$ (see Table~{\ref{tbl-maxB-2}}), i.e., $M_{53,54} = 882$.
It holds also that $n_{53} \leq 260$ and  $n_{54} \leq 130$ (see Table~{\ref{table-large-z}}).
It implies that $d(v) \le 882 + 260 + 130 = 1272$.
A straightforward verification shows that the right-hand side of (\ref{eq-lemma-no-B4-10-1-all-add-3}) when $z = 53$, and the right-hand side of (\ref{eq-lemma-no-B4-10-1-all-add-4}) when $z = 54$, are negative for $1228 \le d(v) \le 1272$.

\smallskip

\noindent
$\bullet$ {\it $G$ contains a combination of $D_{52}$- and $D_{53}$-branches}.

\smallskip

Here, we have  $n_3 + n_4 \leq 919$  (see Table~{\ref{tbl-maxB-2}}), i.e., $M_{52,53} = 919$, $n_{53} \leq 260$ (from Table~{\ref{table-large-z}}).

If $n_{52} < 2\cdot 52 - 3 = 101$, then it must hold that $n_{53} \ge 2 \cdot 53 - 3 = 103$. Then, we choose $z = 53$, and note that $d(v) \le 919 + 100 + 260 = 1279$.
A straightforward verification shows that
the right-hand side of (\ref{eq-lemma-no-B4-10-1-all-add-4}) is negative for $1228 \le d(v) \le 1279$.

If $n_{52} \geq 101$, then we choose $z=52$.
From (\ref{eq-lemma-no-B4-10-1-all-add-3}), the change of the ABC index in this case is
\begin{eqnarray}  \label{eq-lemma-no-B4-20-4}
%
ABC (G') - ABC (G) &\le& - 101 f(d(v),53)  -5252 f(53,4) -  f(d(v),5) - f(4,2) - f(2,1)   \nonumber \\
&& + 103 f(d(v)+1,52)  + 5253 f(52,4) + 918(-f(d(v),4)+f(d(v)+1,4))    \nonumber \\
&& + (d(v)-1020) (- f(d(v),53) + f(d(v) + 1, 53)).
\end{eqnarray}
At this stage, we consider the terms in right-hand side of \eqref{eq-lemma-no-B4-20-4} related to $d(v)$, which are
\begin{eqnarray} \label{eq-lemma-no-B4-20-5}
&&- 101 f(d(v),53) -  f(d(v),5) + 103 f(d(v)+1,52) + 918 (-f(d(v),4)+f(d(v)+1,4))  \nonumber \\
&& + (d(v)-1020) (- f(d(v),53) + f(d(v) + 1, 53)) \nonumber \\
&=&102 (-f(d(v),53) + f(d(v)+1,52)) + 918 (-f(d(v),4)+f(d(v)+1,4)) \nonumber \\
&&+ (d(v)-1020) f(d(v) + 1, 53) - (d(v)-1021)  f(d(v),53) +f(d(v)+1,52) -  f(d(v),5) \nonumber \\
&<&102 (-f(d(v),53) + f(d(v)+1,52))  + (d(v)-1020) f(d(v) + 1, 53) - (d(v)-1021)  f(d(v),53)  \nonumber \\
&& +f(d(v)+1,52) -  f(d(v),5).  \nonumber
\end{eqnarray}

By Proposition \ref{pro-10}, we have that
$- f(d(v),53) + f(d(v)+1,52)$
increases in $d(v)$, and thus
\begin{eqnarray*}
- f(d(v),53) + f(d(v)+1,52)  &\le&  \lim_{d(v) \rightarrow + \infty}(- f(d(v),53) + f(d(v)+1,52))   \quad = \quad  - \sqrt{\frac{1}{53}} + \sqrt{\frac{1}{52}}.
\end{eqnarray*}
It is also not hard to verify that both
$(d(v)-1020) f(d(v) + 1, 53) - (d(v)-1021)  f(d(v),53) $
and $f(d(v) + 1, 52) - f(d(v), 5)$
decrease in $d(v) \ge 3$.
Thus, we get a weaker upper bound than that in \eqref{eq-lemma-no-B4-20-4}:
\begin{eqnarray}  \label{eq-lemma-no-B4-20-6}
ABC (G') - ABC (G) &<&    102 \left( - \sqrt{\frac{1}{53}} + \sqrt{\frac{1}{52}} \right) + (d(v) - 1020) f(d(v) + 1, 53) - (d(v) - 1021) f(d(v), 53)   \nonumber \\
&& + f(d(v) + 1, 52) - f(d(v), 5)  -5252 f(53,4)- 2f(2,1)+ 5253 f(52,4), \nonumber
\end{eqnarray}
which decreases in $d(v)$. In particular, when $d(v) = 1228$, the right-hand side of above inequality is equal to $-0.00201013$.
Consequently, we can conclude that  for $d (v) \ge 1228$, we can always get a negative change after applying the transformation $\mathcal{T}_{7}$.

\bigskip

\noindent
$\bullet$ {\it $G$ contains a combination of $D_{51}$- and $D_{52}$-branches}.

\smallskip

Here, it holds that  $n_3 + n_4 \leq 916$ (see Table~\ref{table-combin-small-z}), i.e., $M_{51,52} = 916$,
and $n_{51} \leq 364$, when $d(v) \geq 3249$ (see Table~{\ref{table-small-z}}).

We consider first the case $z=51$. In this case, we may assume that $n_{52} < 2 \cdot 52 - 3 = 101$. When $d(v) \geq 3249$, $n_{51} \leq 364$, so $d(v) \le 916 + 364 + 100 = 1380$, a contradiction. When $d(v) < 3249$, the right-hand side of (\ref{eq-lemma-no-B4-10-1-all-add-3}) is negative for $1228 \le d(v) < 3249$.
When $z=52$ ($n_{52} \ge 101$), the deduction is very similar to the previous case that $G$ contains a combination of $D_{52}$- and $D_{53}$-branches. Therefore, we omit the repetition of the same arguments here.

\smallskip

\noindent

\noindent
$\bullet$ {\it $G$ contains a combination of $D_{50}$- and $D_{51}$-branches}.

\smallskip

In this case, it holds that  $n_3 + n_4 \leq 908$ (see Table~\ref{table-combin-small-z}), i.e., $M_{50,51} = 908$, $n_{50} \leq 182$ when $d(v) \geq 1358$,
and  $n_{51} \leq 364$ when $d(v) \geq 3249$ (see Table~{\ref{table-small-z}}).

If $d(v) \geq 3249$, then we have $d(v) \leq 908 + 182 + 364 = 1454$, which is a contradiction. Thus $1228 \le d(v) < 3249$.

If $z = 51$, then the right-hand side of (\ref{eq-lemma-no-B4-10-1-all-add-4}) is always negative from direct calculations, for $1228 \leq d(v) < 3249$.
Next suppose that $z = 50$. Assume that $n_{51} < 2 \cdot 51 - 3 = 99$. Note that $d(v) < 1358$, otherwise, $n_{50} \leq 182$, and thus $d(v) \le 908 + 182 + 98 = 1188$, a contradiction. Moreover, $n_{50} \ge 222$, otherwise, $d(v) \le 908 + 221 + 98 = 1227$, a contradiction again.

Here we consider in addition again the transformation $\mathcal{T}_{5}$, to show that particular configurations are impossible.
Adapting to the notation and possible configurations used here, we have the following upper bound on the change of the ABC index after applying $\mathcal{T}_{5}$:
\begin{eqnarray}  \label{eq-Lemma-Dz-30-20-second}
ABC (G') - ABC (G) & \le& -x \, f(d(v),z+1) -x \, z \, f(z+1,4)   \nonumber \\
&& + 3 f(4,2) + 3 f(2,1)  + n_{k}  f(d(v)-7,k+1)  + n_{k-1}  f(d(v)-7,k)  \nonumber \\
&&  + k \, n_{k}  f(k+1,4)  +  (k-1) n_{k-1}  f(k,4) \nonumber \\
&&  + (d(v) - x - n_3 -n_4)(- f(d(v), z+2) + f(d(v)-7, z+2)) \nonumber \\
&& + n_3 (-f(d(v),4)+f(d(v)-7,4))+n_4 (-f(d(v),5)+f(d(v)-7,5)). \qquad \quad
\end{eqnarray}
It is easy to see that the right-hand side of (\ref{eq-Lemma-Dz-30-20-second}) decreases in $n_3$ and $n_4$, respectively (just by noting that the coefficients about $n_3$ and $n_4$ are both negative, from Proposition~\ref{pro-20}). So the worst case is when $n_3 = 0$ and $n_4 = 1$ (the possible minimum values of $n_3$ and $n_4$).
Here, as in the Appendix \ref{appendix-upper_bound-2} (Table~\ref{table-small-z}), we set $z =50$, $x = 222$, $k = 52$ (so $n_{52}=136$ and $n_{51}=79$).
Together with $n_3 = 0$ and $n_4 = 1$, one can check the negativity of right-hand side of (\ref{eq-Lemma-Dz-30-20-second}) directly, for $1228 \le d(v) < 1358$.

\noindent
$\bullet$ {\it $G$ contains a combination of $D_{49}$- and $D_{50}$-branches}.

\smallskip

Here it holds that  $n_3 + n_4 \leq 900$ (see Table~\ref{table-combin-small-z}), i.e., $M_{49,50} = 900$.
We have also that  $n_{49} \leq 121$ when $d(v) \geq 825$, and $n_{50} \leq 182$ when $d(v) \geq 1358$ (see Table~{\ref{table-small-z}}).
Thus, when $d(v) \geq 1358$, then $d(v) \leq 900 + 121 + 182 = 1203$, which is a contradiction to the assumption that
$d(v) \geq 1358$. Thus $1228 \le d(v) < 1358$.

If $n_{50} < 2 \cdot 50 - 3 = 97$, then $d(v) \le 900 + 121 + 96 = 1117$, a contradiction.
So we must have $n_{50} \ge 97$, i.e., $z = 50$.

We claim that $n_{49} \le 94$. Otherwise, in the right-hand side of (\ref{eq-Lemma-Dz-30-20-second}), set $z = 49$, $x = 95$, $k = 53$ (so $n_{53} = 80$ and $n_{52} = 8$), together with $n_3 = 0$ and $n_4 = 1$, one can find that the right-hand side of (\ref{eq-Lemma-Dz-30-20-second}) is negative from direct calculations, for $1228 \le d(v) < 1358$.
Now from $n_{49} \le 94$, we further have $n_{50} \ge 234$, otherwise, $d(v) \le 900 + 94 + 233 = 1227$, a contradiction. In last case, we have shown that when $n_{50} \ge 222$, the right-hand side of (\ref{eq-Lemma-Dz-30-20-second}) would be negative for $1228 \le d(v) < 1358$. Applying it here, we can also deduce a negative change of ABC index (since $n_{50} \ge 234 > 222$).

\subsection{$B$-branches attached to the root--the negativity of right-hand side of (\ref{eq-lemma-no-B-branches-to-root})} \label{appendix-no-B-branches-to-root}

From Proposition \ref{pro-20}, $- f(d(v),d(v_i)) + f(d(v) - x, d(v_i))$ increases in $d(v_i)$,
thus the worst case is when $d(v_i)$ is chosen as large as possible.
Now we get further upper bounds on $ABC (G') - ABC (G)$: If $G$ contains $D$-branches of size $z+1$, in addition to $D_z$-branches, then
\begin{eqnarray}  \label{eq-lemma-no-B-branches-to-root-1}
ABC (G') - ABC (G) &\le& x( -f(d(v),z+1)+f(d(v)-x,z+2)) +x z (-f(z+1,4)+f(z+2,4))    \nonumber \\
&& +x( -f(d(v),4)+f(z+2,4))  \nonumber \\
&& + (d(v)-2x) (- f(d(v),z+2)+f(d(v)-x,z+2)),
\end{eqnarray}
otherwise,
\begin{eqnarray}  \label{eq-lemma-no-B-branches-to-root-1-1}
ABC (G') - ABC (G) &\le& x( -f(d(v),z+1)+f(d(v)-x,z+2)) +x z (-f(z+1,4)+f(z+2,4))   \nonumber \\
&& +x( -f(d(v),4)+f(z+2,4))   \nonumber \\
&& + (d(v)-2x) (- f(d(v),z+1)+f(d(v)-x,z+1)),
\end{eqnarray}
in either case, $1 \le x\le 919$, from Lemma \ref{thm-bound-B-branches-to-root}.

If there is no $D_{52}$-branch in $G$, then no matter $z = 50,51$ or $53$, a direct check would show that the right-hand side of (\ref{eq-lemma-no-B-branches-to-root-1}) or (\ref{eq-lemma-no-B-branches-to-root-1-1}) is always negative, since $d(v)$ is bounded (from Tables \ref{table-large-z} and \ref{table-small-z}).
In the remaining cases, $D_{52}$-branches must occur, so we assume that $z = 52$ in the following.  More precisely, we may assume that $G$ contains at least $x$ $D_{52}$-branches, where $1 \le x\le 919$.

First assume that there are $D$-branches of size $53$ in $G$.
Set $z = 52$ in (\ref{eq-lemma-no-B-branches-to-root-1}), it leads to
\begin{eqnarray}  \label{eq-lemma-no-B-branches-to-root-2}
ABC (G') - ABC (G) &\le& x( -f(d(v),53)+f(d(v)-x,54)) +52x (-f(53,4)+f(54,4))   \nonumber \\
&& +x( -f(d(v),4)+f(54,4)) + (d(v)-2x) (- f(d(v),54)+f(d(v)-x,54)).   \nonumber
\end{eqnarray}
Recall that there can be at most $260$ $D_{53}$-branches (Table \ref{table-large-z}), and at most one $D_{53}^{**}$-branch, thus a more precise version of above estimation of
the difference $ABC (G') - ABC (G)$ should be
\begin{eqnarray}  \label{eq-lemma-no-B-branches-to-root-3}
ABC (G') - ABC (G) &\le& x( -f(d(v),53)+f(d(v)-x,54)) +52x (-f(53,4)+f(54,4)) \nonumber \\
&&  +x( -f(d(v),4)+f(54,4))  + 261 (- f(d(v),54)+f(d(v)-x,54)) \nonumber \\
&& + (d(v)-2x-261) (- f(d(v),53)+f(d(v)-x,53)).
\end{eqnarray}

The terms in the right-hand side of \eqref{eq-lemma-no-B-branches-to-root-3} related to $d(v)$ are
\begin{eqnarray*}
&&x( -f(d(v),53)+f(d(v)-x,54))  - x \, f(d(v),4)  + 261(- f(d(v),54)+f(d(v)-x,54))  \\
&&+ (d(v)-2x - 261)(- f(d(v),53)+f(d(v)-x,53)) \\
&=& x( -f(d(v),53)+f(d(v)-x,54)) + 261(- f(d(v),54)+f(d(v)-x,54)) + x ( - f(d(v),4) + f(d(v),53)) \\
&& - (d(v)-x - 261) f(d(v),53) + (d(v)-2x - 261) f(d(v)-x,53).
\end{eqnarray*}
By Proposition~\ref{pro-20}, we have the following observations:
\begin{itemize}
\item

$-f(d(v),53)+f(d(v)-x,54)$ decreases in $d(v)$;

\item
$- f(d(v),54)+f(d(v)-x,54)$ decreases in $d(v)$;

\item
$- f(d(v),4) + f(d(v),53)$ decreases in $d(v)$.

\end{itemize}
It is not hard to verify that
$- (d(v)-x - 261) f(d(v),53) + (d(v)-2x - 261) f(d(v)-x,53)$
increases in $d(v)$, and thus
\begin{eqnarray*}
&&- (d(v)-x - 261) f(d(v),53) + (d(v)-2x - 261) f(d(v)-x,53)\\
&\le& \lim_{d(v) \rightarrow + \infty}(- (d(v)-x - 261) f(d(v),53) + (d(v)-2x - 261) f(d(v)-x,53) ) \\ 
&= &  - \frac{x}{\sqrt{53}}.
\end{eqnarray*}
Now, we get another upper bound on $ABC (G') - ABC (G)$:
\begin{eqnarray}  \label{eq-lemma-no-B-branches-to-root-4}
ABC (G') - ABC (G) &\le& x( -f(d(v),53)+f(d(v)-x,54)) + 261(- f(d(v),54)+f(d(v)-x,54))\nonumber \\
&& + x ( - f(d(v),4) + f(d(v),53)) - \frac{x}{\sqrt{53}}  \nonumber \\
&&  +52x (-f(53,4)+f(54,4))+x \, f(54,4),
\end{eqnarray}
which decreases in $d(v)$. When $d(v) = 4199$, thus so is for $d(v) \ge 4199$, the right-hand side of \eqref{eq-lemma-no-B-branches-to-root-4} is always negative for $1 \le x \le 919$.
As to $2956 \le d(v) < 4199$, we resort to the right-hand side of (\ref{eq-lemma-no-B-branches-to-root-3}), which is always negative for $1 \le x \le 919$.

Next assume that there is no $D$-branch of size $53$. In this case, we use (\ref{eq-lemma-no-B-branches-to-root-1-1}) by setting $z = 52$, i.e.,
\begin{eqnarray}  \label{eq-lemma-no-B-branches-to-root-6}
ABC (G') - ABC (G) &\le& x( -f(d(v),53)+f(d(v)-x,54)) + 52x  (-f(53,4)+f(54,4))  \nonumber \\
&&  +x( -f(d(v),4)+f(54,4))  \nonumber \\
&& + (d(v)-2x) (- f(d(v),53)+f(d(v)-x,53)).
\end{eqnarray}
As above, we can get
\begin{eqnarray}  \label{eq-lemma-no-B-branches-to-root-5}
ABC (G') - ABC (G) &\le& x( -f(d(v),53)+f(d(v)-x,54)) + x ( - f(d(v),4) + f(d(v),53))\nonumber \\
&& - \frac{x}{\sqrt{53}} +52x (-f(53,4)+f(54,4))+x \, f(54,4),
\end{eqnarray}
which decreases in $d(v)$. For $d(v) \ge 3990$, the right-hand side of \eqref{eq-lemma-no-B-branches-to-root-5} is always negative for $1 \le x \le 919$ (actually we need only to check the case when $d(v) = 3990$). For the remaining case $2956 \le d(v) < 3990$, we use the right-hand side of (\ref{eq-lemma-no-B-branches-to-root-6}), which is again always negative for $1 \le x \le 919$.

\subsection{The upper bound of size of $D_{z,1}^2$-branches--the negativity of right-hand side of (\ref{eq-Lemma-Dz*-10-00})} \label{appendix-nonexist-B2-1}

Clearly, every neighbor of $v$ in $\bar{G}$ is of degree at least $4$, thus
$$
\sum_{x v \in E(\bar{G})} ( - f(d(v), d(x)) + f (d(v) + 3, d(x))) \le (d(v) - 1) ( - f(d(v), 4) + f (d(v) + 3, 4)),
$$
by Proposition~\ref{pro-10}. Thus,
we get
\begin{eqnarray} \label{eq-Lemma-Dz*-10-00-1}
%
ABC (G') - ABC (G) &\le& (d(v) - 1) ( - f(d(v), 4) + f (d(v) + 3, 4)) - f (d(v),z+1)   \nonumber \\
&&+ 2 f \left( d(v) + 3, \frac{z-1}{2} \right)  - f (z+1,3) + 2 f ( d(v) + 3, 5)   - (z-1) f (z+1,4)  \nonumber \\
&& + (z-3) f \left( \frac{z-1}{2}, 4 \right)  \nonumber \\
&=& - (d(v) - 1) f(d(v), 4) + ( d(v) + 2) f (d(v) + 3, 4) \nonumber \\
&& + 2 ( -  f (d(v) + 3, 4) +   f ( d(v) + 3, 5) )  -  f (d(v) + 3, 4)  \nonumber \\
&&+ f \left( d(v) + 3, \frac{z-1}{2} \right)  - f (d(v),z+1) +  f \left( d(v) + 3, \frac{z-1}{2} \right)  \nonumber \\
&&- f (z+1,3)  - (z-1) f (z+1,4) + (z-3) f \left( \frac{z-1}{2}, 4 \right).
\end{eqnarray}

Let us focus on the terms related to $d(v)$ on the right-hand side of (\ref{eq-Lemma-Dz*-10-00-1}):
\begin{itemize}

\item
It is not hard to verify that $- (d(v) - 1) f(d(v), 4) + ( d(v) + 2) f (d(v) + 3, 4)$ decreases in $d(v)$.

\item
Both $- f (d(v) + 3, 4) +  f ( d(v) + 3, 5)$ and $-  f (d(v) + 3, 4) + f \left( d(v) + 3, \frac{z-1}{2} \right)$
decrease in $d(v)$ by Proposition~\ref{pro-20}.

\item
$- f (d(v),z+1) +  f \left( d(v) + 3, \frac{z-1}{2} \right)$ increases in $d(v)$ by Proposition~\ref{pro-10},
which implies that
\begin{eqnarray*}
&& - f (d(v),z+1) +  f \left( d(v) + 3, \frac{z-1}{2} \right) \\ 
&\le& \lim_{d(v) \rightarrow + \infty}\left(- f (d(v),z+1) +  f \left( d(v) + 3, \frac{z-1}{2} \right)\right) \quad  = \quad -\sqrt{\frac{1}{z+1}} + \sqrt{\frac{1}{\frac{z-1}{2}}}.
\end{eqnarray*}

\end{itemize}
Therefore, a (weaker) upper bound on $ABC (G') - ABC (G)$ follows:
\begin{eqnarray} \label{eq-Lemma-Dz*-10-00-2}
%
ABC (G') - ABC (G) &\le&  - (d(v) - 1) f(d(v), 4) + ( d(v) + 2) f (d(v) + 3, 4)  \nonumber \\
&& + 2 ( - f (d(v) + 3, 4) +  f ( d(v) + 3, 5) ) \nonumber \\
&& -  f (d(v) + 3, 4) + f \left( d(v) + 3, \frac{z-1}{2} \right) -\sqrt{\frac{1}{z+1}} + \sqrt{\frac{1}{\frac{z-1}{2}}}  \nonumber \\
&&- f (z+1,3)  - (z-1) f (z+1,4) + (z-3) f \left( \frac{z-1}{2}, 4 \right),
\end{eqnarray}
which decreases in $d(v) \ge z + 1$.

Setting $d(v) = z + 1$ in the right-hand side of (\ref{eq-Lemma-Dz*-10-00-2}), it leads to
\begin{eqnarray} \label{eq-Lemma-Dz*-10-00-3}
%
ABC (G') - ABC (G) &\le&  - z \, f(z + 1, 4) + ( z + 3) f (z + 4, 4)  + 2 (-  f (z + 4, 4) +  f ( z + 4, 5) )  \nonumber \\
&& -  f (z + 4, 4) + f \left( z + 4, \frac{z-1}{2} \right) -\sqrt{\frac{1}{z+1}} + \sqrt{\frac{1}{\frac{z-1}{2}}} \nonumber \\
&& - f (z+1,3)  - (z-1) f (z+1,4) + (z-3) f \left( \frac{z-1}{2}, 4 \right).
\end{eqnarray}
The right-hand side  of (\ref{eq-Lemma-Dz*-10-00-3}) is negative when $113 \le z \le 131$.
When $99 \le z \le 111$ ($z$ is odd), we set $d(v) = 173$ in (\ref{eq-Lemma-Dz*-10-00-2}), and then the right-hand side is negative for $99 \le z \le 111$.
For the remaining cases $99 \le z \le 111$ and $z + 1 \le d(v) \le 172$, we resort to the right-hand side of (\ref{eq-Lemma-Dz*-10-00-1}), which can be also verified that is  negative.

\subsection{The existence of $D_{z,2}^2$-branches--the negativity of right-hand side of (\ref{eq-no-D_2,2^2-10})} \label{appendix-nonexist-B2-2}

For each $y$ such that $y v \in E(\bar{G})$,
by Proposition \ref{pro-10}, $- f(d(v),d(y)) + f(d(v) + 2, d(y))$ decreases in $d(y)$.
Thus, the worst case is when $d(y)$ as small as possible, i.e., $d(y) = 4$.
Now we get an upper bound on $ABC (G') - ABC (G)$:
\begin{eqnarray}  \label{eq-no-D_2,2^2-10-1}
%
ABC (G') - ABC (G) &\le& (d(v) - x - 1) ( - f(d(v), 4) + f (d(v) + 2, 4)) -x \, f(d(v),z+1)  \nonumber \\
&& -x \, z \, f(z+1,4)  - f(d(v),k+1)  -(k-2) f(k+1,4) -2f(k+1,3) \nonumber \\
&& + (x+2) f(d(v)+2,z)  + (x+2)(z-1) f(z,4) +f(d(v)+2, 5)  \nonumber \\
&=& - (d(v) - x - 1) f (d(v), 4) + (d(v) - x + 1) f(d(v) + 2, 4) \nonumber \\
&& + 2 ( - f(d(v) + 2, 4) + f (d(v), z + 1))  - f(d(v), k + 1) + f (d(v) + 2, 5)  \nonumber \\
&& + (x+2) (- f (d(v), z+ 1) + f (d(v) + 2, z))  -x \, z \, f(z+1,4)  \nonumber \\
&&   -(k-2) f(k+1,4) -2f(k+1,3)  + (x+2)(z-1) f(z,4).
\end{eqnarray}

It is easy to verify that $- (d(v) - x - 1) f (d(v), 4) + (d(v) - x + 1) f(d(v) + 2, 4)$ decreases in $d(v)$, for each fixed positive $x$. By Propositions \ref{pro-10} and \ref{pro-20},
\begin{itemize}

\item

$- f(d(v) + 2, 4) + f (d(v), z + 1)$ decreases in $d(v)$;

\item

$- f(d(v), k + 1) + f (d(v) + 2, 5)$ increases in $d(v)$, and thus,
\begin{eqnarray*}
- f(d(v), k + 1) + f (d(v) + 2, 5)
&\le& \lim_{d(v) \rightarrow + \infty}(- f(d(v), k + 1) + f (d(v) + 2, 5))  \quad = \quad -\sqrt{\frac{1}{k+1}} + \sqrt{\frac{1}{5}}; 
\end{eqnarray*}

\item

$- f (d(v), z+ 1) + f (d(v) + 2, z)$ increases in $d(v)$, and therefore,
\begin{eqnarray*}
- f (d(v), z+ 1) + f (d(v) + 2, z)
&\le& \lim_{d(v) \rightarrow + \infty}(- f (d(v), z+ 1) + f (d(v) + 2, z))  \quad = \quad -\sqrt{\frac{1}{z+1}} + \sqrt{\frac{1}{z}}.
\end{eqnarray*}

\end{itemize}
In conclusion, it leads to another (weaker) upper bound on $ABC (G') - ABC (G)$:
\begin{eqnarray}  \label{eq-no-D_2,2^2-10-2}
%
ABC (G') - ABC (G) &\le& - (d(v) - x - 1) f (d(v), 4) + (d(v) - x + 1) f(d(v) + 2, 4)  \nonumber \\
&&  + 2 ( - f(d(v) + 2, 4) + f (d(v), z + 1)) -\sqrt{\frac{1}{k+1}} + \sqrt{\frac{1}{5}}  \nonumber \\
&&  + (x+2) \left( -\sqrt{\frac{1}{z+1}} + \sqrt{\frac{1}{z}} \right) -x \, z \, f(z+1,4) -(k-2) f(k+1,4) \nonumber \\
&& -2f(k+1,3)  + (x+2)(z-1) f(z,4),  \nonumber
\end{eqnarray}
which decreases in $d(v) \ge 146$.

When $d(v) = 2400$, and thus $d(v) \ge 2400$, the right-hand side of  the above inequality is negative, for any combination of $x,k,z$ satisfying
$z \in [43,56]$, $k \in [34,50]$ and $x = 2z-k$
(where $x,k,z$ are all bounded). With the same constraints about $x,k,z$, when $146 \le d(v) < 2400$, the right-hand side of \eqref{eq-no-D_2,2^2-10-1} is negative again.

\subsection{The existence of $D_{z,1}^2$-branches--the negativity of right-hand side of (\ref{eq-no-Dz,1_2-10})} \label{appendix-nonexist-B2-3}

In the same manner as  in Appendix \ref{appendix-nonexist-B2-2}, we
get also here the following upper bound on $ABC (G') - ABC (G)$ after applying transformation $\mathcal{T}_{12}$:
\begin{eqnarray}  \label{eq-no-Dz,1_2-10-3}
%
ABC (G') - ABC (G) &\le& - (d(v) - x - 1) f (d(v), 4) + (d(v) - x + 4) f(d(v) + 5, 4)  \nonumber \\
&& + 5 ( - f (d(v) + 5, 4) + f (d(v), k + 1))   -x \, z \, f(z+1,4) -(k-1) f(k+1,4) \nonumber \\
&& + 6 ( - f (d(v),k+1) + f (d(v), z + 1)) + (x+6) (- f(d(v),z+1)   \nonumber \\
&&  + (x+6) (- f(d(v),z+1)+ f (d(v) + 5, z)) - f(k+1,3)  \nonumber \\
&&  - 2f(3,2)   - 2f(2,1) + (x+6)(z-1) f(z,4), \qquad
\end{eqnarray}
and further get another (weaker) upper bound which decreases in $d(v) \ge x + 1$:
\begin{eqnarray}  \label{eq-no-Dz,1_2-10-2}
ABC (G') - ABC (G) &\le& - (d(v) - x - 1) f (d(v), 4) + (d(v) - x + 4) f(d(v) + 5, 4)  \nonumber \\
&& + 5 ( - f (d(v) + 5, 4) + f (d(v), k + 1))   -x \, z \, f(z+1,4) -(k-1) f(k+1,4) \nonumber \\
&& + 6 ( - f (d(v),k+1) + f (d(v), z + 1)) + (x+6)  \left( -\sqrt{\frac{1}{z+1}} + \sqrt{\frac{1}{z}} \right)  \nonumber \\
&&  - f(k+1,3)   - 2f(3,2)   - 2f(2,1) + (x+6)(z-1) f(z,4).
\end{eqnarray}
Substituting $d(v) = 4100$ into the right-hand side of (\ref{eq-no-Dz,1_2-10-2}), it can be verified that the right-hand side of (\ref{eq-no-Dz,1_2-10-2}),
for each  $z \in \{51, 52 \}$ with constraints  $k \in [z-5,z]$ and $x = 6z - k - 5$,
is negative. Due to the monotonicity on $d (v)$,
it follows that it is also negative for $d (v) \ge 4100$. For the remaining part $x + 1 \le d(v) < 4100$, by direct calculations,
the right-hand side of (\ref{eq-no-Dz,1_2-10-3}) is negative under the same constraints on $x$, $k$, and $z$.

\subsection{The upper bound of size of $D_z^{**}$-branches--the negativity of right-hand side of (\ref{eq-Lemma-Dz**-10})} \label{appendix-upper-D**}

Observe that
$d(x) \ge 4$ for each $x$ with $x v \in E(\bar{G})$, and thus
$$
\sum_{x v \in E(\bar{G})} ( - f(d(v), d(x)) + f (d(v) + 2, d(x))) \le (d(v) - 1) ( - f(d(v), 4) + f (d(v) + 2, 4)), \qquad
$$
by Proposition~\ref{pro-10}. So we have the following (weaker) upper bound on $ABC (G') - ABC (G)$:
\begin{eqnarray} \label{eq-Lemma-Dz**-10-1}
%
ABC (G') - ABC (G) &\le& (d(v) - 1) ( - f(d(v), 4) + f (d(v) + 2, 4)) - f (d(v),z+1)  \nonumber \\
&& + 2 f \left( d(v) + 2, \frac{z+1}{2} \right)  - z \cdot f (z+1,4) + (z-1) f \left( \frac{z+1}{2}, 4 \right)  \nonumber \\
&&  - f(4,3) + f (d(v)+2,5) \nonumber \\
&=& - (d(v) - 1) f(d(v), 4) + (d(v) + 1 ) f (d(v) + 2, 4)   \nonumber \\
&& + 2 \left( -  f (d(v) + 2, 4) +  f \left( d(v) + 2, \frac{z+1}{2} \right)\right) - f (d(v),z+1)  \nonumber \\
&& + f (d(v)+2,5) - z \, f (z+1,4) + (z-1) f \left( \frac{z+1}{2}, 4 \right)- f(4,3).
\end{eqnarray}

As to the terms related to $d(v)$ in the right-hand side of (\ref{eq-Lemma-Dz**-10-1}), we have the following observations:
\begin{itemize}
\item

It is not hard to verify that $- (d(v) - 1) f(d(v), 4) + (d(v) + 1 ) f (d(v) + 2, 4)$ decreases in $d(v)$;

\item

$-  f (d(v) + 2, 4) +  f \left( d(v) + 2, \frac{z+1}{2} \right)$ decreases in $d(v)$, by Proposition~\ref{pro-20};

\item

$- f (d(v),z+1) + f (d(v)+2,5)$ increases in  $d(v)$, by Proposition~\ref{pro-10}, i.e.,
\begin{eqnarray*}
- f (d(v),z+1) + f (d(v)+2,5)
&\le& \lim_{d(v) \rightarrow + \infty}(- f (d(v),z+1) + f (d(v)+2,5)) \\
&=& -\sqrt{\frac{1}{z+1}} + \sqrt{\frac{1}{5}}.
\end{eqnarray*}
\end{itemize}
Therefore, we obtain one more weaker upper bound on $ABC (G') - ABC (G)$:
\begin{eqnarray} \label{eq-Lemma-Dz**-10-2}
%
ABC (G') - ABC (G) &\le&  - (d(v) - 1) f(d(v), 4) + (d(v) + 1 ) f (d(v) + 2, 4) \nonumber \\
&&+ 2 \left( -  f (d(v) + 2, 4) +  f \left( d(v) + 2, \frac{z+1}{2} \right) \right) \nonumber \\
&&-\sqrt{\frac{1}{z+1}} + \sqrt{\frac{1}{5}} - z \, f (z+1,4) + (z-1) f \left( \frac{z+1}{2}, 4 \right)- f(4,3), \qquad
\end{eqnarray}
which decreases in $d(v) \ge z + 1$.
After setting $d(v) = z + 1$ in the right-hand side of (\ref{eq-Lemma-Dz**-10-2}),  we have
\begin{eqnarray} \label{eq-Lemma-Dz**-10-3}
%
ABC (G') - ABC (G) &\le& - z \, f(z +1, 4) + (z + 2 ) f (z + 3, 4) + 2 \left( -  f (z + 3, 4) +  f \left( z + 3, \frac{z+1}{2} \right)\right)  \nonumber \\
&&-\sqrt{\frac{1}{z+1}} + \sqrt{\frac{1}{5}}  - z \, f (z+1,4) + (z-1) f \left( \frac{z+1}{2}, 4 \right)- f(4,3),   \nonumber
\end{eqnarray}
in which the right-hand side is always negative for $117 \le z \le 131$.
As to $75 \le z \le 115$ ($z$ is odd), set $d(v) = 264$ in the right-hand side of (\ref{eq-Lemma-Dz**-10-2}), the right-hand side would be again always negative for $75 \le z \le 115$.
For the remaining parts when $75 \le z \le 115$ and $z + 1 \le d(v) \le 263$, we resort to the right-hand side of (\ref{eq-Lemma-Dz**-10-1}), which is always negative from direct calculations.

\subsection{The existence of $D_{z}^{**}$-branches--the negativity of right-hand side of (\ref{eq-no-Dz_**-v2})} \label{appendix-nonexist-B3**}

Similarly as in Appendix \ref{appendix-nonexist-B2-2}, we get a new upper bound on $ABC (G') - ABC (G)$:
\begin{eqnarray}  \label{eq-no-Dz_**-v2-1}
%
ABC (G') - ABC (G) &\le& - (d(v) - x - 1) f (d(v), 4) + (d(v) - x + 1) f(d(v) + 2, 4) \nonumber \\
&& + ( -  f(d(v) + 2, 4) + f (d(v), z + 1) )  + ( - f (d(v) + 2, 4) + f (d(v) + 2, k) )  \nonumber \\
&&  + ( - f(d(v), k + 1)+f(d(v)+2,5) ) \nonumber \\
&&+ (x+1) (- f(d(v), z + 1) + f (d(v) + 2, z))  -x \, z \, f(z+1,4)  \nonumber \\
&&  - k \, f(k+1,4)  - f(4,3) + (x+1)(z-1) f(z,4) + (k-1) f(k,4). 
\end{eqnarray}
Further we obtain
\begin{eqnarray}  \label{eq-no-Dz_**-v2-2}
%
ABC (G') - ABC (G) &\le& - (d(v) - x - 1) f (d(v), 4) + (d(v) - x + 1) f(d(v) + 2, 4) \nonumber \\
&&+ ( -  f(d(v) + 2, 4) + f (d(v), z + 1) )  + ( - f (d(v) + 2, 4) + f (d(v) + 2, k)   ) \nonumber \\
&&+  \left( -\sqrt{\frac{1}{k+1}} + \sqrt{\frac{1}{5}} \right)+ (x+1) \left( -\sqrt{\frac{1}{z+1}} + \sqrt{\frac{1}{z}} \right) -x \, z \, f(z+1,4) \nonumber \\
&&   - k \, f(k+1,4)  - f(4,3) + (x+1)(z-1) f(z,4) + (k-1) f(k,4),
\end{eqnarray}
decreasing in $d(v)$.
After substituting $d(v)$ by  $556$ into the right-hand side of (\ref{eq-no-Dz_**-v2-2}),
a straightforward verification shows that the right-hand side of (\ref{eq-no-Dz_**-v2-2}) is
negative for each $z \in [46,57]$, $k \in [\max\{47,z - 1\},z+1]$ and $x = z-1$.
Since the right-hand side of (\ref{eq-no-Dz_**-v2-2}) decreases in $d(v) \ge x + 1$, it follows that it is negative for
any $d(v)  \geq 556$ under the above mentioned constraints.
When $x + 1 \le d(v) < 556$, we resort to the right-hand side of (\ref{eq-no-Dz_**-v2-1}),
which can be shown that is negative by direct calculations.

\subsection{All $B_4$-branches are adjacent to the root vertex--the negativity of right-hand side of (\ref{eq:B4-2-ori})} \label{appendix-B4-to-root}

Here we use again the techniques in Appendix \ref{appendix-nonexist-B2-2}. Then an upper bound on the change of the ABC index comes:
\begin{eqnarray}\label{eq:B4-2-ori-1}
ABC (G') - ABC (G) &\le& - (d(v) + x - 5) f (d(v), 16) + (d(v) + 2x - 5) f (d(v) + x, 16)  \nonumber \\
&& + (4-x) (-f(d(v),5) + f (d(v) + x, 5)) \nonumber \\
&&+ x (-f(d(v)+x,16) + f(d(v) + x, 5)) \nonumber \\
&& + (z - x) (-f(z+1,4) + f (z + 1 - x,4)) \nonumber \\
&& + ( - f (d(v),z+1) + f(d(v) + x, z -x + 1)) - x \, f(z + 1, 5).
\end{eqnarray}
Aiming to obtain a (weaker) upper bound which decreases in $d(v)$, we get:
\begin{eqnarray}\label{eq:B4-2-ori-2}
ABC (G') - ABC (G) &\le& - (d(v) + x - 5) f (d(v), 16) + (d(v) + 2x - 5) f (d(v) + x, 16) \nonumber \\ 
 &&+ x \left( - \sqrt{\frac{1}{16}} + \sqrt{\frac{1}{5}} \right) + (z - x) (-f(z+1,4) + f (z + 1 - x,4))  \nonumber \\
&& + \left( - \sqrt{\frac{1}{z+1}} + \sqrt{\frac{1}{z-x+1}} \right) - x \, f(z + 1, 5).
\end{eqnarray}
By direct calculations, we deduce that the right-hand side of (\ref{eq:B4-2-ori-2}) is negative for $d(v) \ge 300$ (actually set $d(v) = 300$ is enough due to the monotonicity on $d(v)$), $1 \le x \le 4$ and $15 \le z \le 215$. For the remaining part $z + 1 \le d(v) < 300$, $1 \le x \le 4$ and $15 \le z \le 215$, the right-hand side of (\ref{eq:B4-2-ori-1}) could help us, which is negative from direct calculations again.

\subsection{Minimal-ABC trees} \label{appendix-figures}

In this section, we present the structures of the minimal-ABC trees.
The exact parameters are given only for trees of smaller orders.
The parameters for an arbitrary tree can be determined as it was elaborated in Section~\ref{section-compuation}.
For example, it can be obtained that:
\begin{itemize}
\item
When $n = 5047$, then $z = 50$, $n_z=7$, $n_{z+1}=4$, $n_3 = 164$, and $n_4 = 1$;

\item
When $n = 6956$, then $z = 49$, $n_z=1$, $n_{z+1}=15$, $n_3 = 191$, and $n_4 = 1$;

\item
When $n = 16443$, then $z = 49$, $n_z=0$, $n_{z+1}=41$, and $n_3 = 293$;

\item
When $n = 1014814$, then $z = 51$, $n_z=2594$, $n_{z+1}=236$, and $n_3 = 3$;

\item
When $n = 1142741$, then $z = 51$, $n_z=3035$, $n_{z+1}=154$;

\item
When $n = 1257073$, then $z = 51$, $n_z=259$, $n_{z+1}=3190$;

\item
When $n = 13290000000000000$, then $z = 51$, $n_z=178$, $n_{z+1}=36410958903935$.

\end{itemize}

\begin{figure}[!ht]
\begin{center}
\includegraphics[scale=0.739]{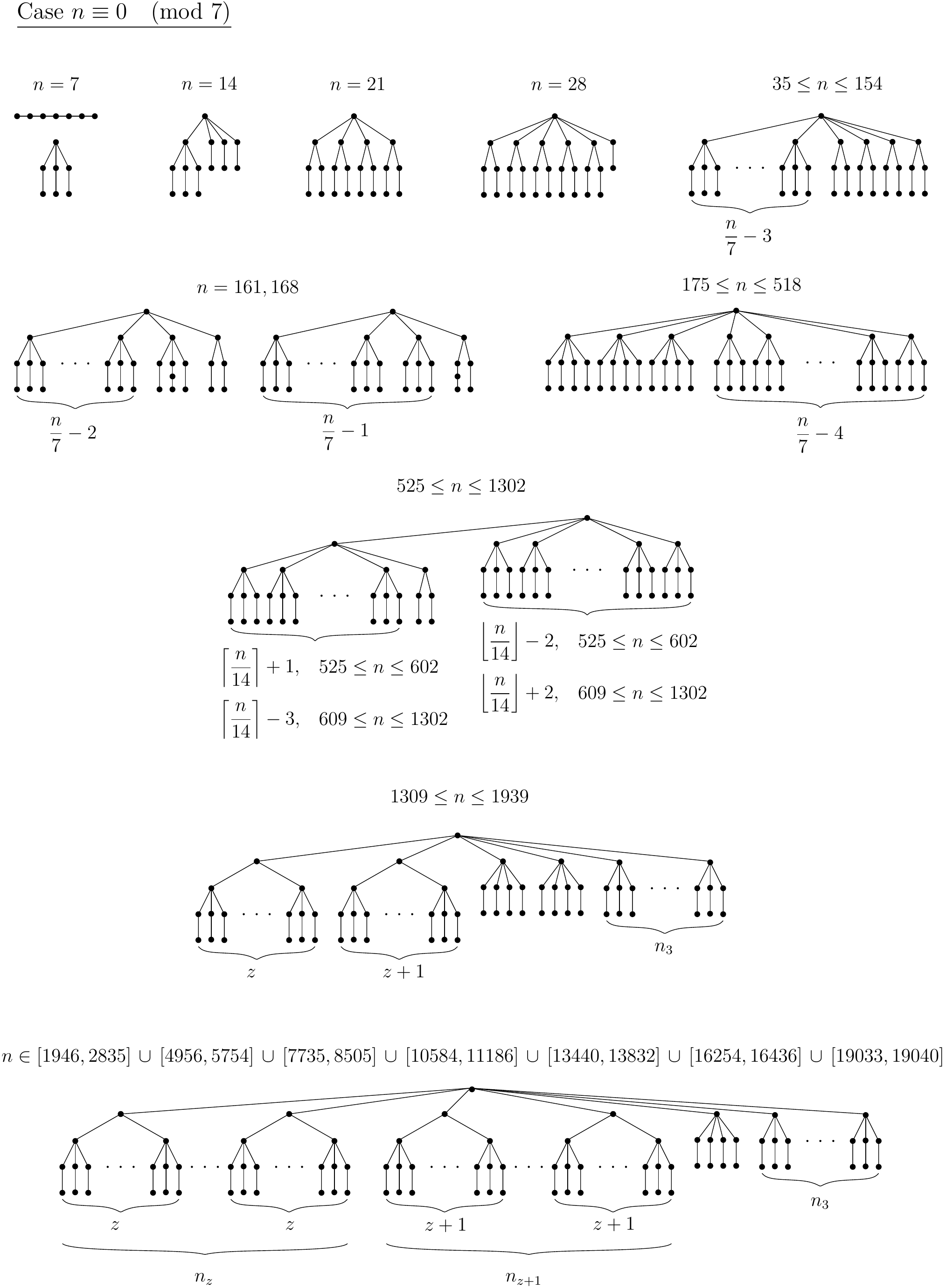}
\caption{Minimal-ABC trees with $n$ vertices, where $n \equiv 0 \pmod{7}$.}
\label{fig_minABC-mod_0}
\end{center}
\end{figure}
\begin{figure}[!ht]
\begin{center}
\includegraphics[scale=0.75]{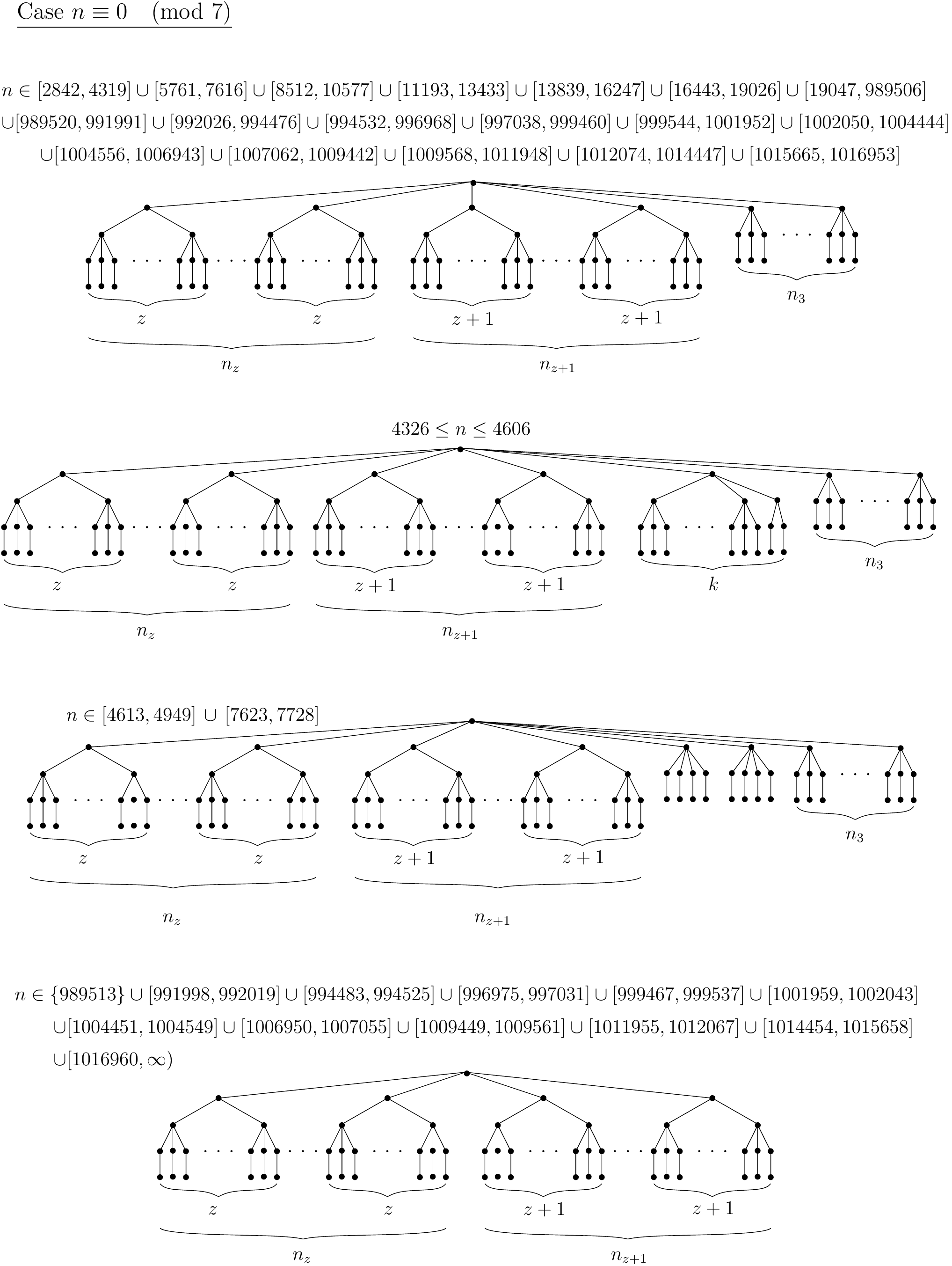}
\caption{Minimal-ABC trees with $n$ vertices, where $n \equiv 0 \pmod{7}$.
}
\label{fig_minABC-mod_0}
\end{center}
\end{figure}
\begin{figure}[!ht]
\begin{center}
\includegraphics[scale=0.75]{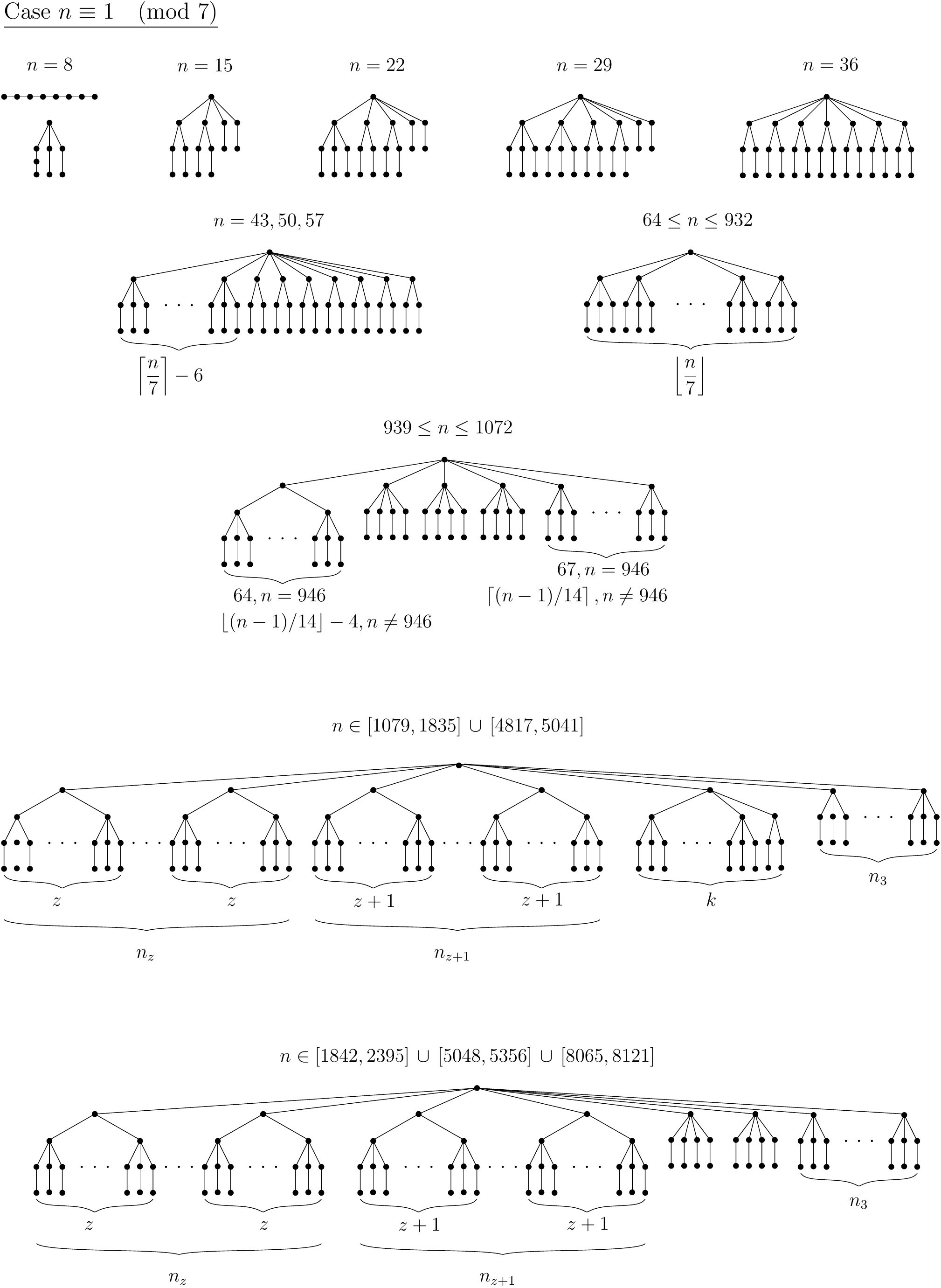}
\caption{Minimal-ABC trees with $n$ vertices, where $n \equiv 1 \pmod{7}$.}
\label{fig_minABC-mod_1}
\end{center}
\end{figure}
\begin{figure}[!ht]
\begin{center}
\includegraphics[scale=0.75]{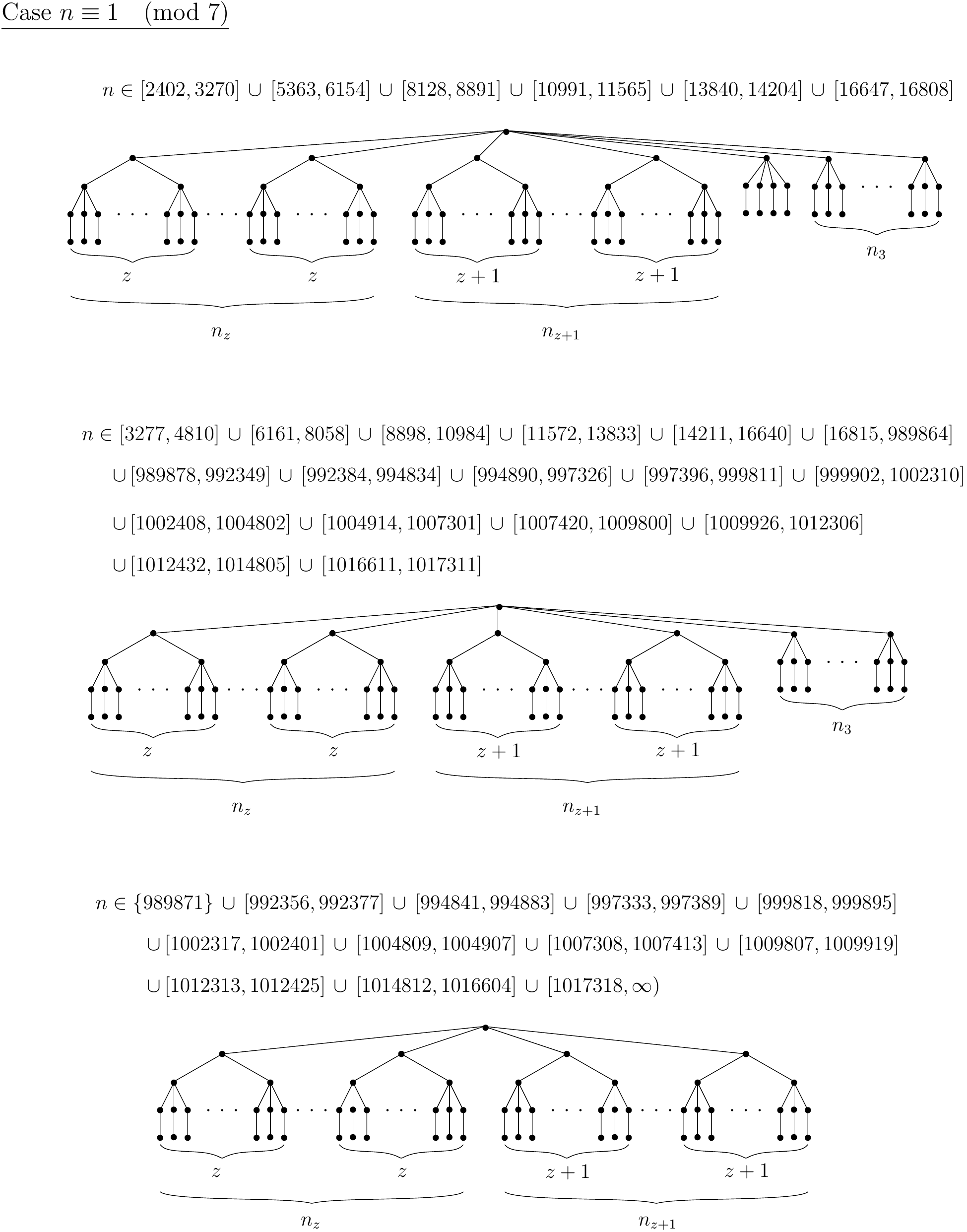}
\caption{Minimal-ABC trees with $n$ vertices, where $n \equiv 1 \pmod{7}$.}
\label{fig_minABC-mod_1}
\end{center}
\end{figure}
\begin{figure}[!ht]
\begin{center}
\includegraphics[scale=0.75]{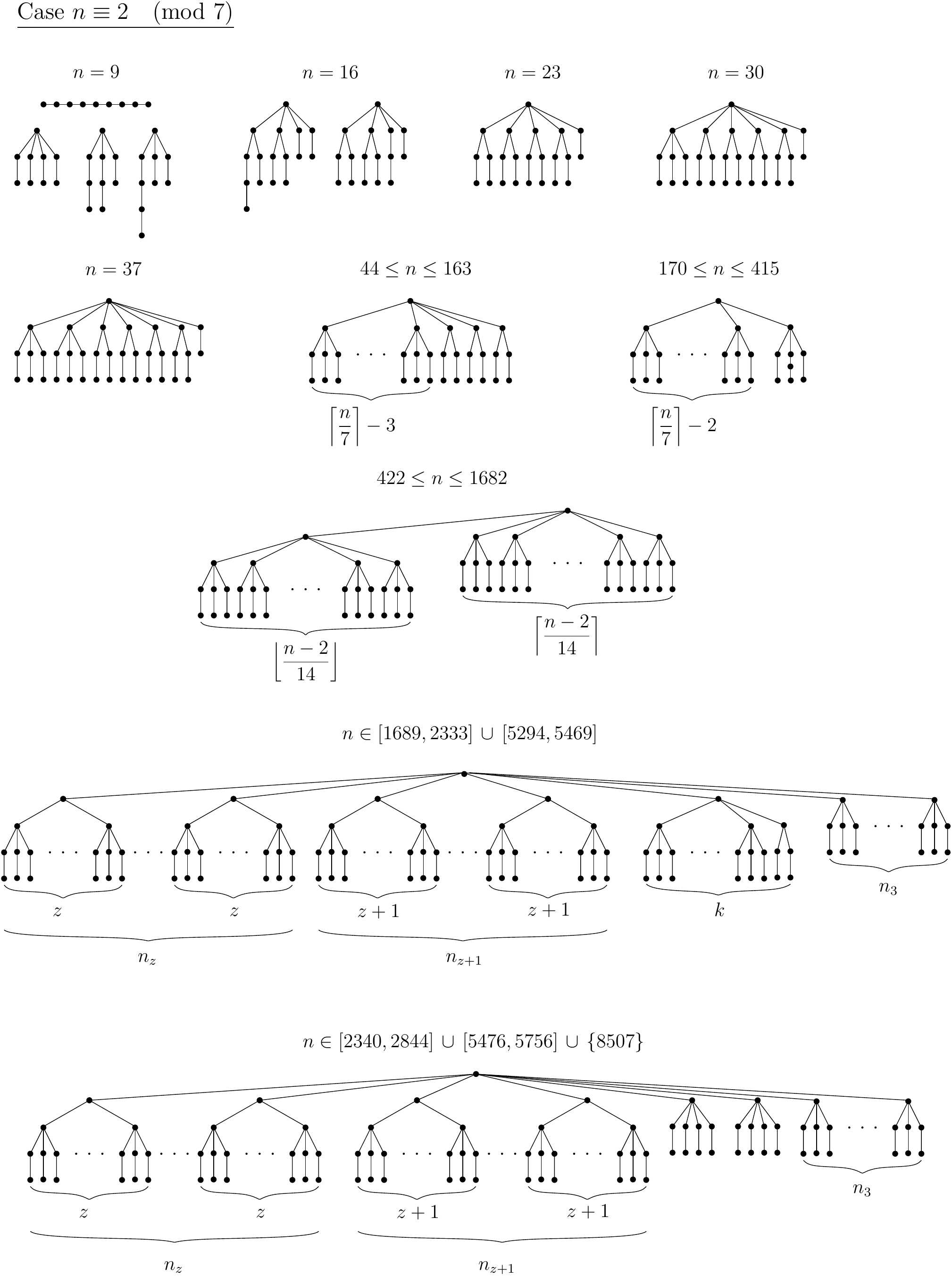}
\caption{Minimal-ABC trees with $n$ vertices, where $n \equiv 2 \pmod{7}$.}
\label{fig_minABC-mod_2}
\end{center}
\end{figure}
\begin{figure}[!ht]
\begin{center}
\includegraphics[scale=0.75]{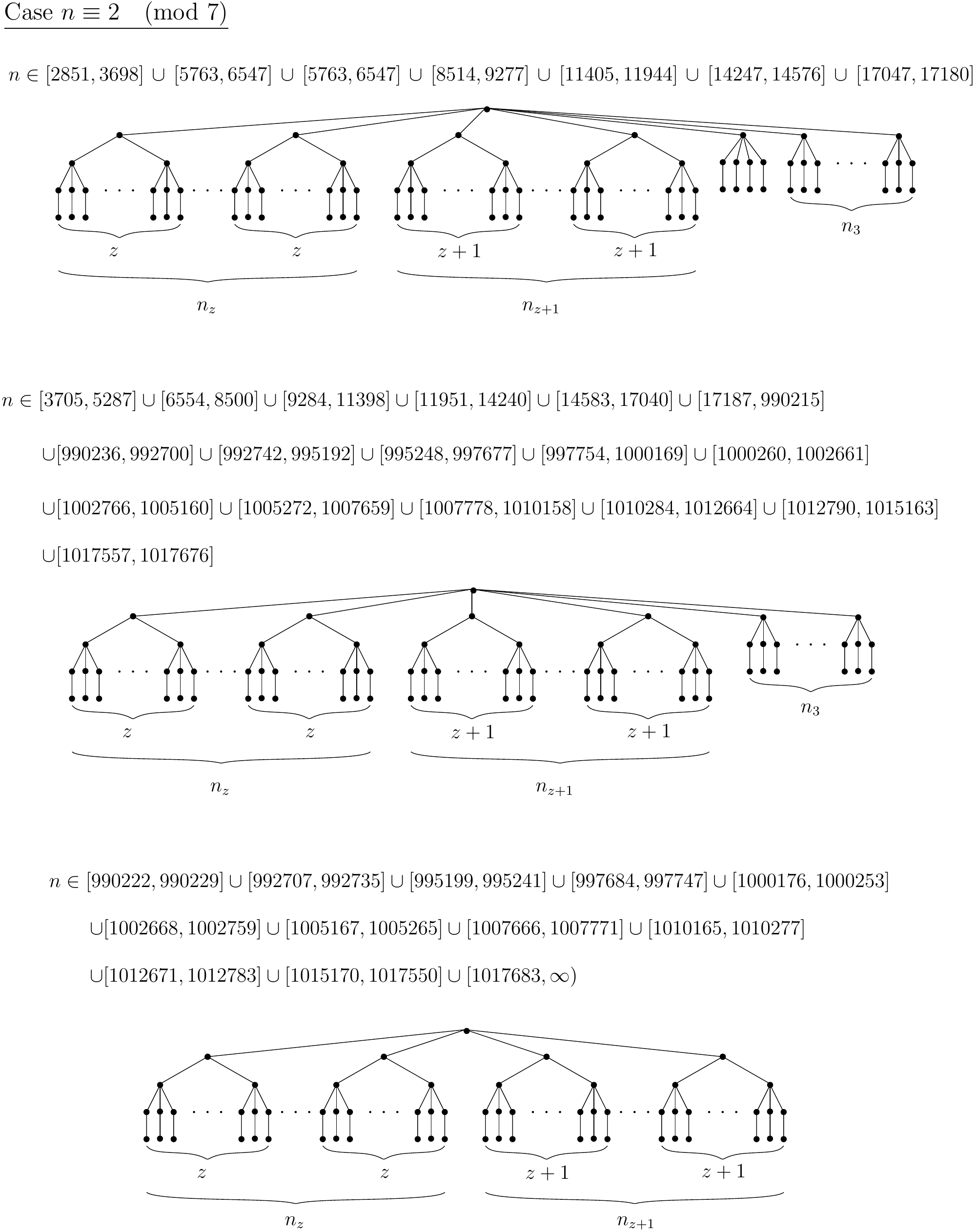}
\caption{Minimal-ABC trees with $n$ vertices, where $n \equiv 2 \pmod{7}$.}
\label{fig_minABC-mod_2}
\end{center}
\end{figure}
\begin{figure}[!ht]
\begin{center}
\includegraphics[scale=0.75]{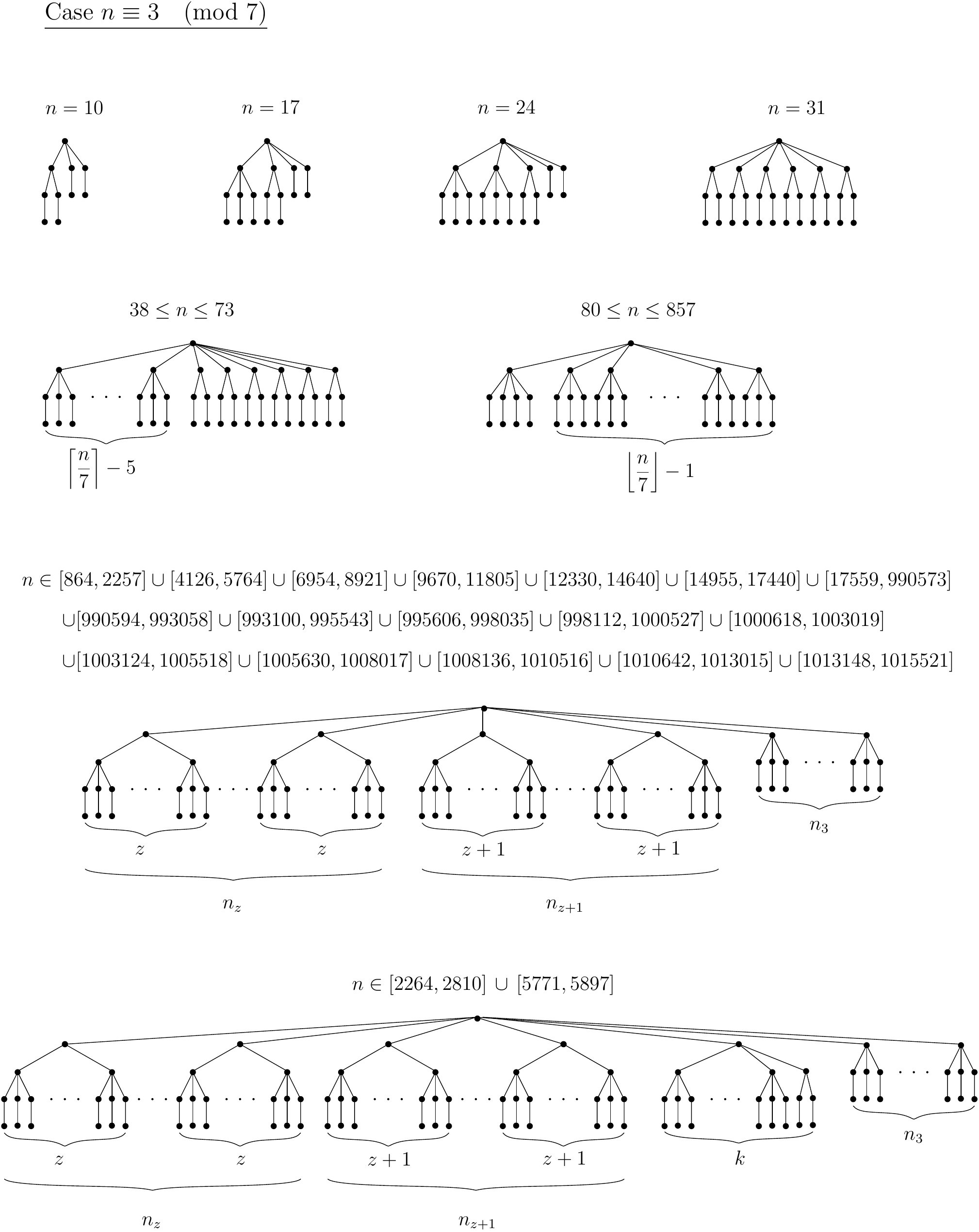}
\caption{Minimal-ABC trees with $n$ vertices, where $n \equiv 3 \pmod{7}$.}
\label{fig_minABC-mod_3}
\end{center}
\end{figure}
\begin{figure}[!ht]
\begin{center}
\includegraphics[scale=0.75]{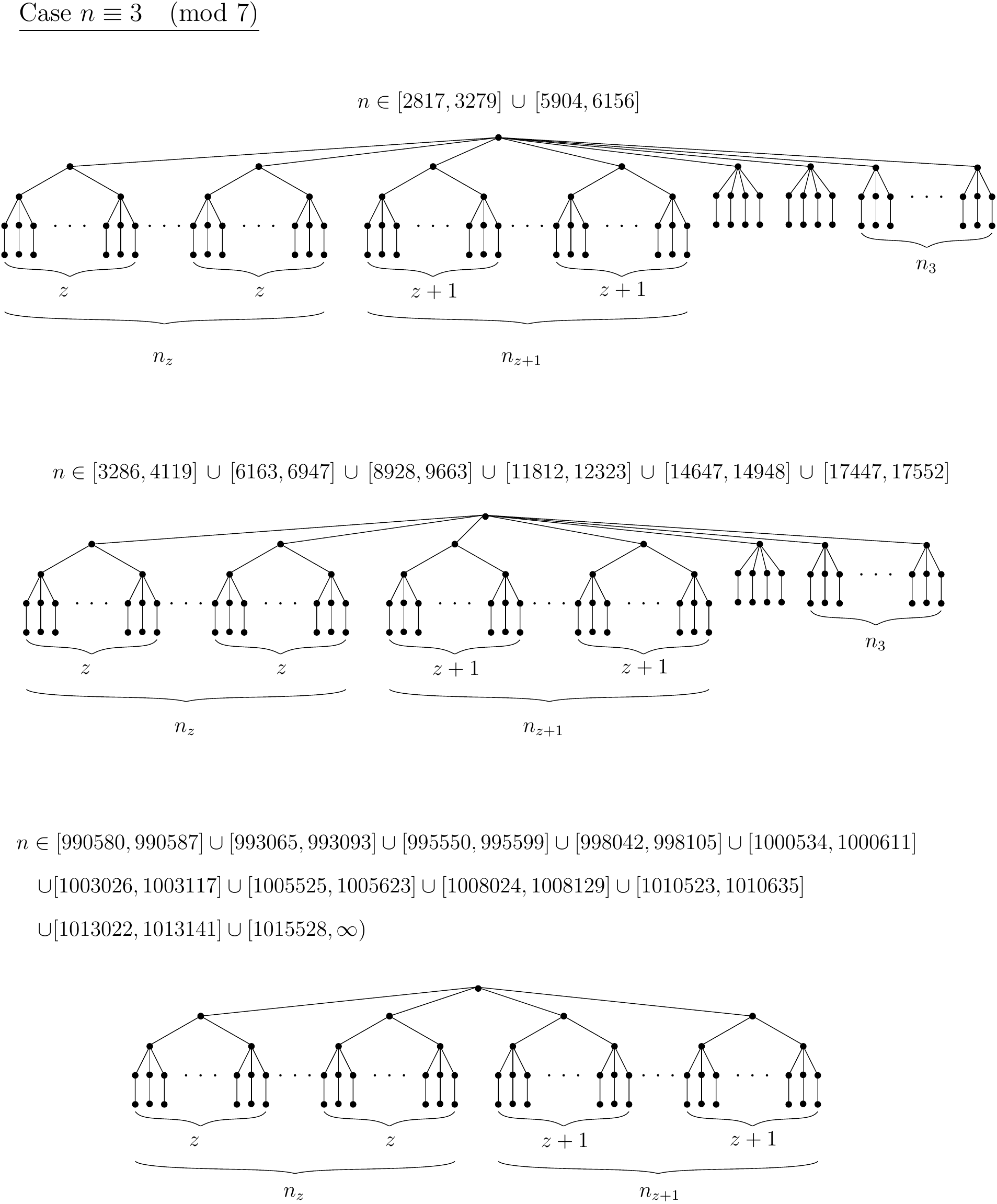}
\caption{Minimal-ABC trees with $n$ vertices, where $n \equiv 3 \pmod{7}$.}
\label{fig_minABC-mod_3}
\end{center}
\end{figure}
\begin{figure}[!ht]
\begin{center}
\includegraphics[scale=0.75]{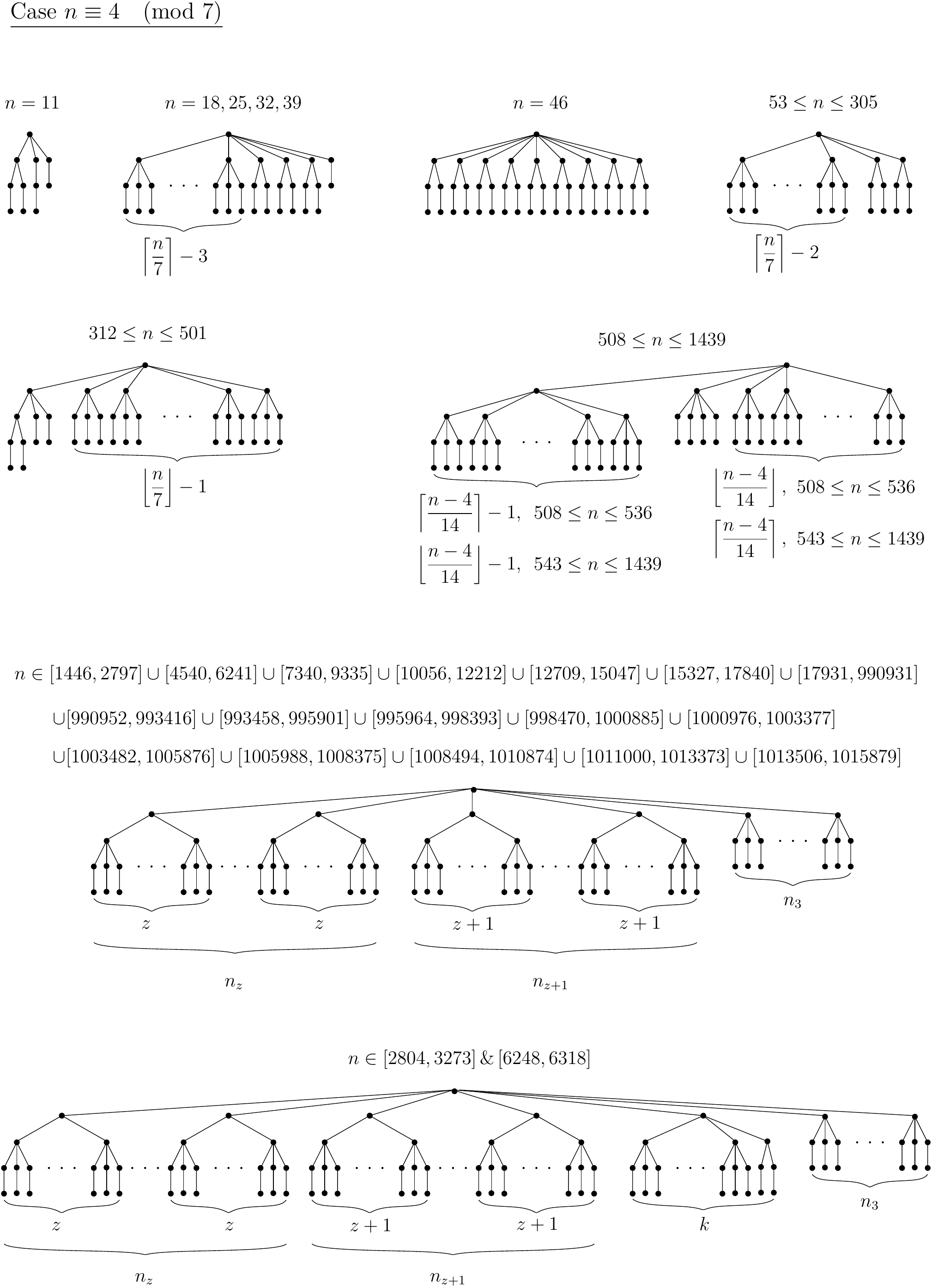}
\caption{Minimal-ABC trees with $n$ vertices, where $n \equiv 4 \pmod{7}$.}
\label{fig_minABC-mod_4}
\end{center}
\end{figure}
\begin{figure}[!ht]
\begin{center}
\includegraphics[scale=0.75]{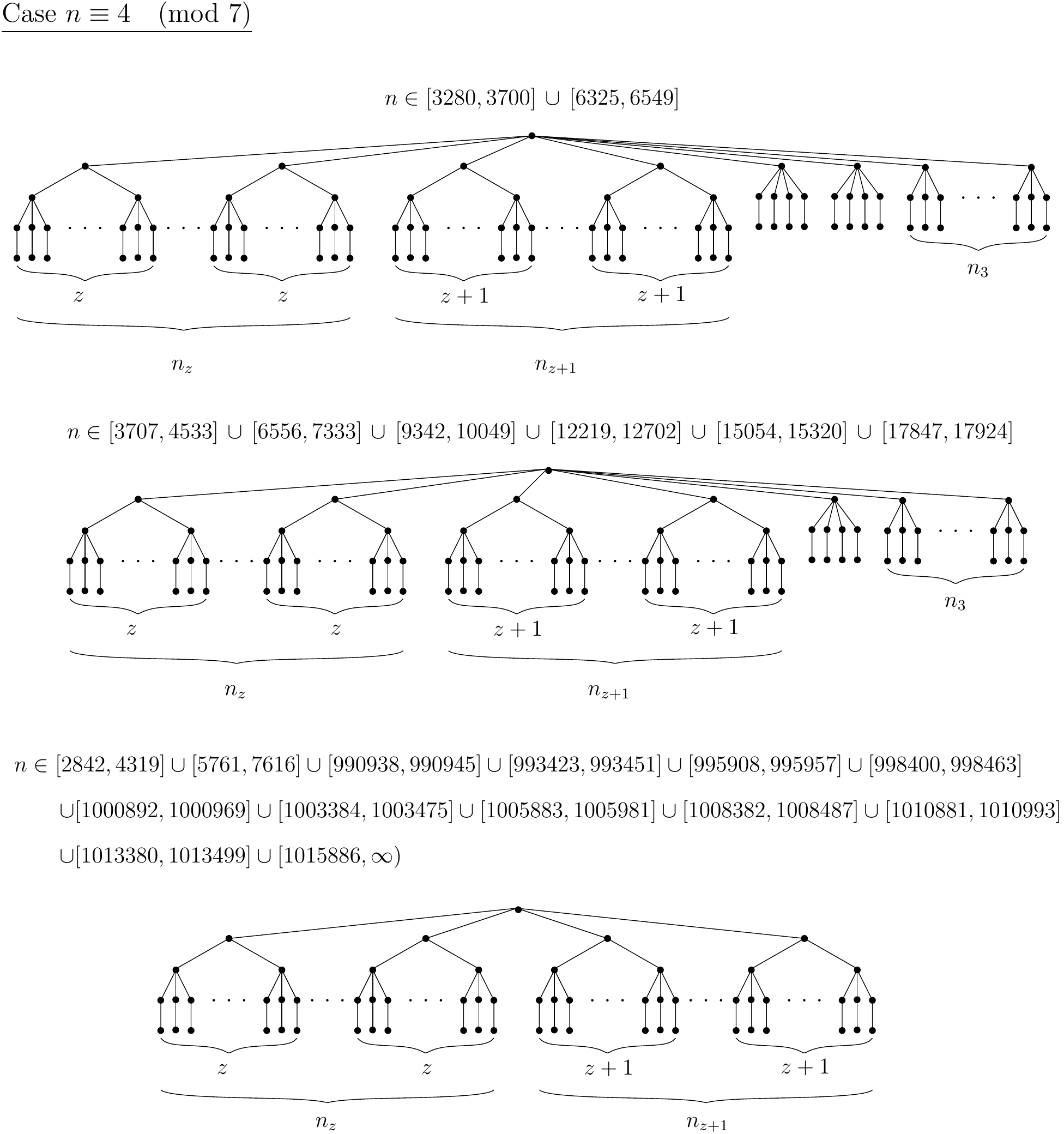}
\caption{Minimal-ABC trees with $n$ vertices, where $n \equiv 4 \pmod{7}$.}
\label{fig_minABC-mod_4}
\end{center}
\end{figure}
\begin{figure}[!ht]
\begin{center}
\includegraphics[scale=0.75]{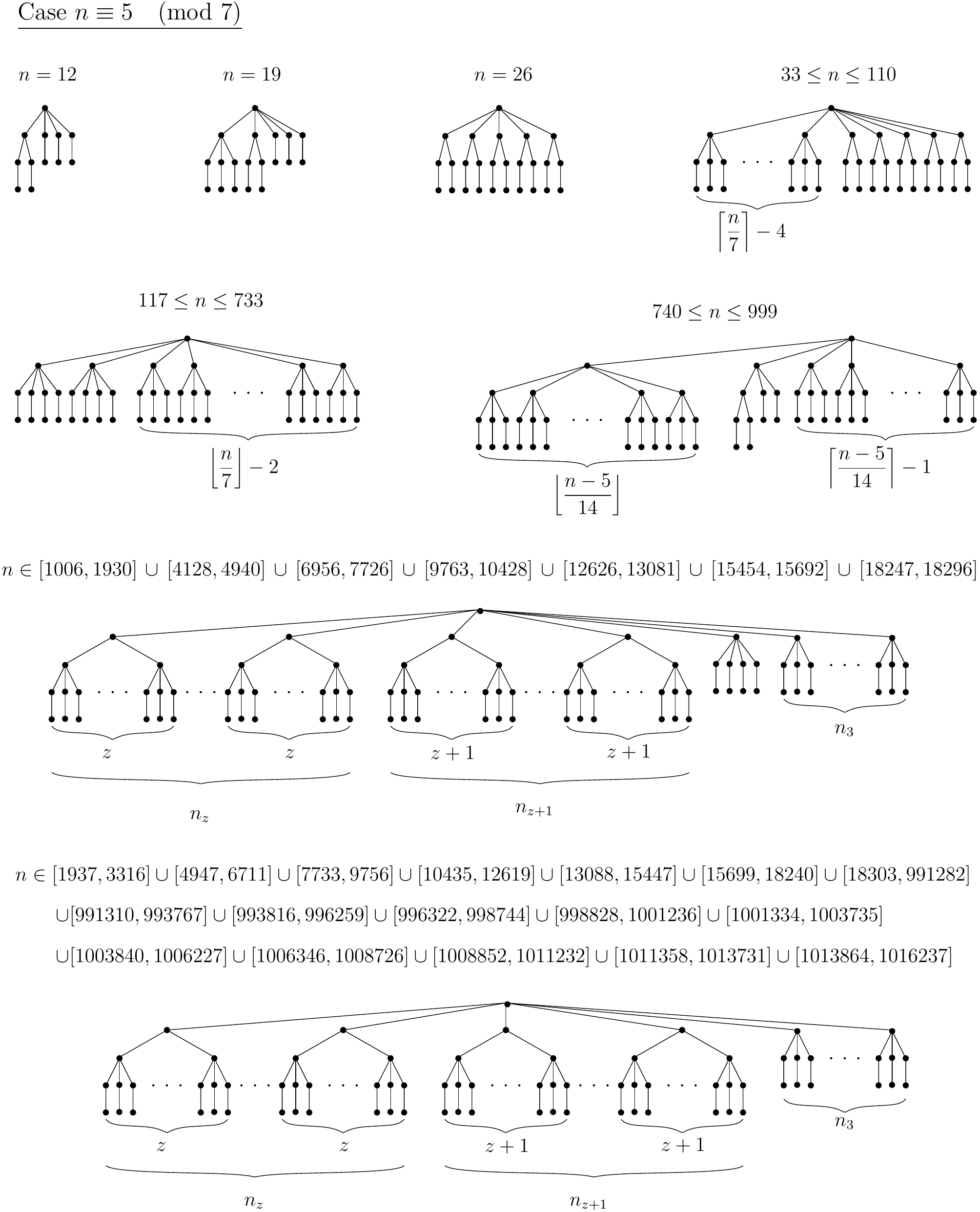}
\caption{Minimal-ABC trees with $n$ vertices, where $n \equiv 5 \pmod{7}$.}
\label{fig_minABC-mod_5}
\end{center}
\end{figure}
\begin{figure}[!ht]
\begin{center}
\includegraphics[scale=0.75]{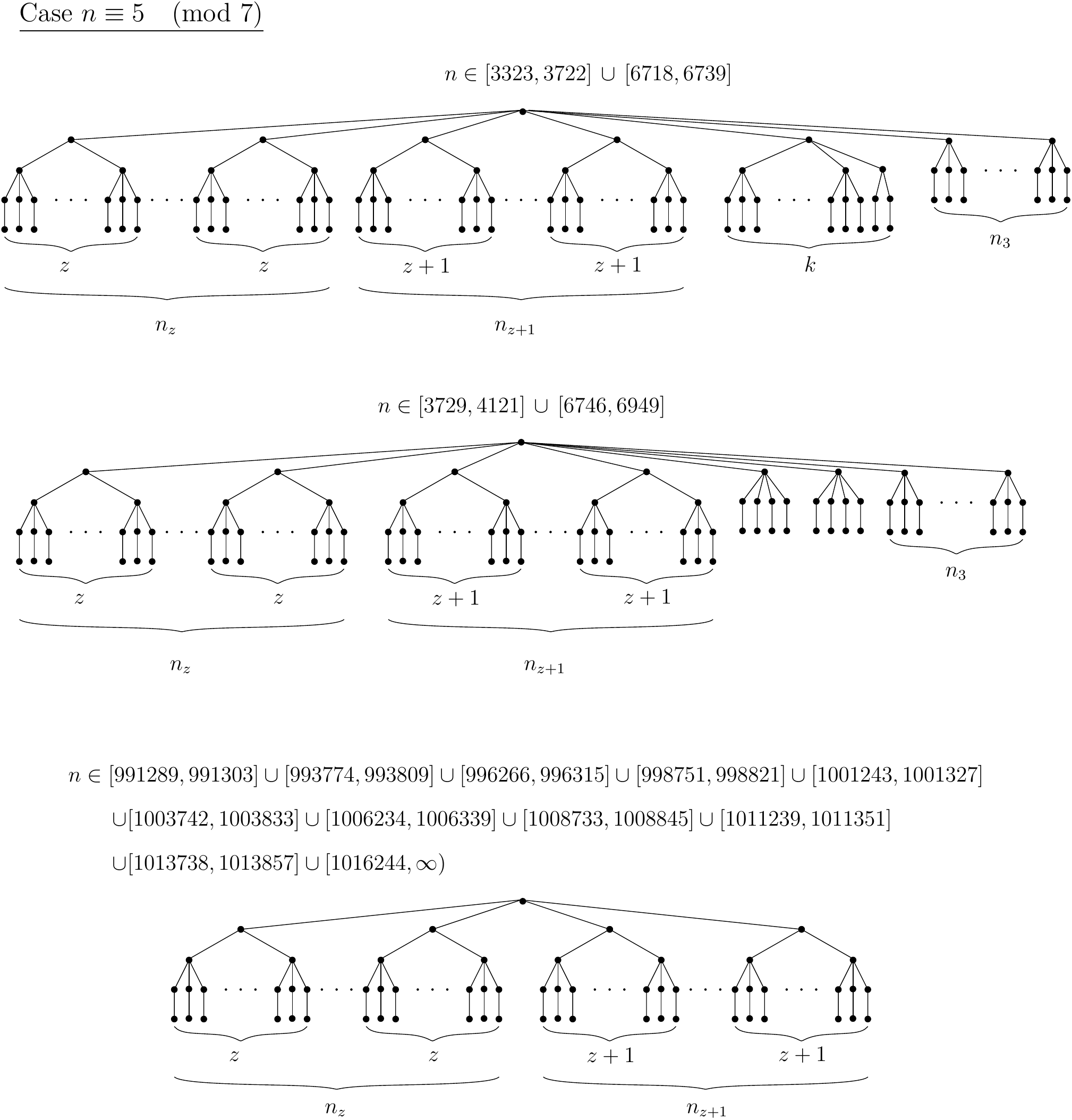}
\caption{Minimal-ABC trees with $n$ vertices, where $n \equiv 5 \pmod{7}$.}
\label{fig_minABC-mod_5}
\end{center}
\end{figure}
\begin{figure}[!ht]
\begin{center}
\includegraphics[scale=0.75]{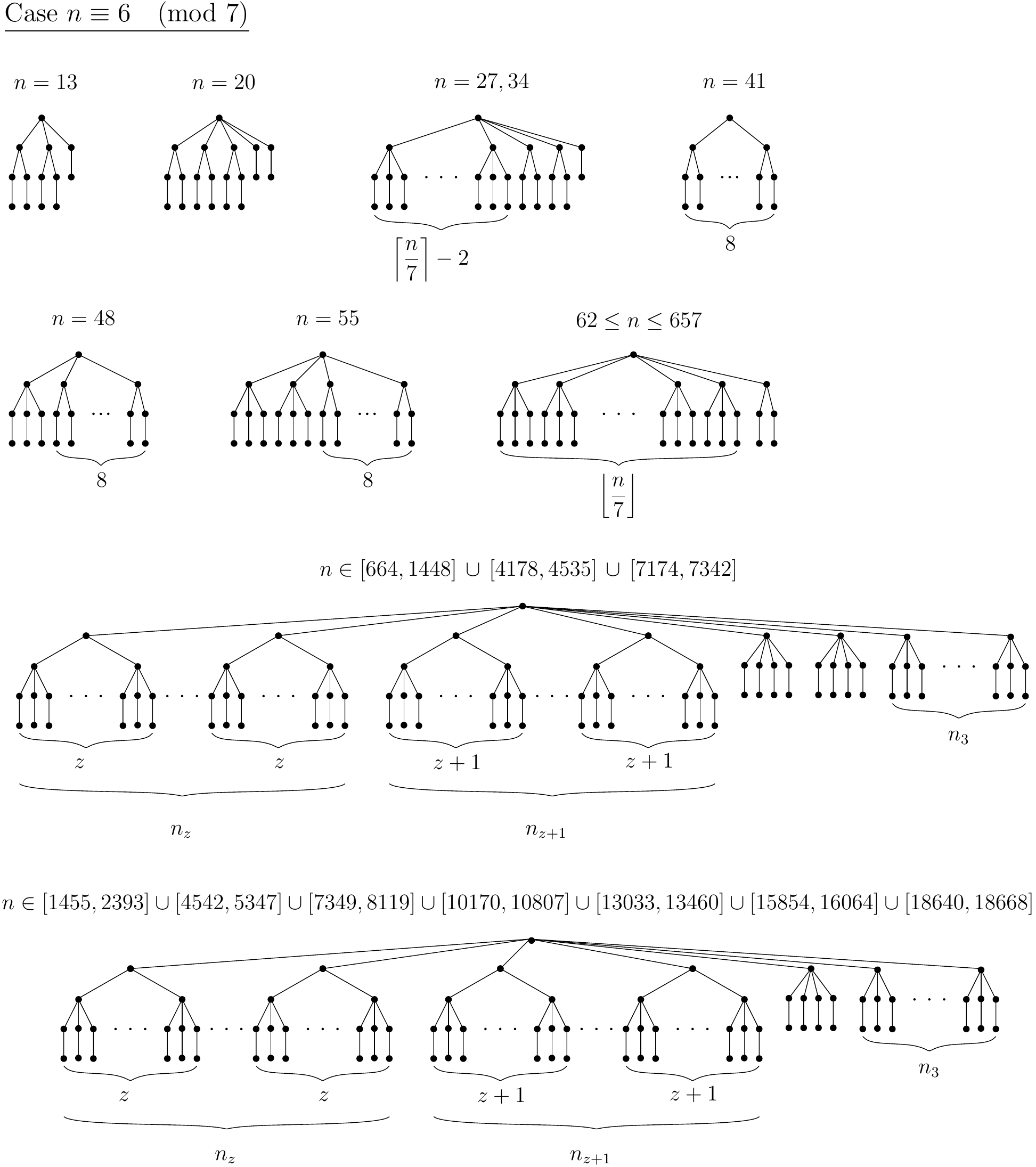}
\caption{Minimal-ABC trees with $n$ vertices, where $n \equiv 6 \pmod{7}$.}
\label{fig_minABC-mod_6}
\end{center}
\end{figure}
\begin{figure}[!ht]
\begin{center}
\includegraphics[scale=0.75]{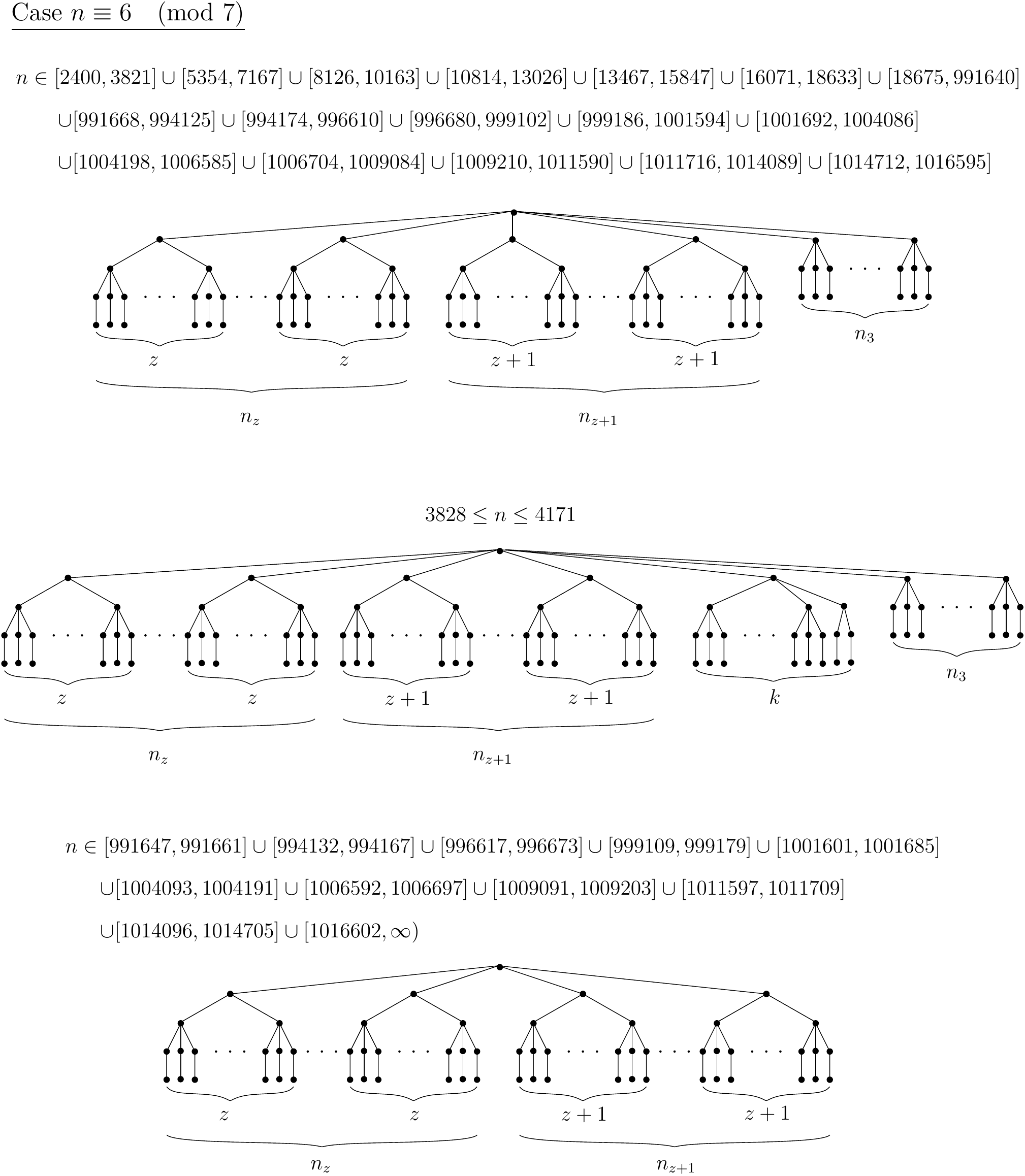}
\caption{Minimal-ABC trees with $n$ vertices, where $n \equiv 6 \pmod{7}$.}
\label{fig_minABC-mod_6}
\end{center}
\end{figure}

\end{document}